\documentclass[a4paper,10pt]{article}
\usepackage{mathtools}
\usepackage{geometry}
\usepackage[pdftex]{graphicx}
\usepackage{amsfonts,amssymb,amsmath,amsthm}
\usepackage{epstopdf}
\usepackage{csquotes}
\usepackage{color}
\usepackage{authblk}
\usepackage{tikz}
\usetikzlibrary{matrix}
\usepackage{tabularx}

\usepackage[percent]{overpic}
\usepackage{url}
\numberwithin{equation}{section}
\newtheorem{theorem}{Theorem}
\newtheorem{remark}{Remark}

\newtheorem{proposition}[theorem]{Proposition}

\numberwithin{theorem}{section}
\usepackage{algorithm}
\usepackage{algorithmic}
\usepackage{tcolorbox}
\usepackage{enumerate}
\usepackage[title]{appendix}
\usepackage{cite}
\usepackage{pdfpages}
\usepackage{authblk}
\usepackage{titling}
\newcommand{\pderiv}[2]{\frac{\partial #1}{\partial #2}}

\begin{document}

\title{A contour method for time-fractional PDEs\\and an application to fractional viscoelastic beam equations}

\author{Matthew J. Colbrook\thanks{Corresponding author: m.colbrook@damtp.cam.ac.uk},$\quad$ Lorna J. Ayton}
\affil{Department of Applied Mathematics and Theoretical Physics\\University of Cambridge}

\date{\today}

\maketitle

\begin{abstract}
We develop a rapid and accurate contour method for the solution of time-fractional PDEs. The method inverts the Laplace transform via an optimised stable quadrature rule, suitable for infinite-dimensional operators, whose error decreases like $\exp(-cN/\log(N))$ for $N$ quadrature points. The method is parallisable, avoids having to resolve singularities of the solution as $t\downarrow 0$, and avoids the large memory consumption that can be a challenge for time-stepping methods applied to time-fractional PDEs. The ODEs resulting from quadrature are solved using adaptive sparse spectral methods that converge exponentially with optimal linear complexity. These solutions of ODEs are reused for different times. We provide a complete analysis of our approach for fractional beam equations used to model small-amplitude vibration of viscoelastic materials with a fractional Kelvin--Voigt stress-strain relationship. We calculate the system's energy evolution over time and the surface deformation in cases of both constant and non-constant viscoelastic parameters. An infinite-dimensional ``solve-then-discretise'' approach considerably simplifies the analysis, which studies the generalisation of the numerical range of a quasi-linearisation of a suitable operator pencil. This allows us to build an efficient algorithm with explicit error control. The approach can be readily adapted to other time-fractional PDEs and is {not} constrained to fractional parameters in the range $0<\nu<1$.
\end{abstract}


\section{Introduction}
\label{sec:intro}

Fractional derivatives and fractional differential equations (FDEs) are playing an increasingly important and widely used role in the modelling of biological and physical processes \cite{rossikhin2010application,hilfer2000applications,mainardi2010fractional,sabatier2007advances,magin2010fractional,oldham2010fractional,scalas2000fractional,sheng2011fractional,metzler2000random,barkai2000continuous,klages2008anomalous,colbrook2017scaling,bronstein2009transient,holm2011causal}. One of the most useful modelling properties of fractional derivatives is their non-local nature. However, this feature also makes the design of accurate, fast and robust numerical methods for the solution of FDEs decidedly non-trivial. The most common numerical methods used for FDEs (see \cite{baleanu2012fractional,li2015numerical} for overviews) have historically been convolution quadrature \cite{lubich1986discretized,lubich1988convolution,fritz2021time,lubich1983stability}, finite differences \cite{cui2009compact,meerschaert2006finite,yuste2005explicit,lin2007finite,sun2006fully,jin2013error,meerschaert2006finite2,tadjeran2007second}, and finite elements \cite{deng2009finite,ford2011finite,liu2004numerical}. A major challenge with these \textit{local} approaches is obtaining high accuracy, particularly over large time/spatial scales. Due to the \textit{global} (nonlocal) nature of fractional derivatives, spectral methods have recently been proposed as an attractive alternative \cite{hale2018fast,li2009space,zayernouri2014fractional,mao2018spectral,zayernouri2013fractional,chen2016generalized,li2010existence}. For example, the work of Zayernouri \& Karniadakis \cite{zayernouri2014fractional,zayernouri2013fractional} developed an exponentially accurate collocation scheme, leading to dense matrices and $\mathcal{O}(N^3)$ complexity. An alternative sparse spectral method for rational-order derivatives with linear complexity was developed by Hale \& Olver \cite{hale2018fast}.

In this paper, we focus on time-fractional differential equations (TFDEs). A challenge in this scenario is that the solution at time $t>0$ depends on the solution at all previous times, which typically makes the approximation of the solution computationally demanding. Rather than taking a spectral method approach in time, we develop a contour method based on the Laplace transform and quadrature rules for the inverse Laplace transform. There is a large literature on contour methods for inverting the Laplace transform (see \S \ref{cont_descrip} and the summary in \cite{trefethen2014exponentially}), but relatively little work on the use of such methods for TFDEs. An exception is the time-fractional diffusion equation, studied in \cite{weideman2006optimizing,pang2016fast}, and the earlier work in \cite{mclean2010maximum,mclean2010numerical}. The inverse Laplace transform via contour methods has, however, been used to great effect in evaluating (scalar) Mittag--Leffler functions \cite{garrappa2015numerical,mclean2021numerical}, which appear ubiquitously when studying fractional differential equations. A goal of this paper is to show that contour methods can be very effective at solving more complicated TFDEs.

There are many potential advantages of contour method\footnote{Contour methods can also be interpreted in terms of rational approximations \cite{trefethen2006talbot}. For another rational function approach to TFDEs, see \cite{khristenko2021solving}.} approaches to TFDEs, particularly when the Laplace transform of the solution can be computed efficiently. In our setting, the Laplace transform of the solution can be expressed as the solution of a differential operator pencil equation (see \eqref{laplace_form}). The linear systems that result from applying quadrature rules can be solved in parallel and can then be reused for different times. Another advantage, of particular importance for TFDEs, is that contour methods avoid the large memory consumption of time-stepping methods applied to time-fractional PDEs (which occurs due to the non-local nature of fractional derivatives - see \cite{ford2001numerical,diethelm2006efficient} for some techniques aimed at handling this) and provide high accuracy over large time intervals. Moreover, contour methods also avoid having to resolve the singularities of the solution as $t\downarrow0$ \cite{mao2018spectral,chen2016generalized}. Contour methods are suitable for evaluating the solution on time intervals of the form $[t_0,\Lambda t_0]$ with $t_0>0$ and $\Lambda>1$. This means that by considering geometrically scaled intervals, we can evaluate the solution at small or large times (see, for example, Figure \ref{fig:Case1Energy}).

Contour methods also allow the \textit{infinite-dimensional} problem, posed in a separable Hilbert space $\mathcal{H}$, to be tackled directly, as opposed to a truncation or discretisation.\footnote{This ``solve-then-discretise'' paradigm has recently been applied to spectral computations \cite{colbrook2019compute,horning2020feast,colbrook2019foundations,colbrook2019computation2,ben2015can,johnstone2021bulk}, extensions of classical methods such as the QL and QR algorithms \cite{webb_thesis,colbrook2019infinite} (see also \cite{townsend2015continuous}), Krylov methods \cite{olver2009gmres,gilles2019continuous}, spectral measures \cite{webb2017spectra,colbrook2019computing,colbrook2020computing} and computing semigroups \cite{colbrook2021computing}. Related work includes that of Olver, Townsend and Webb, providing a foundational and practical framework for infinite-dimensional numerical linear algebra and computations with infinite data structures \cite{Olver_Townsend_Proceedings, Olver_SIAM_Rev, Olver_code1, Olver_code2}.} In the case of the present paper, this is achieved through adaptive sparse spectral methods (in the space domain) used to compute the solutions of \eqref{laplace_form} with linear complexity and exponential convergence. This infinite-dimensional approach means that we can prove rigorous convergence results under quite general assumptions\footnote{See Theorem \ref{thrown_in_theorem} for our specific example. The result generalises to when the Laplace transform of the solution can be computed with error control and has known regions of analyticity (with bounds).} and also gain explicit \textit{error control} of our approximation.

To achieve rigorous error bounds using contour methods, it is necessary to bound the generalised spectrum and resolvent of the operator pencil. This boils down to an analysis of the singularities of the Laplace transform of the solution and suitable contours of integration. Whilst this must be done separately for each class/family of equations, it is considerably easier to perform this analysis on the infinite-dimensional operator directly (e.g., Theorem \ref{thm:RES_BOUND}), as opposed to a truncation or discretisation. Thus, another goal of this paper is to advocate an infinite-dimensional approach to contour methods, as opposed to applying contour methods after a reduction to a finite-dimensional system. To be concrete, we focus in this paper on the case of fractional beam equations that arise in the modelling of thin viscoelastic materials. The analysis is performed by a quasi-linearisation of the operator pencil (which is non-linear in the spectral parameter) and by studying the generalisation of the numerical range. Though we have restricted the analysis to the class of TFDEs discussed in \S \ref{sec:phys_setup}, similar approaches may be used for other TFDEs. In our case, the location of the generalised spectrum allows deformation of the Bromwich contour (appearing in the inverse Laplace transform) into the left-half plane. For TFDEs where this does not occur, regularisation of the integrand should be performed \cite{colbrook2021computing}.

\subsection{TFDEs modelling thin viscoelastic beam vibration}

The modelling of the vibration of thin viscoelastic beams is an ideal application for our contour method. Traditional Euler--Bernoulli beam theory, governing the linear vibration of thin elastic beams undergoing small-amplitude oscillation, has been a cornerstone of structural engineering since the 19th century \cite{Timoshenko}. However, recently, modern materials embedded with polymer structures or biomaterials have been shown to exhibit more exotic structural properties than their earlier counterparts, in particular, viscoelasticity. Such materials have both elastic and viscous properties, which can be captured experimentally and fitted to constitutive stress-strain relationships \cite{Pritz}. In such a curve-fitting game, one aims to describe the experimental data as concisely as possible with as few parameters as possible. Permitting time-fractional derivatives in the constitutive relationship is therefore a popular way of accurately describing the experimental data with a small number of parameters \cite{KV1,KV_AIAA}.

As a result, the equations governing the vibration of a viscoelastic beam, due to either an initial deformation or ongoing forcing, contain time-fractional derivatives, and the accurate prediction of the beam vibration is reliant on the accurate solution to these governing equations. Previous numerical methods tailored to fractional beam equations include Galerkin/separation of variables approaches \cite{Galerkin,HybridGalerkin,Galerkin2}, finite element methods \cite{FEM,Example2}, wave propagation approaches \cite{Xu2020,mace1984wave} or Laplace transform methods (for the restricted case where the Laplace transform of the solution can be written down analytically) \cite{zheng2002quasi,akoz1999mixed,chen1995hybrid}. Such approaches can suffer from the limited accuracy and high computational cost associated with local approaches, as described above, or have been restricted to the case of constant beam parameters where one can obtain semi-analytical results. We also note that it is popular to use the Riemann--Lioville form of the fractional derivative. Whilst this is entirely correct for zero initial conditions and commonly used to simplify the study of the effects of forcing, it does not account correctly for the history of the deformation for non-zero initial conditions. Rather, one should use the Caputo definition to ensure the time-history is appropriately accounted for \cite{podlubny1998fractional,caputo1967linear}. We outline how our computational approach can use either definition\footnote{Under suitable regularity assumptions, both definitions agree for the case of zero initial conditions.}, but focus on the case of the Caputo definition when analysing physical models.

In this paper, we rectify the modelling of the fractional constituent relation of a linear viscoelastic beam and use it to illustrate our contour-based numerical method. Our approach, as well as being able to cope with non-zero initial deformations, variable coefficients and forcing etc., converges rapidly, stably and can produce results to a user-specified desired accuracy. Code for our method is provided at \textcolor[rgb]{0,0,1}{\url{https://github.com/MColbrook/Contour_frac_PDE}}.

\subsection{Organisation of paper}

The paper is organised as follows. In \S \ref{sec:gen_desc_cont}, we describe the general contour method, including a discussion of the choice of contour and optimal contour parameters in \S \ref{cont_descrip} (see Algorithms \ref{alg:spec_meas} and \ref{alg:spec_meas2}). \S \ref{sec:three} analyses the case of fractional beam equations introduced in \S \ref{sec:phys_setup}, with an analysis of the convergence of the method in \S \ref{sec:solve_desc}. Physically relevant examples are presented in \S \ref{sec:phys_res} and concluding remarks in \S \ref{sec:conc}.

\section{Description of the general contour method}
\label{sec:gen_desc_cont}

Here, we provide a formal description of the method. The specific example considered in this paper is a class of fractional beam equations, described in \S \ref{sec:phys_setup}. The general problem reduces to the inversion of the Laplace transform \cite{arendt2001cauchy}, which is achieved via contour methods and quadrature. The Laplace transform itself is computed using spectral methods (see \S \ref{sec:solve_desc}) for a generalised resolvent of a fractional operator pencil (see \S \ref{sec:finding_spec}). 

\subsection{General setup}\label{sec:gen_abstract}

We first formally describe the method for a general time-fractional equation of the form
\begin{equation}
\label{abstract_form}
\sum_{j=1}^M \mathcal{D}^{\nu_j}_{\mathcal{I}_j,t} A_j q= f(t) \text{ for }t\geq 0,
\end{equation}
where $\mathcal{D}^{\nu}_{\mathcal{I},t}$ is given by
$$
\left[\mathcal{D}^{\nu}_{\mathcal{I},t} g\right](t)=\begin{dcases}
\frac{1}{\Gamma(n-\nu)}\frac{d^n}{dt^n}\int_{0}^t (t-\tau)^{n-\nu-1}g(\tau)d\tau,\quad &\text{if }\mathcal{I}=\mathrm{RL},\\
\frac{1}{\Gamma(n-\nu)}\int_{0}^t (t-\tau)^{n-\nu-1}g^{(n)}(\tau)d\tau ,\quad &\text{if }\mathcal{I}=\mathrm{C},
\end{dcases}
$$
for $n-1<\nu< n$ ($n\in\mathbb{N}$ so that $\nu>0$, also if $\nu\in\mathbb{N}$ then we define $\mathcal{D}^{\nu}_{\mathcal{I},t}$ to be the usual integer order derivative). The case of $\mathcal{I}=\mathrm{RL}$ is the Riemann--Liouville fractional derivative, whereas $\mathcal{I}=\mathrm{C}$ corresponds to Caputo's fractional derivative \cite{caputo1967linear}. We suppose that $(n_j-1)<\nu_j\leq n_j$ for $n_j\in\mathbb{N}$, that each $A_j$ represents an operator (independent of time) on a Hilbert space $\mathcal{H}$ with dense domain $\mathcal{D}(A_j)\subset\mathcal{H}$, and that $f$ is some known suitable forcing term. Our convention/notation for the Laplace transform throughout is
\begin{equation}
\label{laplace_defjkjk}
\hat g(z)=\mathcal{L}\{g\}(z):=\int_0^\infty e^{-zt}g(t)dt.
\end{equation}
Taking Laplace transforms of \eqref{abstract_form}, we formally arrive at
\begin{equation}
\label{laplace_form}
\left[\sum_{j=1}^M z^{\nu_j}A_j\right]\hat q (z)=\hat f(z)+\sum_{\mathcal{I}_j=\mathrm{C}\text{ or }\nu_j=n_j}A_j\sum_{k=1}^{n_j}z^{\nu_j-k}q^{(k-1)}(0).
\end{equation}
For the Laplace transforms of fractional derivatives, see \cite[Ch. 2]{podlubny1998fractional}.\footnote{We have implicitly assumed sufficient regularity so that the limit values of the fractional derivatives/integrals, that appear when taking Laplace transforms of Riemann--Liouville derivatives, vanish at the lower terminal $t=0$. This is also of physical importance since there is no natural interpretation (in the models we consider) for the initial values of these fractional derivatives/integrals.}
We set
\begin{equation}
\label{laplace_form2}
K(z):=\hat f(z)+\sum_{\mathcal{I}_j=\mathrm{C}\text{ or }\nu_j=n_j}A_j\sum_{k=1}^{n_j}z^{\nu_j-k}q^{(k-1)}(0),
\end{equation}
and assume that this is known. We formally define the fractional operator pencil as
$$
T(z)=\sum_{j=1}^M z^{\nu_j}A_j,
$$
which we assume for suitable complex $z$ can be extended from the initial domain $\cap_{j=1}^M\mathcal{D}(A_j)$ to a closed operator with dense domain $\mathcal{D}(T(z))$. In general, care must be taken to ensure that this is done correctly (e.g., see \cite{liu1998spectrum} for an example where $\mathcal{D}(T(z))\neq \cap_{j=1}^M\mathcal{D}(A_j)$) and we provide the required analysis for our example in \S \ref{sec:finding_spec}. We then define the generalised spectrum and resolvent set as
$$
\mathrm{Sp}\left(\{A_j\},\{\nu_j\}\right)=\overline{\{z\in\mathbb{C}: T(z)\text{ is not invertible}\}},\quad \rho\left(\{A_j\},\{\nu_j\}\right)=\mathbb{C}\backslash \mathrm{Sp}\left(\{A_j\},\{\nu_j\}\right),
$$
respectively. Finally, we invert the Laplace transform to arrive at
\begin{equation}
\label{eq:inv_lap}
q(t)=\frac{1}{2\pi i}\int_{\omega -i\infty}^{\omega+i\infty}e^{zt}\underbrace{\left[T(z)^{-1}K(z)\right]}_{\hat q (z)\in\mathcal{H}}dz,
\end{equation}
where $\omega\in\mathbb{R}$ is such that singularities of $T(z)^{-1}K(z)$ are to the left of the contour.

\subsection{Contour choice and quadrature}
\label{cont_descrip}

Our approach when computing \eqref{eq:inv_lap} is to deform the contour of integration into the left-half plane, so that the integrand decays, and apply quadrature methods. This idea can be traced back to the 1950s and Talbot's doctoral student Green \cite{green1955calculation}, as well as Butcher \cite{butcher1957numerical}. Later, Talbot published a landmark paper \cite{talbot1979accurate}, where he generalized and improved the earlier work of Green. One deforms the contour of integration into a contour $\gamma$ (parametrised by $s\in\mathbb{R}$) that begins and ends in the left half-plane, such that $\mathrm{Re}(\gamma(s))\rightarrow-\infty$ as $|s|\rightarrow\infty$. Popular contour choices include variations of Talbot's contour \cite{weideman2006optimizing,dingfelder2015improved}, parabolic contours \cite{gavrilyuk2001exponentially,weideman2019gauss,weideman2007parabolic} and hyperbolic contours \cite{lopez2006spectral,sheen2003parallel,weideman2007parabolic}. Fast and accurate integration can then be achieved by the simple trapezoidal or midpoint rules \cite{martensen1968numerischen,mcnamee1964error,stenger2012numerical} (see also \cite{weideman2019gauss} for Gauss--Hermite quadrature). In this paper, we focus on the trapezoidal rule given by
\begin{equation}
\label{bromwich_quad_example}
q(t)\approx \frac{h}{2\pi i}\sum_{j=-N}^Ne^{z_jt}\hat q(z_j)\gamma'(jh),\quad z_j=\gamma(jh).
\end{equation}
It is crucial to consider the \textit{numerical stability} of the sum in \eqref{bromwich_quad_example} \cite{lopez2004numerical}. If $t\max(\mathrm{Re}(z_j))$ is unbounded as $N\rightarrow\infty$, then the exponential terms in \eqref{bromwich_quad_example} increase and render the sum unstable. This is demonstrated in the error plots of \cite{weideman2007parabolic}, which considers optimal choices of parameters (in exact arithmetic) for parabolic and hyperbolic contours. Several methods have been developed to overcome this instability. A mechanism for providing stability in the case of Talbot contours and for Laplace transforms with singularities on the negative real axis is given in \cite{dingfelder2015improved}. The papers \cite{lopez2006spectral,colbrook2021computing} and \cite{weideman2010improved} provide stable choices of parameters for hyperbolic and parabolic contours respectively.

In the setting of the present paper, the singularities of the Laplace transform are not restricted to the negative real axis. We therefore focus on hyperbolic and parabolic contours. A hyperbolic contour is suited to the general case where the singularities of $\hat q(z)$ lie in the exterior of a translated sector (see Remark \ref{rem:poles} for a discussion of isolated poles in the interior)
\begin{equation}
\label{sec_def_n}
S_{\delta,\sigma}:=\{z\in\mathbb{C}:\mathrm{arg}(z-\sigma)<\pi-\delta\},
\end{equation}
for some $\delta\in[0,\pi/2)$ and $\sigma\in\mathbb{R}$. We parametrise hyperbolic contours as in \cite{weideman2007parabolic}:
\begin{equation}
\label{contour_choice}
\gamma(s)=\sigma+\mu(1+\sin(is-\alpha)),\quad \mu>0,\quad 0<\alpha<\frac{\pi}{2}-\delta.
\end{equation}
An example contour is shown in Figure \ref{fig:cont_examples} (left). Since it is beneficial to reuse the computed resolvents at different times, we consider computing $u(t)$ for $t\in[t_0,t_1]$ where $0<t_0\leq t_1$. In \cite{colbrook2021computing}, the parameters were optimised, under the assumption that $\gamma(0)t_1=\mu t_1(1-\sin(\alpha))\leq\beta$ for some $\beta>0$, as $N\rightarrow\infty$. This extra free parameter $\beta$ controls the maximum size of the exponential terms in the sum \eqref{bromwich_quad_example}. In the case of $\sigma\leq 0$, this is clear. In the case of $\sigma>0$, there is an unavoidable intrinsic instability of the problem as $t\rightarrow\infty$. However, one can still obtain stable bounds for \textit{relative} errors for arbitrary $t$ and stable bounds for absolute errors for bounded $t$. For all of the examples in this paper, $\sigma=0$.

\begin{figure}[t!]
  \centering
  \begin{minipage}[b]{0.48\textwidth}
    \begin{overpic}[width=\textwidth,trim={32mm 92mm 32mm 92mm},clip]{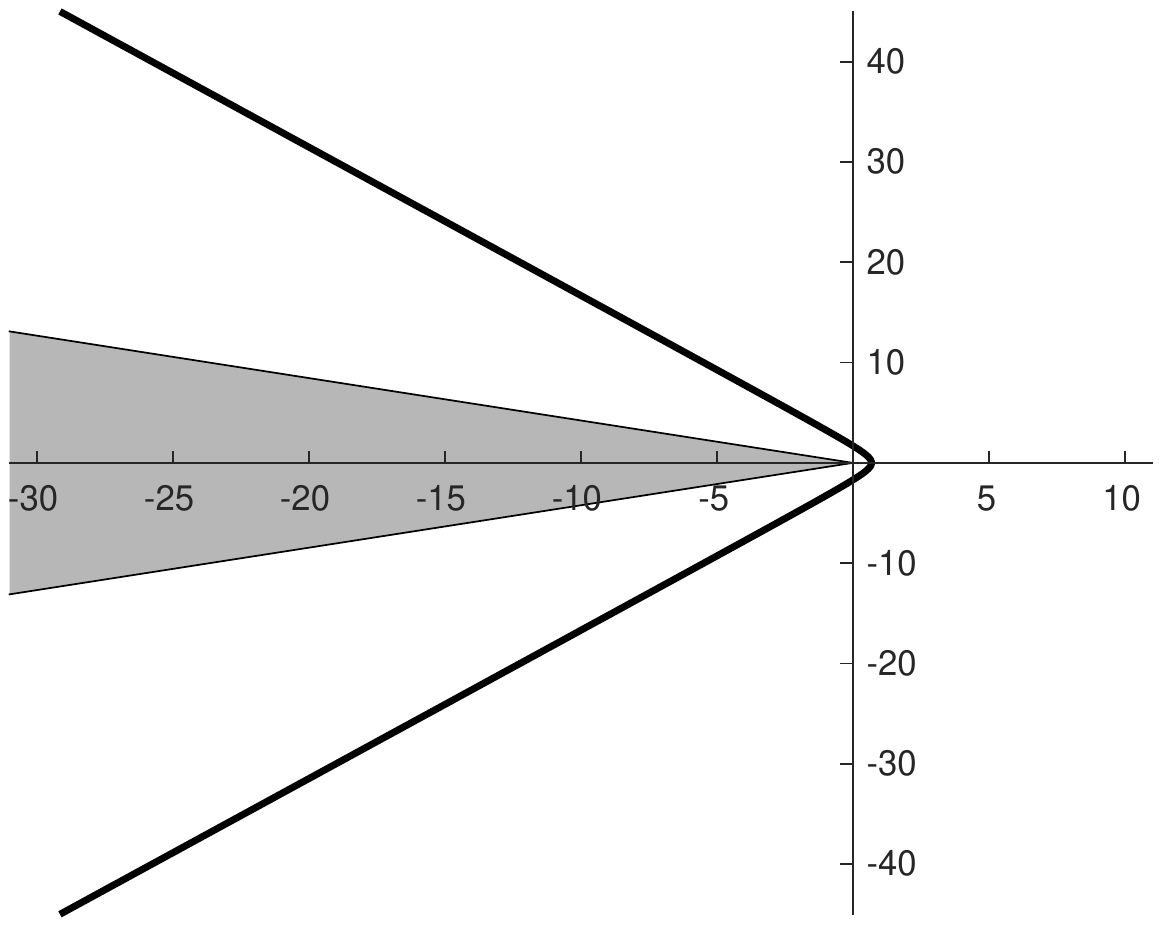}
		\put (47,59) {$\displaystyle \gamma$}
		\put (18,43) {$\displaystyle S_{\delta,0}^{c}$}
		\put (57,12) {$\displaystyle \mathrm{Im}(z)$}
		\put (82,43) {$\displaystyle \mathrm{Re}(z)$}
     \end{overpic}
  \end{minipage}
  \hfill
  \begin{minipage}[b]{0.48\textwidth}
    \begin{overpic}[width=\textwidth,trim={32mm 92mm 32mm 92mm},clip]{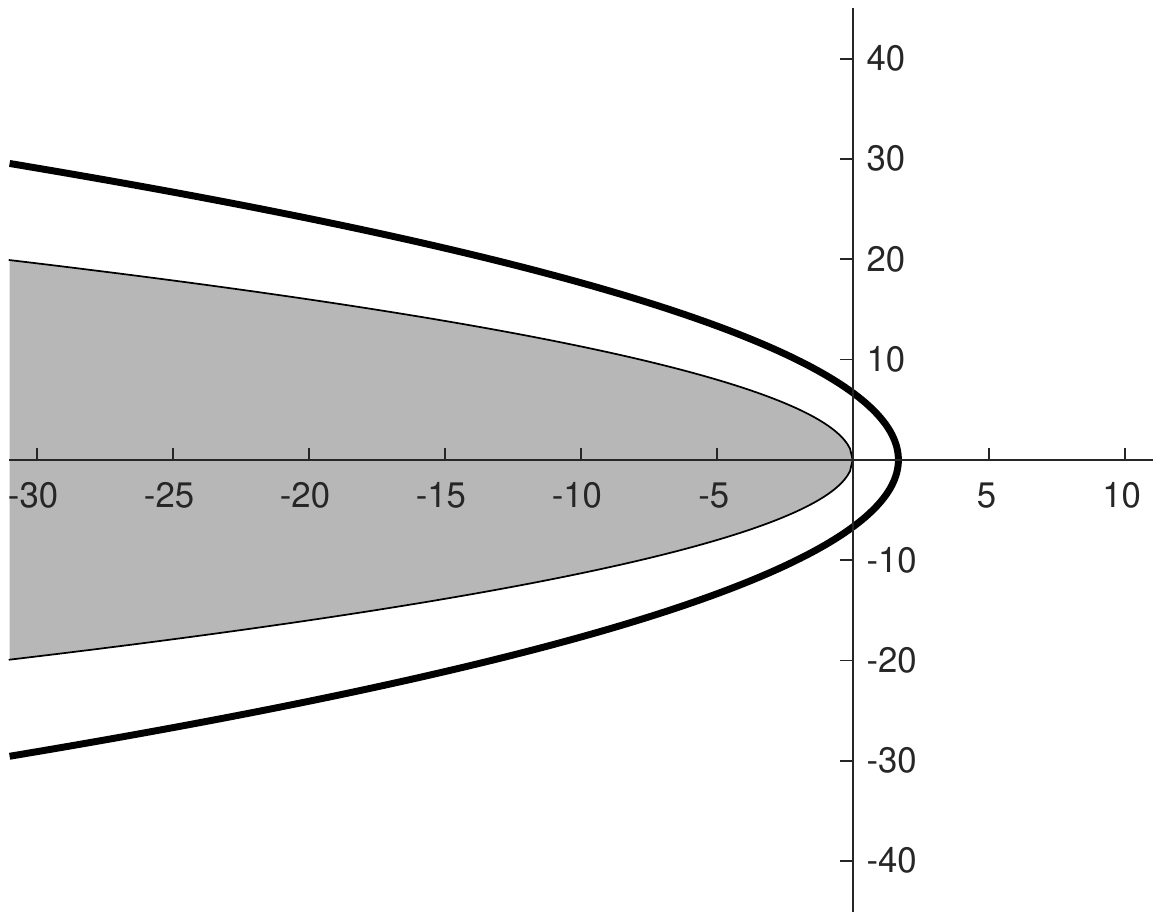}
		\put (47,57) {$\displaystyle \gamma$}
		\put (20,45) {$\displaystyle P_{\delta,0}^{c}$}
		\put (57,12) {$\displaystyle \mathrm{Im}(z)$}
		\put (82,43) {$\displaystyle \mathrm{Re}(z)$}
		\end{overpic}
  \end{minipage}
  \caption{Left: Example sector (complement of $S_{\delta,0}$) and contour $\gamma$ (using Algorithm 1) for $\delta=0.4$. Right:  Example parabolic region (complement of $P_{\delta,0}$) and contour $\gamma$ (using Algorithm 2) for $\delta=0.05$.}
\label{fig:cont_examples}
\end{figure}

The algorithm is summarised in Algorithm \ref{alg:spec_meas} and the following error bound was proven in \cite[Thm. 5.1]{colbrook2021computing}\footnote{This theorem was stated for analytic semigroups and $\sigma=0$, but the proof carries over with minor modifications.} for $t\in[t_0,t_1]$ as $N\rightarrow\infty$:
\begin{equation}
\label{analytic_bound}
\begin{split}
e^{-\sigma t}\left\| q(t)-{q}_N(t)\right\|&\leq  \underbrace{\left({2\mu e^{\frac{\beta}{1-\sin(\alpha)}}}{\pi^{-1}}   \int_0^\infty  e^{x-\mu t\sin(\alpha)\cosh(x)}dx\right) \eta}_{\text{error due to inexact computation of $\hat q(z)$}}\\& +\underbrace{C e^{\frac{\beta}{1-\sin(\alpha)}}\cdot\exp\left(-\frac{N\pi(\pi-2\delta)/2}{\log\left(\Lambda  \frac{\sin(\pi/4-\delta/2)^{-1}-1}{\beta}N\pi(\pi-2\delta)\right)}\right)}_{\text{quadrature error}},
\end{split}
\end{equation}
where $C$ denotes a constant (that in general depends on bounds on $\|\hat q(z)\|$ in suitable regions of the complex plane), $\|\cdot\|$ denotes the norm of $\mathcal{H}$ and $\eta$ denotes an error tolerance for computing $\hat q(z)$. The first error term corresponds to the approximation of $\hat q(z)$, whereas the second corresponds to the quadrature error of the integral of the inverse Laplace transform.

\begin{algorithm}[t]
\textbf{Input:} $\hat q$ (with singularities in the exterior of $S_{\delta,\sigma}$), $0<t_0\leq t_1<\infty$, $\beta>0$, $N\in\mathbb{N}$ and $\eta>0$. \\
\vspace{-4mm}
\begin{algorithmic}[1]
\STATE Let $\gamma$ be defined as in \eqref{contour_choice}, with $\mu, h$ and $\alpha$ given by:
\begin{align*}
\mu&=(1-\sin(({\pi-2\delta})/{4}))^{-1}\beta/t_1,\\
h&=\frac{1}{N}W\left(N\frac{\pi(\pi-2\delta)}{t_0\mu \sin\left(\frac{\pi-2\delta}{4}\right)}\right)=\frac{1}{N}W\left(\Lambda N\frac{\pi(\pi-2\delta)}{\beta \sin\left(\frac{\pi-2\delta}{4}\right)}\left(1-\sin\left(\frac{\pi-2\delta}{4}\right)\right)\right),\\
\alpha&=({h\mu t_1+\pi^2-2\pi\delta})/({4\pi}),
\end{align*}
where $\Lambda=t_1/t_0$ and $W$ denotes the principal branch of the Lambert $W$ function.
\STATE Set $z_j=\gamma(jh)$ and $w_j=\frac{h}{2\pi i}\gamma'(jh)$.
\STATE Compute $v_j=\hat q(z_j)$ for $-N\leq j\leq N$ to an accuracy $\eta$. In our setting, this is achieved by approximating $T(z_j)^{-1}K(z_j)$.
\end{algorithmic} \textbf{Output:} ${q}_N(t)=\sum_{j=-N}^Ne^{z_jt}w_jv_j$ for $t\in[t_0,t_1]$.
\caption{Optimal stable parameter selection for hyperbolic contour based on \cite{colbrook2021computing}.}\label{alg:spec_meas}
\end{algorithm}

In practice, to compute the Laplace transform $\hat q(z)$ to a given accuracy $\eta$, we can use bounds on $\|T(z)^{-1}\|$ along the contour of integration (see \S \ref{sec:solve_desc}), so it is important that $\|T(z)^{-1}\|$ is not too large along the contour. Moreover, to ensure good convergence properties as $N\rightarrow\infty$, it is important that the contour parameters are chosen based on regions in the complex plane where $\|T(z)^{-1}\|$ is suitably bounded. In the special case that $T(z)$ is a linear operator pencil of the form $zI-A$ (for operator $A$), this corresponds to bounding the pseudospectrum\footnote{Recall that for $\epsilon>0$, the pseudospectrum, $\mathrm{Sp}_{\epsilon}(A)$, is the closure of the set of points where $\|(A-zI)^{-1}\|^{-1}<\epsilon$.} of $A$. Examples of pseudospectral analysis performed for contour integral methods include \cite{in2011contour,weideman2010improved,guglielmi2020numerical,guglielmi2020pseudospectral}. In our case, we bound $\|T(z)^{-1}\|$ in Theorem \ref{thm:RES_BOUND} (see Figure \ref{fig:pseudospectra} for some examples) through the analysis of a quasi-linearised operator pencil defined in Proposition \ref{prop:quasi_lin}. In certain parameter regimes and for $z$ reasonably close to the imaginary axis, the relevant sets can be bounded by parabolae. For points $z_j$ with $\mathrm{Re}(z_jt)\ll 0$, the contribution to the sum in \eqref{bromwich_quad_example} is negligible and can be safely neglected. We therefore also consider parabolic contours.

In the case that the singularities of $\hat q(z)$ lie in the exterior of
$$
P_{\delta,\sigma}:=\left\{z\in\mathbb{C}:\mathrm{Re}(z)>\sigma-\delta \cdot \mathrm{Im}(z)^2\right\},
$$
for $\sigma\in\mathbb{R}$ and $\delta>0$, we parametrise parabolic contours as in \cite{in2011contour}:
\begin{equation}
\label{contour_choice_parab}
\gamma(s)=\sigma-\frac{1}{4\delta}+\mu(1+is)^2,\quad \mu>\frac{1}{4\delta}.
\end{equation}
An example contour is shown in Figure \ref{fig:cont_examples} (right). Following \cite{in2011contour}, there are three competing error terms in exact arithmetic:
\begin{align*}
E_1=\mathcal{O}\left(e^{\sigma t-\frac{2\pi}{h}\left(1-\frac{1}{2\sqrt{\mu \delta}}\right)}\right),\quad E_2=\mathcal{O}\left(e^{\sigma t-\frac{t}{4\delta}-\frac{\pi^2}{\mu t h^2}+\frac{2\pi}{h}}\right),\quad E_3=\mathcal{O}\left(e^{\sigma t-\frac{t}{4\delta}+\mu t(1-(hN)^2)}\right).
\end{align*}
The first two correspond to discretisation errors of the integral, whereas the third corresponds to a truncation error of the Trapezoidal rule to a finite sum. To model errors in inexact arithmetic (and ensure stability), we follow \cite{weideman2010improved} and consider a fourth error term given by
$$
E_4=\mathcal{O}\left(e^{\sigma t-\frac{t}{4\delta}+\mu t +\log(\eta)}\right).
$$
The optimal parameters are then selected according to Algorithm \ref{alg:spec_meas2}. The optimisation problem in \eqref{parab_opt_prob} can be efficiently approximately solved using MATLAB's \texttt{fminsearch}, for example.

\begin{algorithm}[t]
\textbf{Input:} $\hat q$ (with singularities in the exterior of $P_{\delta,\sigma}$), $0<t_0\leq t_1<\infty$, $N\in\mathbb{N}$ and $\eta>0$. \\
\vspace{-4mm}
\begin{algorithmic}[1]
\STATE Let $\gamma$ be defined as in \eqref{contour_choice_parab}, where $h$ and $\mu$ are selected as solutions of
\begin{equation}
\label{parab_opt_prob}
\hspace{-4mm}\min_{\substack{h>0\\\mu>1/(4\delta)}} \max_{t\in\{t_0,t_1\}}\left\{-\frac{2\pi}{h}\left(1-\frac{1}{2\sqrt{\mu \delta}}\right),-\frac{t}{4\delta}-\frac{\pi^2}{\mu t h^2}+\frac{2\pi}{h},-\frac{t}{4\delta}+\mu t(1-(hN)^2),-\frac{t}{4\delta}+\mu t +\log(\eta)\right\}.
\end{equation}
\STATE Set $z_j=\gamma(jh)$ and $w_j=\frac{h}{2\pi i}\gamma'(jh)$.
\STATE Compute $v_j=\hat q(z_j)$ for $-N\leq j\leq N$ to an accuracy $\eta$. In our setting, this is achieved by approximating $T(z_j)^{-1}K(z_j)$.
\end{algorithmic} \textbf{Output:} ${q}_N(t)=\sum_{j=-N}^Ne^{z_jt}w_jv_j$ for $t\in[t_0,t_1]$.
\caption{Optimal stable parameter selection for parabolic contour based on \cite{in2011contour,weideman2010improved}.}\label{alg:spec_meas2}
\end{algorithm}

\begin{remark}[Dealing with singularities of $\hat f(z)$]\label{rem:poles}
If the singularities of $T(z)^{-1}(K(z)-\hat f(z))$ lie in a region of the complex plane suited to the above analysis but $\hat f(z)$ has additional poles, we can isolate the contributions from the poles to the right of the contour $\gamma$ by Cauchy's residue theorem. This is done in the examples in \S \ref{sec:solve_desc} and \S \ref{sec:phys_res}.
\end{remark}

\begin{remark}[Oscillatory integrals]
The contours $\gamma$ that we use extend into the left-half-plane so that the integrand decays exponentially. However, if large portions of the contour are close to a vertical line of constant real part, a portion of the integral becomes oscillatory for large times. In this case, it may be beneficial to use quadrature methods for oscillatory integrals \cite{deano2017computing} for that part of the integral. A comparison between the plain trapezoidal rule and a hybrid trapezoidal-oscillatory quadrature rule is beyond the scope of this paper but merits further investigation.
\end{remark}

\subsection{Causality and the computation of $\hat f(z)$}

In the above, we have implicitly assumed access to $\hat f(z)$. In the examples of this paper, one can analytically compute the Laplace transform of the forcing term (which is needed for $K(z)$). In general, however, this must be done numerically. Suppose that we wish to evaluate the solution at time $t<t_1$. Then the part of the solution in \eqref{eq:inv_lap} corresponding to the forcing term can be written as
\begin{equation}
\label{forcing_split}
\frac{1}{2\pi i}\int_{\omega-i\infty}^{\omega+i\infty}e^{zt}T(z)^{-1}[\hat{f_1}(z)+\hat{f_2}(z)]\, dz,
\end{equation}
where we write
$$
f(s)=f_1(s)+f_2(s)
$$
for $f_1$ supported on $[0,t_1]$ and $f_2$ supported on $(t_1,+\infty)$. The contribution of $\hat{f_1}$ in \eqref{forcing_split} encodes the forcing up to time $t_1$. We can use quadrature methods to compute $\hat{f_1}(z)$, noting that the integration in \eqref{laplace_defjkjk} is now over a finite time interval. The contribution of $\hat{f_2}$ in \eqref{forcing_split} may initially seem troubling to the reader since it includes the forcing at future times. Fortunately this part vanishes, as it must due to causality, as we now argue. One has
$$
\frac{1}{2\pi i}\int_{\omega-i\infty}^{\omega+i\infty}e^{zt}T(z)^{-1}\hat{f_2}(z)\, dz=\frac{1}{2\pi i}\int_{\omega-i\infty}^{\omega+i\infty}e^{z(t-t_1)}T(z)^{-1}\int_{0}^\infty e^{-zs}{f_2}(t_1+s)\,ds\, dz,
$$
where the equality comes from a change of variables in the integral definition of $\hat{f_2}(z)$. Since $t<t_1$, the integrand decays exponentially as $\mathrm{Re}(z)\rightarrow \infty$. This allows us to deform in the right-half plane and see that the above integral vanishes. Similarly, one can argue that only $f(s)$ up to $s=t$ contributes to $q(t)$ in \eqref{eq:inv_lap}.

\section{Fractional beam equations}
\label{sec:three}

In this section, we introduce the fractional beam equations studied in this paper. After describing the physical setup, we analyse the relevant fractional pencil $T(z)$ (defined in \eqref{def:T}) via a quasi-linearisation (Proposition \ref{prop:quasi_lin}). Linearisation is a common tool employed when studying polynomial operator pencils \cite{liu1998spectrum} \cite[Chapter VI]{engel1999one}. However, in our case, the fractional power $z^\nu$ in \eqref{eq:beam_pencil} only allows a linearisation of the dominant quadratic term (and hence we cannot reduce the time evolution problem to a semigroup). Proposition \ref{prop:quasi_lin} leads to results describing suitable $z$ for which the inverse $T(z)^{-1}$ exists and suitable bounds on its norm (Theorem \ref{thm:RES_BOUND}). With this information in hand, we discuss the choice of contour and method of computing $T(z)^{-1}$. We end this section with an example showing numerical convergence, which follows rigorously from Algorithm \ref{alg:spec_meas}, the bound \eqref{analytic_bound}, and the convergence of the used spectral method.

\subsection{Physical setup}
\label{sec:phys_setup}

Constituent equations determining the stress-strain relationship for viscoelastic beams rely on the inclusion of vibrational-damping terms. Arguably, the most well-known damped stress-strain relation is the Kelvin--Voigt equation \cite{Kelvin,Voigt}, applicable for beams undergoing small deformations (governed by Euler--Bernoulli beam theory). Since its conception as a simple constitutive model, the Kelvin--Voigt equation has been modified by the inclusion of time-fractional derivatives to better represent the relaxation behaviour observed in viscoelastic materials \cite{KV_AIAA,KV1,KV2}. The use of fractional derivatives in these adapted constituent equations enables key physical features of the material to be modelled up to the desired accuracy (with additional fractional derivatives and scaling parameters potentially permitting greater accuracy) \cite{Pritz}. Therefore, the ability to solve these equations swiftly and accurately is crucial to a range of vibration problems.

As discussed in \S \ref{sec:intro}, typical solutions of fractional derivative equations are computed in the time domain and can suffer from numerical convergence issues. For example, a frequency-domain approach \cite{Xu2020} is posed to rectify some of these issues. However, its accuracy is unclear (with a 1.2\% error quoted in the paper). In contrast, our method can attain a pre-specified degree of accuracy, suitable for fractional constituent relations and variable beam parameters (e.g., variable mass). Thus, this approach can ensure the predicted vibrations of the structure are at least as accurate as whatever the chosen constituent equation is.

\begin{figure}[t!]
  \centering
    \begin{overpic}[width=0.8\textwidth]{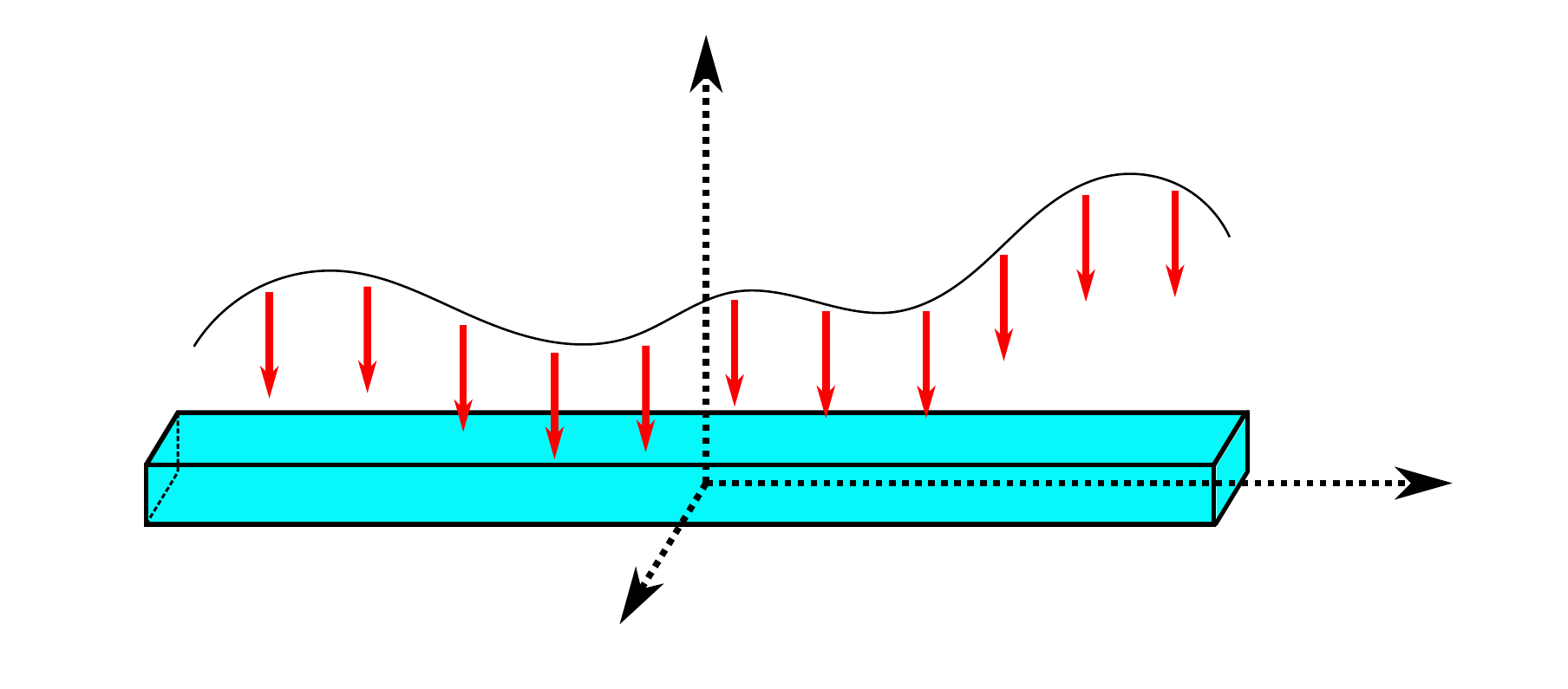}
		\large
		\put (82,9) {$\displaystyle x$}
		\put (42,35) {$\displaystyle y$}
		\put (37,6) {$\displaystyle z$}
		\put (75,35) {$\displaystyle F(x,t)$}
		\normalsize
     \end{overpic}
    \caption{Physical setup of the model beam (shaded) in this paper.}
\label{fig:beam}
\end{figure}
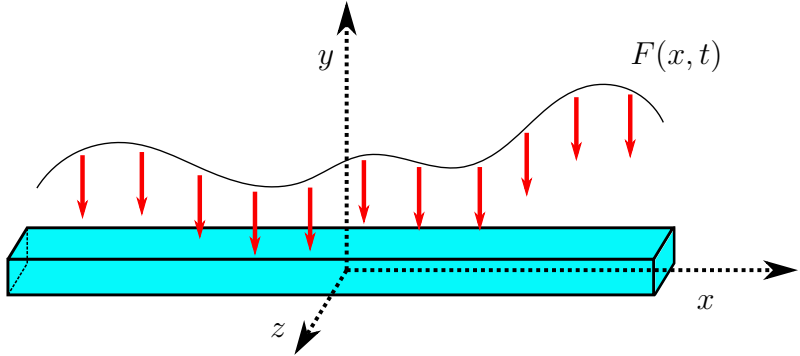

We consider the following viscoelastic constituent equation on a beam depicted in Figure \ref{fig:beam}, with the beam extending lengthwise between $x=-1$ and $x=1$,
\begin{equation}
\label{phsysics_needed101}
\sigma(x,z,t)=E_{0}(x)\epsilon(x,z,t)+E_{1}(x)\mathcal{D}^{\nu}_{\mathcal{I},t}\epsilon(x,z,t),
\end{equation}
where $\sigma$ indicates the normal stresses, $\epsilon$ is the axial strain, $E_{0}$ is the Young's modulus of the beam, and $E_{1}$ and $\nu\in(0,2)$ are the experimentally determined viscoelaticity parameters. We have implicitly non-dimensionalised by a lengthscale (based on half the total beam-length) when considering $x\in[-1,1]$. Further non-dimensionalisation will take place later in Section \ref{sec:phys_res}. For the physical model, we consider the Caputo derivative $\mathcal{I}=\mathrm{C}$ but note that the analysis of \S \ref{sec:finding_spec} holds for the Riemann--Liouville definition of the fractional derivative as well. In particular, the operator pencil does not change, but the right-hand sides of the resulting linear systems when we apply quadrature do.

Under a transverse loading, $F(x,t)$, the transverse displacement, $y(x,t)$, of the beam satisfies
$$
\pderiv{^{2}M}{x^{2}}+\underbrace{\rho(x)A(x)}_{:=\tilde\rho(x)}\pderiv{^{2}y}{t^{2}}=F(x,t),
$$
where $\rho$ is the density of the beam, $A$ is its cross-sectional area, and $M$ is the lateral moment
$$
M(x,t)=-\int_{A(x)} z\sigma(x,z,t)dA.
$$
By assuming small deformations, the axial strain can be approximated by 
$$
\epsilon(x,z,t)=-z\pderiv{^{2}y}{x^{2}}
$$
and hence, using \eqref{phsysics_needed101}, we arrive at the governing equation for the transverse displacement:
\begin{equation}
    \tilde\rho(x) \pderiv{^{2}y}{t^{2}}+\pderiv{^{2}}{x^{2}}\left[E_{0}(x)I(x)\pderiv{^{2}y}{x^{2}}+E_{1}(x)I(x)\mathcal{D}^{\nu}_{\mathcal{I},t}\pderiv{^{2}y}{x^{2}} \right]=F(x,t),\quad (x,t)\in[-1,1]\times[0,\infty).
    \label{eq:gov}
\end{equation}
Here, $I(x)$ denotes the moment of inertia of the cross-section of the beam,
$$
I(x)=\int_{A(x)}z^{2}dA.
$$
End conditions on the beam supplement the governing equation. For simplicity, we shall consider the following typical conditions applied at an end $x_{0}$: a clamped end corresponding to
$$
y(x_{0},t)=\pderiv{y}{x}(x_{0},t)=0,
$$
or a simply supported end corresponding to
$$
y(x_{0},t)=\pderiv{^{2}y}{x^{2}}(x_{0},t)=0.
$$
A mixture of the conditions is also suitable.

\begin{remark}[Generalisations of the model]
Different types of boundary conditions can also be dealt with similarly. However, an analysis similar to that of \S \ref{sec:finding_spec} must be performed to bound the generalised spectrum. Since the methodology presented here is suitable for any number of fractional derivatives, additional fractional parameters of the form $E_{j}(x)\mathcal{D}^{\nu_j}_{\mathcal{I}_j,t}\epsilon(x,z,t)$ could also be included into the stress-strain relation to match experimental data, as suggested by Pritz \cite{Pritz}. Again, in this case, the bounds on the generalised spectrum would need to be performed similarly to \S \ref{sec:finding_spec}. For conciseness, we restrict our attention in this paper to a single fractional derivative.
\end{remark}

\subsection{Generalised spectrum and generalised resolvent of fractional pencil}
\label{sec:finding_spec}

The general form of the governing equation in \S \ref{sec:phys_setup} can be written as 
\begin{equation}
\label{NL_gen_form}
    \pderiv{^{2}y}{t^{2}}+\frac{1}{\tilde\rho(x)}\pderiv{^{2}}{x^{2}}\left[a(x)\pderiv{^{2}y}{x^{2}}+b(x)\mathcal{D}^{\nu}_{\mathcal{I},t}\pderiv{^{2}y}{x^{2}} \right]=\frac{F(x,t)}{\tilde\rho(x)},
\end{equation}
where we assume that $a,b$ and $\tilde\rho$ are continuous and strictly positive on the domain $[-1,1]$. We also have $\nu\in(0,2)$. For simplicity, we shall also assume that $a$ and $b$ are smooth, though weaker assumptions can also be taken in what follows. However, from a numerical viewpoint, this is also a practical consideration. To capture the boundary conditions, we define three suitable Hilbert spaces that are used in the analysis below. The reason for considering these three Hilbert spaces will become clear when we define a quasi-linearisation of the Laplace transform of \eqref{NL_gen_form} in Proposition \ref{prop:quasi_lin}. The operator $\mathcal{A}(z)$ appearing in Proposition \ref{prop:quasi_lin} will be used to bound the generalised spectrum of a suitable fractional pencil $T(z)$.

\subsubsection{The setup}

We define $\mathcal{H}^4_{\mathrm{BC}}$ as the subspace of all $u\in H^4(-1,1)$ that satisfy the boundary conditions of $y(\cdot,t)$ (for a fixed time). Recall that $H^m(-1,1)$ denotes the space of $L^2(-1,1)$ functions whose weak derivatives up to order $m$ exist and lie in $L^2(-1,1)$ \cite{evans2010partial}. Moreover, by the Sobolev embedding theorems \cite[e.g., Section 5.6.3]{evans2010partial}, for any $m\in\mathbb{N}$, any $u\in H^m(-1,1)$ also lies in $C^{m-1,1/2}([-1,1])$ and hence the boundary conditions are well-defined in the classical sense. We define $\mathcal{H}^2_{\mathrm{BC1}}$ as the closed subspace of $H^2(-1,1)$ consisting of all $u$ that satisfy the boundary conditions of $y(\cdot,t)$ (for a fixed time) that involve derivatives up to and including first order. The fact that $\mathcal{H}^2_{\mathrm{BC1}}$ is closed and hence a Hilbert space follows from the Sobolev embedding theorems that bound pointwise evaluations of $u$ and $u'$ in terms of the $H^2$ norm of $u$. Furthermore, by using the boundary conditions, an application of Poincaré's inequality and integration by parts, we have that if $u\in\mathcal{H}^2_{\mathrm{BC1}}$, then
$$
\int_{-1}^1|u(x)|^2+|u'(x)|^2dx\lesssim \int_{-1}^1 a(x) u''(x)\overline{u''(x)}dx.
$$
Since $a$ is strictly positive and bounded, we have
$$
\int_{-1}^1 u''(x)\overline{u''(x)}dx\lesssim  \int_{-1}^1 a(x) u''(x)\overline{v''(x)}dx\lesssim \int_{-1}^1 u''(x)\overline{u''(x)}dx.
$$
It follows that we can equip $\mathcal{H}^2_{\mathrm{BC1}}$ with the equivalent inner product
$$
\left\langle u,v \right\rangle_a:=\int_{-1}^1 a(x) u''(x)\overline{v''(x)}dx.
$$
Finally, we define $\mathcal{H}^2_{\mathrm{BC2}}$ as the subspace of $H^2(-1,1)$ consisting of all $u$ that satisfy the boundary conditions of $y(\cdot,t)$ (for a fixed time) that involve derivatives of second order, but with $y''$ replaced by $u$. In other words, if we are given the boundary condition $y''(x_0,t)=0$, then this corresponds to $u(x_0)=0$ for $u\in\mathcal{H}^2_{\mathrm{BC2}}$.

We consider the abstract setup of \S \ref{sec:gen_abstract} with $q(t)=y(\cdot,t)$ and the Hilbert space $\mathcal{H}=L^2_{\tilde\rho}(-1,1)$ equipped with the weighted inner product
$$
\langle y_1,y_2 \rangle_{\tilde\rho}:=\int_{-1}^1 \tilde\rho(x)y_1(x)\overline{y_2(x)}dx.
$$
To transform \eqref{NL_gen_form} into \eqref{abstract_form}, we need to define the suitable operators on $L^2_{\tilde\rho}(-1,1)$. We first set $\nu_1=2$ and let $A_1$ be the identity operator. The operator $A_2=A$ is then defined by $Ay =(ay'')''/{\tilde\rho}$ with $\mathcal{D}(A)=\mathcal{H}^4_{\mathrm{BC}}$. Similarly, $A_3=B$ is defined by $By =(by'')''/{\tilde\rho}$ with $\mathcal{D}(B)=\mathcal{H}^4_{\mathrm{BC}}$. Finally, we set $\nu_2=0$ and $\nu_3=\nu\in(0,2)$. We therefore consider the \textit{initial} fractional pencil:
\begin{equation}
\label{eq:beam_pencil}
[S(z)]u=z^2u +(au''+z^\nu b u'')''/{\tilde\rho},\quad \mathcal{D}(S(z))=\mathcal{H}^4_{\mathrm{BC}},
\end{equation}
where the branch cut for the fractional power is taken to be $\mathbb{R}_{\leq 0}$. The following simple Proposition is based on a standard argument used in the analysis of PDEs and shows that $S(z)$ is closable. However, a key difference with the usual elliptic PDE theory is that for arbitrary $z\in\mathbb{C}\backslash\mathbb{R}_{\leq 0}$ and $\nu\in(0,2)$, $S(z)$ is not necessarily a closed operator. For example, for some parameter choices $a+z^\nu b$ could vanish at points in $[-1,1]$.

\begin{proposition}
For any $z\in\mathbb{C}\backslash\mathbb{R}_{\leq 0}$, the operator $S(z)$ defined in \eqref{eq:beam_pencil} is closable.
\end{proposition}

\begin{proof}
Let $v\in L^2_{\tilde\rho}(-1,1)$ and $u\in\mathcal{D}(S(z))$. Then
$$
\langle S(z)u,v\rangle_{\tilde\rho}=z^2 \langle u,v\rangle_{\tilde\rho}+\int_{-1}^1 [(au''+z^\nu bu'')''](x)\overline{v(x)}dx.
$$
If $v\in \mathcal{H}^4_{\mathrm{BC}}$, then a simple integration by parts shows that
$$
\langle S(z)u,v\rangle_{\tilde\rho}=z^2 \langle u,v\rangle_{\tilde\rho}+\int_{-1}^1 u(x)\overline{[(av''+\overline{z^{\nu}} bv'')''](x)}dx.
$$
It follows that $v\in\mathcal{D}(S(z)^*)$. Hence $S(z)^*$ is densely defined and so $S$ is closable.
\end{proof}

In what follows, we work with the operator
\begin{equation}
\label{def:T}
T(z):=\mathrm{cl}(S(z))=[S(z)]^{**},
\end{equation}
where $\mathrm{cl}$ denotes taking the closure of the closable operator, and $S(z)$ is the operator defined in \eqref{eq:beam_pencil}.

\subsubsection{Quasi-linearisation}

To study the inverse of $T(z)$ and determine when the inverse exists, we first linearise the quadratic term via the following proposition. This quasi-linearisation also has the benefit of making it considerably easier to prove that the relevant operator, $\mathcal{A}(z)$, is closed.

\begin{proposition}
\label{prop:quasi_lin}
Consider the product space $\mathcal{H}^2_{\mathrm{BC1}}\times L^2_{\tilde\rho}(-1,1)$ equipped with the inner product
\begin{equation}
\label{product_inner_product}
\left\langle (u_0,u_1),(v_0,v_1) \right\rangle=\int_{-1}^1 a(x) u_0''(x)\overline{v_0''(x)}dx+\int_{-1}^1\tilde\rho(x)u_1(x)\overline{v_1(x)}dx.
\end{equation}
For $z\in\mathbb{C}\backslash\mathbb{R}_{\leq 0}$, consider the following operator
\begin{equation*}
\begin{split}
[\mathcal{A}(z)]\begin{pmatrix}
u_0, u_1
\end{pmatrix}
&=z\begin{pmatrix}
u_0,u_1
\end{pmatrix}+\begin{pmatrix}
-u_1,\frac{1}{\tilde\rho}(au_0''+z^{\nu-1}bu_1'')''
\end{pmatrix},\\
\mathcal{D}(\mathcal{A}(z))&=\left\{(u_0,u_1)\in\mathcal{H}^2_{\mathrm{BC1}}\times\mathcal{H}^2_{\mathrm{BC1}}:au_0''+z^{\nu-1}bu_1''\in\mathcal{H}^2_{\mathrm{BC2}}\right\}.
\end{split}
\end{equation*}
The operator $\mathcal{A}(z)$ is closed. Furthermore, if $z$ is such that $\mathcal{A}(z)$ is invertible and $a+z^\nu b$ is non-vanishing on $[-1,1]$, then $S(z)$ is closed (so that $S(z)=T(z)$) and $T(z)$ is invertible with
\begin{equation}
\label{res_form_quasi}
[\mathcal{A}(z)]^{-1}\begin{pmatrix}
0, v
\end{pmatrix}=\begin{pmatrix}
[T(z)]^{-1}v,z[T(z)]^{-1}v
\end{pmatrix},\quad \forall v\in L^2_{\tilde\rho}(-1,1).
\end{equation}
\end{proposition}

\begin{proof}
Let $z\in\mathbb{C}\backslash\mathbb{R}_{\leq 0}$. First we show that $\mathcal{A}(z)$ is closed by considering its adjoint. Suppose that $(v_0,v_1)\in\mathcal{D}(\mathcal{A}(z)^*)$, then the map
$$
(u_0,u_1)\rightarrow \left\langle \mathcal{A}(z)(u_0,u_1),(v_0,v_1) \right\rangle
$$
is continuous (viewed as a map from $\mathcal{D}(\mathcal{A}(z))\subset \mathcal{H}^2_{\mathrm{BC1}}\times L^2_{{\tilde\rho}}(-1,1) $ to $\mathbb{C}$). By restricting this map to vectors with $u_1=0$, we see that
$$
u_0\rightarrow \int_{-1}^1 (a(x)u_0''(x))''\overline{v_1(x)}dx
$$
is continuous on $\mathcal{H}^4_{\mathrm{BC}}\subset \mathcal{H}^2_{\mathrm{BC1}}$. By considering the closed quadratic form associated with the inner product $\langle \cdot,\cdot\rangle_{a}$, it follows that $v_1\in\mathcal{H}^2_{\mathrm{BC1}}$ (on a formal level, the required boundary conditions appear from studying when the boundary terms after integration by parts vanish). For general $(u_0,u_1)\in\mathcal{D}(\mathcal{A}(z))$, this allows us to integrate by parts to deduce that
\begin{equation*}
\begin{split}
\left\langle \mathcal{A}(z)(u_0,u_1),(v_0,v_1) \right\rangle&=z\left\langle (u_0,u_1),(v_0,v_1) \right\rangle\\
&+\int_{-1}^1 a(x)u_0''(x)\overline{v_1''(x)}dx+\int_{-1}^1 u_1''(x)\overline{[-av_0''(x)+\overline{z}^{\nu-1}bv_1''(x)]}dx.
\end{split}
\end{equation*}
The first integral term is continuous in $u_0\in\mathcal{H}^2_{\mathrm{BC1}}$. Set $w=-av_0''+\overline{z}^{\nu-1}bv_1''$, then we must have that
$$
u_1\rightarrow \int_{-1}^1 u_1''(x)\overline{w(x)}dx
$$
is continuous on the subspace $\mathcal{H}^2_{\mathrm{BC1}}\subset L^2_{{\tilde\rho}}(-1,1)$. An elementary analysis and computation now shows that we must have $w\in\mathcal{H}^2_{\mathrm{BC2}}$. It follows that
\begin{equation*}
\begin{split}
[\mathcal{A}(z)^*]\begin{pmatrix}
u_0, u_1
\end{pmatrix}
&=\overline{z}\begin{pmatrix}
u_0,u_1
\end{pmatrix}+\begin{pmatrix}
u_1,\frac{1}{{\tilde\rho}}(-au_0''+\overline{z}^{\nu-1}bu_1'')''
\end{pmatrix},\\
\mathcal{D}(\mathcal{A}(z)^*)&=\left\{(u_0,u_1)\in\mathcal{H}^2_{\mathrm{BC1}}\times\mathcal{H}^2_{\mathrm{BC1}}:-au_0''+\overline{z}^{\nu-1}bu_1''\in\mathcal{H}^2_{\mathrm{BC2}}\right\}.
\end{split}
\end{equation*}
We apply the same arguments to deduce that $\mathcal{A}(z)^{**}=\mathcal{A}(z)$. It follows that $\mathcal{A}(z)$ is closed.

Now suppose that $z\in\mathbb{C}\backslash\mathbb{R}_{\leq 0}$ is such that $\mathcal{A}(z)$ is invertible and $a+z^\nu b$ is non-vanishing on $[-1,1]$. Let $v\in L^2_{{\tilde\rho}}(-1,1)$ and set
\begin{equation}
\label{unravel1}
(u_0,u_1)=[\mathcal{A}(z)]^{-1}\begin{pmatrix}
0, v
\end{pmatrix}.
\end{equation}
It follows that $u_0\in\mathcal{H}^2_{\mathrm{BC1}}$, $au_0''+z^{\nu-1}bu_1''\in \mathcal{H}^2_{\mathrm{BC2}}$ and $
zu_1+(au_0''+z^{\nu-1}bu_1'')''/{\tilde\rho}=v.$ By considering the first component of \eqref{unravel1}, we must have $u_1=zu_0$ and hence since $a+z^\nu b$ is non-vanishing on $[-1,1]$, we must have that $u_0\in\mathcal{H}^4_{\mathrm{BC}}$ with $[S(z)]u_0=v.$ It follows that $S(z)$ is surjective and hence $T(z)$ is surjective. Similarly, we see that $\mathrm{Ker}(S(z))=\{0\}$, otherwise the kernel of $\mathcal{A}(z)$ is non-trivial, yielding a contradiction. Let $v\in L^2_{{\tilde\rho}}(-1,1)$ and consider
$$
(u_0(v),u_1(v))=[\mathcal{A}(z)]^{-1}\begin{pmatrix}
0, v
\end{pmatrix},
$$
where the notation means that the vectors $u_0$ and $u_1$ depend on $v$. We define $[L(z)]v$ by $[L(z)]v=u_1(v)/z.$ Note that $L$ must be bounded (as an operator on $L^2_{{\tilde\rho}}(-1,1)$) and
$$
[L(z)S(z)]v=v,\quad \forall v\in \mathcal{H}^4_{\mathrm{BC}}.
$$
Since $L(z)$ is bounded and $\mathrm{cl}(S(z))=T(z)$, it follows that $[L(z)T(z)]v=v$ for all $v\in\mathcal{D}(T(z))$. This shows that $\mathrm{Ker}(T(z))=\{0\}$ and hence both $T(z)$ and $S(z)$ are bijections from their respective domains to $L^2_{{\tilde\rho}}(-1,1)$. Since $T(z)$ extends $S(z)$, it follows that $S(z)=T(z)$. In particular, $S(z)$ is closed. Since $S(z)=T(z)$ is closed and a bijection from its domain to $L^2_{{\tilde\rho}}(-1,1)$, its inverse must be closed and hence bounded (by the closed graph theorem). We have $T(z)^{-1}=L(z)$ and \eqref{res_form_quasi} now follows.
\end{proof}

\begin{remark}[When can $a+z^\nu b$ vanish?] If $z\in\mathbb{C}\backslash\mathbb{R}_{\leq 0}$ is such that $|\mathrm{arg}(z)|<\pi/\nu$, then $a+z^\nu b$ is non-vanishing on $[-1,1]$ since $a$ and $b$ are strictly positive. If $\nu\leq 1$, then this sector contains the whole of $\mathbb{C}\backslash\mathbb{R}_{\leq 0}$. Note that $\{z\in \mathbb{C}\backslash\mathbb{R}_{\leq 0}:a(x)+z^\nu b(x)=0 \text{ for some }x\in[-1,1]\}$ must be a bounded (possibly empty) set in the closed left-half-pane for $\nu\in(0,2)$. We will see in \S \ref{sec:gen_num_range} that $\mathcal{A}(z)$ is invertible for $z$ in the interior of a sector $S_{\delta,0}$ (defined in \eqref{sec_def_n}) for some $\delta\in[0,\pi/2)$. It follows that for $\nu>1$, we can still deform the contour of integration away from the set of $z$ where $a+z^\nu b$ can vanish.
\end{remark}

\subsubsection{Bounding the generalised spectrum and generalised resolvent}
\label{sec:gen_num_range}

To bound the generalised spectrum of $\mathcal{A}(z)$ and hence of $T(z)$, we will make use of the following proposition. We shall use this in Theorem \ref{thm:RES_BOUND} to bound the inverse $T(z)^{-1}$ in suitable regions of the complex plane that allow us to use the contour method of \S \ref{cont_descrip}.

\begin{proposition}
\label{prop:polar_curve}
Let $z=re^{i\theta}\in \mathbb{C}\backslash \mathbb{R}_{\geq 0}$ and $\epsilon>0$. Suppose that $\mathrm{Re}(z)\geq \epsilon$ or $0<|(2-\nu)\theta|<\pi$ with
\begin{equation}
\label{r_cond_gen_pseudo}
r^{\nu/2}\geq 2\sqrt{\max_{x\in[-1,1]}\frac{a(x)}{b(x)}}\sqrt{\frac{\left(\frac{\epsilon}{r}-\cos(\theta)\right)\left|\cos((\nu-1)\theta)\right|}{\sin^2((2-\nu)\theta)}}+\frac{r^{\nu/2-1}\epsilon\sqrt{2}}{\left|\sin((2-\nu)\theta)\right|}.
\end{equation}
Then $\mathcal{A}(z)^{-1}$ exists with
\begin{equation}
\|\mathcal{A}(z)^{-1}\|\leq\epsilon^{-1},
\end{equation}
where $\|\cdot\|$ denotes the operator norm induced by the norm of $\mathcal{H}^2_{\mathrm{BC1}}\times L^2_{{\tilde\rho}}(-1,1)$.
\end{proposition}

\begin{proof}
The numerical range of the operator $\mathcal{A}(z)$ is given by
$$
\mathcal{W}(\mathcal{A}(z)):=\left\{\left\langle [\mathcal{A}(z)](u_0,u_1),(u_0,u_1)\right\rangle:(u_0,u_1)\in \mathcal{D}(\mathcal{A}(z)),\|(u_0,u_1)\|=1\right\},
$$
where $\|\cdot\|$ denotes the norm induced by the inner product on $\mathcal{H}^2_{\mathrm{BC1}}\times L^2_{{\tilde\rho}}(-1,1)$. We set
$$
r_n(z)=\inf_{\hat z\in\mathcal{W}(\mathcal{A}(z))}|\hat z|.
$$
If $r_n(z)>0$, then since $\mathcal{A}(z)$ is closed, $\mathcal{A}(z)^{-1}$ exists with $\|\mathcal{A}(z)^{-1}\|\leq {r_n(z)}^{-1}.$ Therefore, to prove the lemma, it is enough to provide lower bounds of the form $r_n(z)\geq \epsilon$.

Let $(u_0,u_1)\in \mathcal{D}(\mathcal{A}(z))$ with $\|(u_0,u_1)\|=1$. It follows that
$$
\left\langle [\mathcal{A}(z)](u_0,u_1),(u_0,u_1)\right\rangle=z-\int_{-1}^1 a(x)u_1''(x)\overline{u_0''(x)}dx+\int_{-1}^1\left[a(x)u_0''(x)+z^{\nu-1}b(x)u_1''(x)\right]''\overline{u_1(x)}dx.
$$
The fact that $u_0,u_1\in\mathcal{H}^2_{\mathrm{BC1}}$ and $au_0''+z^{\nu-1}bu_1''\in\mathcal{H}^2_{\mathrm{BC2}}$ means that upon integration by parts, the boundary terms vanish and we are left with
\begin{align*}
\langle [\mathcal{A}(z)](u_0,u_1),(u_0,u_1)\rangle&=z-\int_{-1}^1 a(x)u_1''(x)\overline{u_0''(x)}dx+\int_{-1}^1\left[a(x)u_0''(x)+z^{\nu-1}b(x)u_1''(x)\right]\overline{u_1''(x)}dx\notag\\
&=z+2i\mathrm{Im}\left[\int_{-1}^1 a(x)u_0''(x)\overline{u_1''(x)}dx\right]+z^{\nu-1}\int_{-1}^1b(x)|u_1''(x)|^2dx.
\end{align*}
Let $l=\langle u_1,u_1\rangle_b$ and $w=\langle u_0,u_1\rangle_a$, then $l\in\mathbb{R}_{\geq 0}$ and we can write
\begin{equation}
\label{param_NR_form}
\langle [\mathcal{A}(z)](u_0,u_1),(u_0,u_1)\rangle=z+2i\mathrm{Im}(w)+z^{\nu-1}l.
\end{equation}
Moreover, by H\"older's inequality (on $L^2(-1,1)$), we have
\begin{equation}
|w|\leq \sqrt{\max_{x\in[-1,1]}\frac{a(x)}{b(x)}}\sqrt{\langle u_0,u_0\rangle_a}\cdot l^{1/2}\leq\sqrt{\max_{x\in[-1,1]}\frac{a(x)}{b(x)}}\cdot l^{1/2},
\end{equation}
where the second inequality follows from the fact that $\sqrt{\langle u_0,u_0\rangle_a}\leq\|(u_0,u_1)\|=1$. For notational convenience, we set $\smash{C=2\sqrt{\max_{x\in[-1,1]}{a(x)}/{b(x)}}},$ so that $2|w|\leq Cl^{1/2}$.

Let $\epsilon>0$ and suppose first of all that $\mathrm{Re}(z)\geq \epsilon$. By \eqref{param_NR_form}, the real part of any point in $\mathcal{W}(\mathcal{A}(z))$ is given by $\mathrm{Re}(z)+r^{\nu-1}\cos\left((\nu-1)\theta\right)l,
$ for some $l>0$, where $z=re^{i\theta}$ and $(r,\theta)$ denote the usual polar coordinates. Since we restricted to $\nu\in(0,2)$, it follows that for $\mathrm{Re}(z)>0$ we must have $\cos\left((\nu-1)\theta\right)>0$. Hence it follows that $r_n(z)\geq\mathrm{Re}(z)\geq \epsilon>0.$

Now consider the case that $x=\mathrm{Re}(z)< \epsilon$ and suppose that \eqref{r_cond_gen_pseudo} holds and that $|(\nu-1)\theta|=\pi/2$ so that $\cos\left((\nu-1)\theta\right)=0$. Then the condition \eqref{r_cond_gen_pseudo} reduces to  $\sqrt{2}\epsilon\leq r|\cos(\theta)|$. But $|\mathrm{Re}(\hat z)|=r|\cos(\theta)|$ for any $\hat z\in\mathcal{W}(\mathcal{A}(z))$ since $\cos\left((\nu-1)\theta\right)=0$, and hence $r_n(z)\geq \epsilon>0.$

Finally, consider the case that $x=\mathrm{Re}(z)< \epsilon$ and suppose that \eqref{r_cond_gen_pseudo} holds and that $|(\nu-1)\theta|\neq\pi/2$ so that $\cos\left((\nu-1)\theta\right)\neq0$.
 Let
$$
\hat z= z+2i\mathrm{Im}(w)+z^{\nu-1}l \in\mathcal{W}(\mathcal{A}(z)),
$$
where we use the same notation as in \eqref{param_NR_form}. Suppose also that $|\mathrm{Re}(\hat z)|\leq \epsilon$ (otherwise we must have $|\hat z|\geq \epsilon$ and there is nothing to prove). Considering polar coordinates, we can re-write the equation for the real part of $\hat z$ as
\begin{equation}
\label{real_part_zhat}
r^{\nu-1}l=\frac{\mathrm{Re}(\hat z)-r\cos(\theta)}{\cos\left((\nu-1)\theta\right)},
\end{equation}
from which we obtain
$$
\mathrm{Im}(z^{\nu-1}l)=\tan\left((\nu-1)\theta\right)(\mathrm{Re}(\hat z)-r\cos(\theta)).
$$
Combining with $2|w|\leq Cl^{1/2}$, we obtain
\begin{equation}
\label{im_hatz_nearly}
\begin{split}
|\mathrm{Im}(\hat z)|&\geq \left|\mathrm{Im}(z)-r\cos(\theta)\tan\left((\nu-1)\theta\right)\right|-\epsilon\left|\tan\left((\nu-1)\theta\right)\right|-Cl^{1/2}\\
&=r\left|\sin(\theta)-\cos(\theta)\tan\left((\nu-1)\theta\right)\right|-\epsilon\left|\tan\left((\nu-1)\theta\right)\right|-Cl^{1/2}\\
&=r\frac{\left|\sin\left((2-\nu)\theta\right)\right|}{\left|\cos\left((\nu-1)\theta\right)\right|}-\epsilon\left|\tan\left((\nu-1)\theta\right)\right|-Cl^{1/2}
\end{split}
\end{equation}
Returning to \eqref{real_part_zhat}, we see that
$$
l=r^{1-\nu}\frac{\mathrm{Re}(\hat z)-\mathrm{Re}(z)}{\cos\left((\nu-1)\theta\right)}.
$$
If $\cos(\theta)\geq 0$, then we have that $\cos\left((\nu-1)\theta\right)=\left|\cos\left((\nu-1)\theta\right)\right|$ and $\mathrm{Re}(\hat z)-\mathrm{Re}(z)\leq \epsilon-r\cos(\theta)$. If $\cos(\theta)< 0$ then $\left|\mathrm{Re}(\hat z)-\mathrm{Re}(z)\right|\leq \epsilon-r\cos(\theta)$. It follows that
$$
l\leq r^{1-\nu}\frac{\epsilon-r\cos(\theta)}{\left|\cos\left((\nu-1)\theta\right)\right|}.
$$
Substituting this into \eqref{im_hatz_nearly}, we see that
\begin{equation}
\label{im_hatz}
|\mathrm{Im}(\hat z)|\geq r\frac{\left|\sin\left((2-\nu)\theta\right)\right|}{\left|\cos\left((\nu-1)\theta\right)\right|}-\epsilon\left|\tan\left((\nu-1)\theta\right)\right|-Cr^{1-\nu/2}\sqrt{\frac{\epsilon/r-\cos(\theta)}{\left|\cos\left((\nu-1)\theta\right)\right|}}.
\end{equation}
Suppose for a contradiction that $|\mathrm{Im}(\hat z)|<\epsilon$, then \eqref{im_hatz} implies that
\begin{equation}
\label{im_hatz2}
r\frac{\left|\sin\left((2-\nu)\theta\right)\right|}{\left|\cos\left((\nu-1)\theta\right)\right|}-\epsilon\left|\tan\left((\nu-1)\theta\right)\right|-Cr^{1-\nu/2}\sqrt{\frac{\epsilon/r-\cos(\theta)}{\left|\cos\left((\nu-1)\theta\right)\right|}}<\epsilon.
\end{equation}
Rearranging \eqref{im_hatz2} leads to the condition
\begin{equation}
\label{im_hatz3}
r^{\nu/2}<C\sqrt{\frac{\left(\epsilon/r-\cos(\theta)\right)\left|\cos\left((\nu-1)\theta\right)\right|}{\sin^2\left((2-\nu)\theta\right)}}+\frac{r^{\nu/2-1}\epsilon\left(1+\left|\tan\left((\nu-1)\theta\right)\right|\right)\left|\cos\left((\nu-1)\theta\right)\right|}{\left|\sin\left((2-\nu)\theta\right)\right|}
\end{equation}
Finally we use
\begin{align*}
\left(1+\left|\tan\left((\nu-1)\theta\right)\right|\right)\left|\cos\left((\nu-1)\theta\right)\right|&=\left|\cos\left((\nu-1)\theta\right)\right|+\left|\sin\left((\nu-1)\theta\right)\right|\leq \sqrt{2}.
\end{align*}
Using this to bound the second term on the right-hand side of \eqref{im_hatz3}, we see that \eqref{r_cond_gen_pseudo} is violated, giving the required contradiction. This completes the proof of the proposition.
\end{proof}

We are now ready to bound the inverse of $T(z)$ via the following theorem.

\begin{theorem}
\label{thm:RES_BOUND}
Let $z=x+iy=re^{i\theta}\in \mathbb{C}\backslash \mathbb{R}_{\geq 0}$ and $\epsilon>0$, and suppose that either one of the conditions of Proposition \ref{prop:polar_curve} holds. Then $T(z)^{-1}$ exists and satisfies
\begin{equation}
\label{T_bd_final}
\sqrt{\langle T(z)^{-1}v,T(z)^{-1}v\rangle_a+|z|^2\|T(z)^{-1}v\|^2_{L^2_{\tilde\rho}(-1,1)}}\leq \epsilon^{-1}\|v\|_{L^2_{\tilde\rho}(-1,1)}.
\end{equation}
In particular, there exists a computable constant $c$ such that $\|T(z)^{-1}\|\leq {c\epsilon^{-1}}/({1+|z|}),$ where $\|\cdot\|$ denotes the operator norm on the space of bounded linear maps on $L^2_{\tilde\rho}(-1,1)$. Moreover, $T(z)^{-1}$ is bounded by $\epsilon^{-1}$ when viewed as a map from $L^2_{\tilde\rho}(-1,1)$ to $\mathcal{H}^2_{\mathrm{BC1}}$.
\end{theorem}

\begin{proof}
Suppose that the stated conditions hold. By Proposition \ref{prop:polar_curve}, $\mathcal{A}(z)^{-1}$ exists and satisfies
$$
\|\mathcal{A}(z)^{-1}\|\leq \epsilon^{-1}.
$$
To apply Proposition \ref{prop:quasi_lin}, we must show that $a+z^{\nu} b$ is non-vanishing on $[-1,1]$. If $a(x)+z^{\nu} b(x)=0$, then $\theta=\pm \pi/\nu$ and $\mathrm{Re}(z)\leq 0$. Hence by our assumptions on $z$, \eqref{r_cond_gen_pseudo} must hold. It follows that
\begin{equation}
\label{hgb}
r^{\nu/2}\geq 2\sqrt{\max_{x\in[-1,1]}\frac{a(x)}{b(x)}}\sqrt{\frac{\left|\cos(\theta)\right|\left|\cos((\nu-1)\theta)\right|}{\sin^2((2-\nu)\theta)}}.
\end{equation}
But if $\theta=\pm \pi/\nu$, then $\left|\cos((\nu-1)\theta)\right|=|\cos(\pi/\nu)|$ and $\sin^2((2-\nu)\theta)=4\sin^2(\pi/\nu)\cos^2(\pi/\nu)$. Using \eqref{hgb}, we see that
$$
r^{\nu/2}\geq \frac{1}{|\sin(\pi/\nu)|}\sqrt{\max_{x\in[-1,1]}\frac{a(x)}{b(x)}}>\sqrt{\max_{x\in[-1,1]}\frac{a(x)}{b(x)}},
$$
which contradicts $a(x)+z^{\nu} b(x)=0$. It follows that the conditions of Proposition \ref{prop:quasi_lin} hold. (Note that $|\sin(\pi/\nu)|\neq 0$ since we restricted to $|(2-\nu)\theta|<\pi$.)

By Proposition \ref{prop:quasi_lin}, $T(z)$ is invertible and satisfies
\begin{equation*}
[\mathcal{A}(z)]^{-1}\begin{pmatrix}
0, v
\end{pmatrix}=\begin{pmatrix}
[T(z)]^{-1}v,z[T(z)]^{-1}v
\end{pmatrix},\quad \forall v\in L^2_{{\tilde\rho}}(-1,1).
\end{equation*}
It follows that
$$
\sqrt{\langle T(z)^{-1}v,T(z)^{-1}v\rangle_a+|z|^2\|T(z)^{-1}v\|^2_{L^2_{\tilde\rho}(-1,1)}}=\left\|[\mathcal{A}(z)]^{-1}\begin{pmatrix}
0, v
\end{pmatrix}\right\|\leq \epsilon^{-1}\|(0,v)\|.
$$
The bound \eqref{T_bd_final} and conclusion now follow, since $\langle T(z)^{-1}v,T(z)^{-1}v\rangle_{\tilde\rho}$ is bounded above by a multiple of $\langle T(z)^{-1}v,T(z)^{-1}v\rangle_a$ by using the boundary conditions, an application of Poincaré's inequality and integration by parts. The constant $c$ can also be computed if desired.
\end{proof}

By a simple monotonicity argument, if $\mathrm{Re}(z)\leq \epsilon$, then for $0<|(2-\nu)\theta|<\pi$ and $\epsilon\geq0$, there exists $r_{\nu}^*(\theta,\epsilon)$ such that if $r\geq r_{\nu}^*(\theta,\epsilon)$, then \eqref{r_cond_gen_pseudo} holds. By combining with the case of $\mathrm{Re}(z)\geq \epsilon$, this allows us to compute bounding curves for the conditions of Theorem \ref{thm:RES_BOUND}. These curves can be thought of as generalised pseudospectral contours for $\mathcal{A}(z)$ and hence the fractional pencil $T(z)$. For simplicity, we now analyse $r_{\nu}^*(\theta,0)$ which is given by
\begin{equation}
\label{r_cond_gen_pseudo0}
r_{\nu}^*(\theta,0)=\left[4\max_{x\in[-1,1]}\frac{a(x)}{b(x)}\frac{\left|\cos(\theta)\right|\left|\cos((\nu-1)\theta)\right|}{\sin^2((2-\nu)\theta)}\right]^{1/\nu},\quad |\theta|\geq \pi/2.
\end{equation}
There are three qualitative cases, all of which are shown in Figure \ref{fig:pseudospectra}:

\vspace{2mm}

\noindent{}\textbf{Case 1: $\nu\in(0,1)$:} In this case,
$$
\lim_{|\theta|\uparrow \pi/(2-\nu)}r_{\nu}^*(\theta,0)=\infty
$$
and the curve is asymptotic to the line segments with $|\theta|=\pi/(2-\nu)$.

\noindent{}\textbf{Case 2: $\nu=1$:} In this case, \eqref{r_cond_gen_pseudo0} reduces to the parabola
$$
\mathrm{Im}(z)= 2\sqrt{\max_{x\in[-1,1]}\frac{a(x)}{b(x)}}\times\sqrt{-\mathrm{Re}(z)}.
$$
In fact, in this case $-\mathcal{A}(0)$ generates a strongly continuous semigroup of contractions \cite{liu1998spectrum,zhang2011spectrum}.

\noindent{}\textbf{Case 3: $\nu\in(1,2)$:} In this case, the curve for $r_{\nu}^*(\theta,\epsilon)$ is bounded as $\theta$ varies. Moreover, for $\nu \geq 3/2$,
$$
r_{\nu}^*(\pm\pi/(2(\nu-1)),0)=0,
$$
and as $\nu\uparrow 2$, the curve degenerates to the negative real axis.

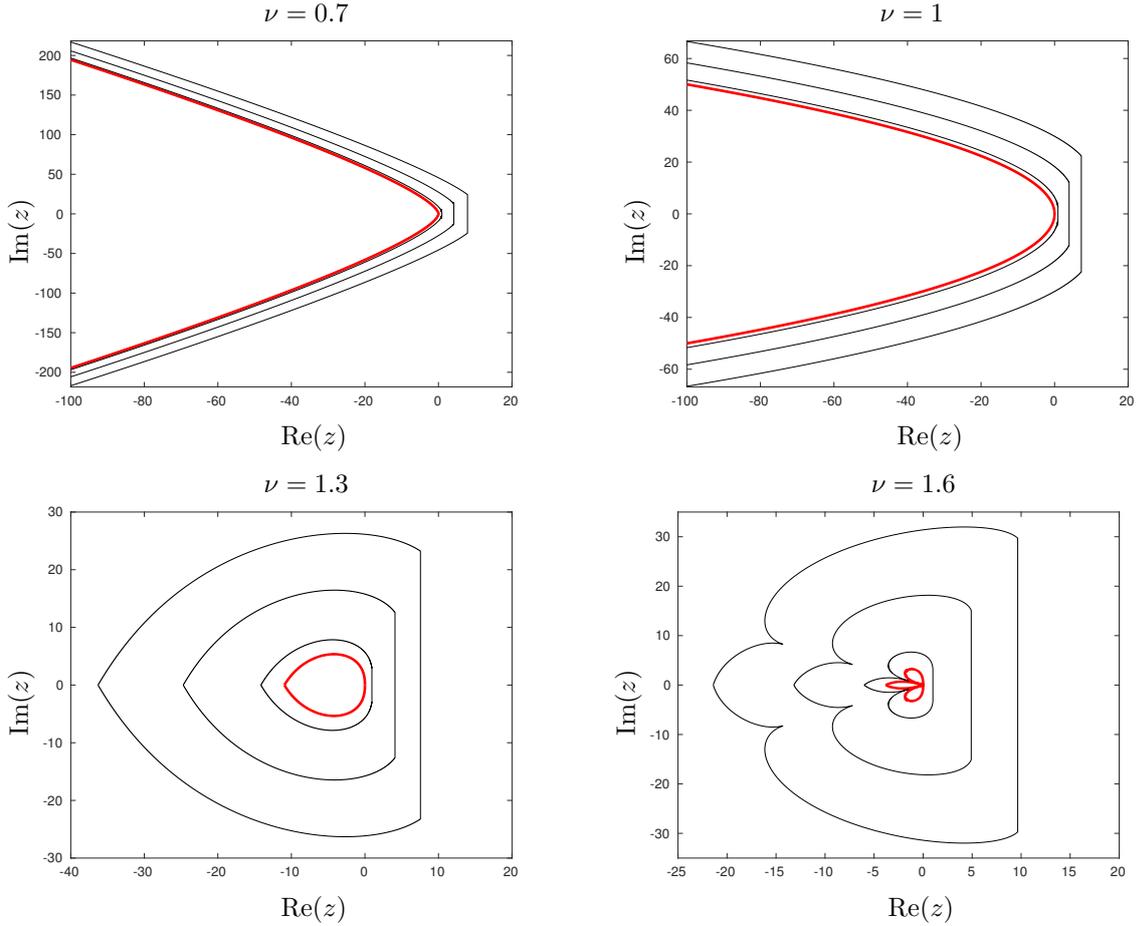
\begin{figure}[t!]
  \centering
  \begin{minipage}[b]{0.48\textwidth}
    \begin{overpic}[width=\textwidth]{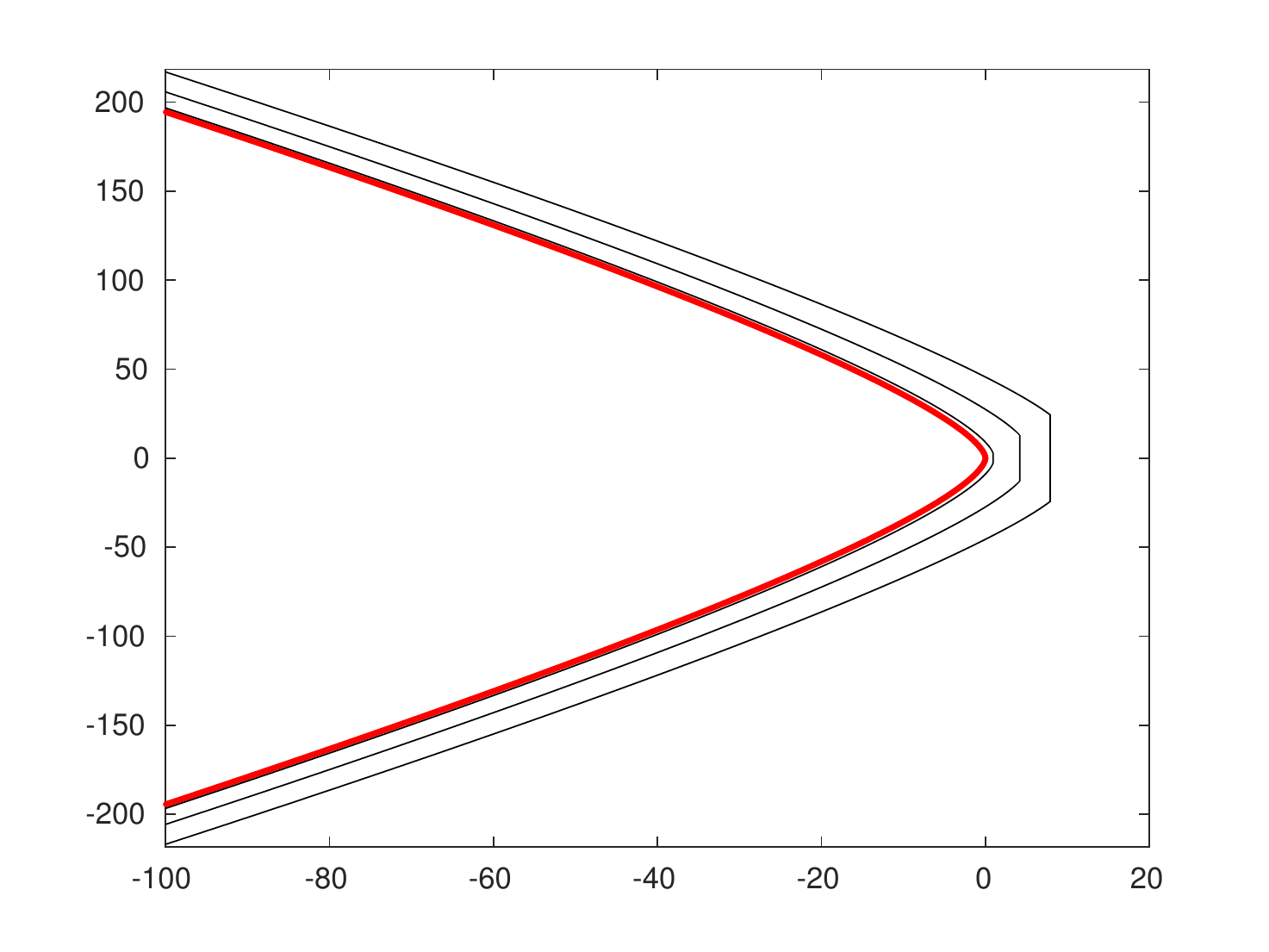}
     \put (47,73) {$\displaystyle \nu=0.7$}
		\put (50,-2) {$\displaystyle \mathrm{Re}(z)$}
		\put (2,30) {\rotatebox{90}{$\displaystyle \mathrm{Im}(z)$}}
     \end{overpic}
  \end{minipage}
  \hfill
  \begin{minipage}[b]{0.48\textwidth}
    \begin{overpic}[width=\textwidth]{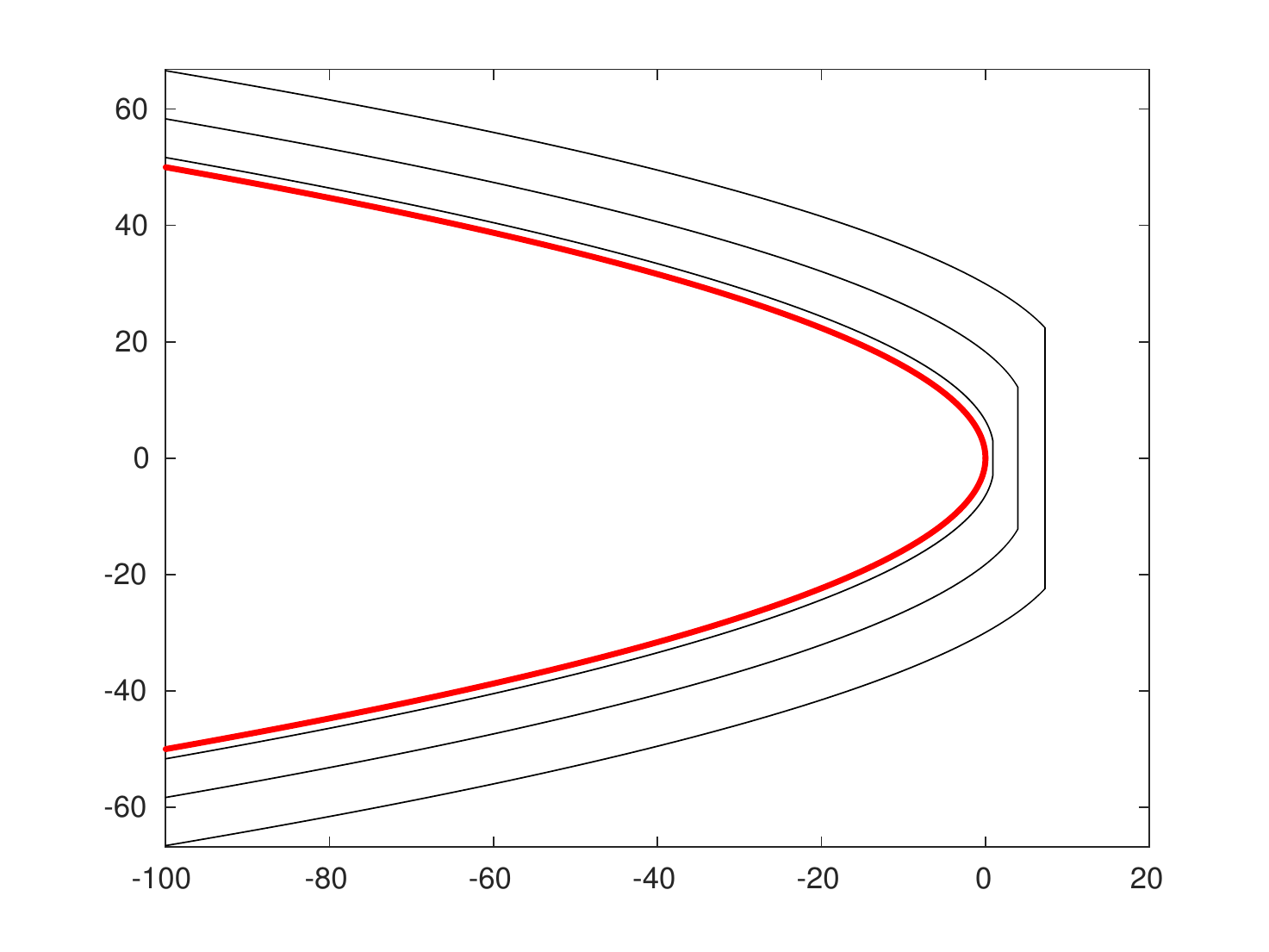}
    \put (47,73) {$\displaystyle \nu=1$}
		\put (50,-2) {$\displaystyle \mathrm{Re}(z)$}
		\put (2,30) {\rotatebox{90}{$\displaystyle \mathrm{Im}(z)$}}
     \end{overpic}
  \end{minipage}\\ \vspace{6mm}
	\begin{minipage}[b]{0.48\textwidth}
    \begin{overpic}[width=\textwidth]{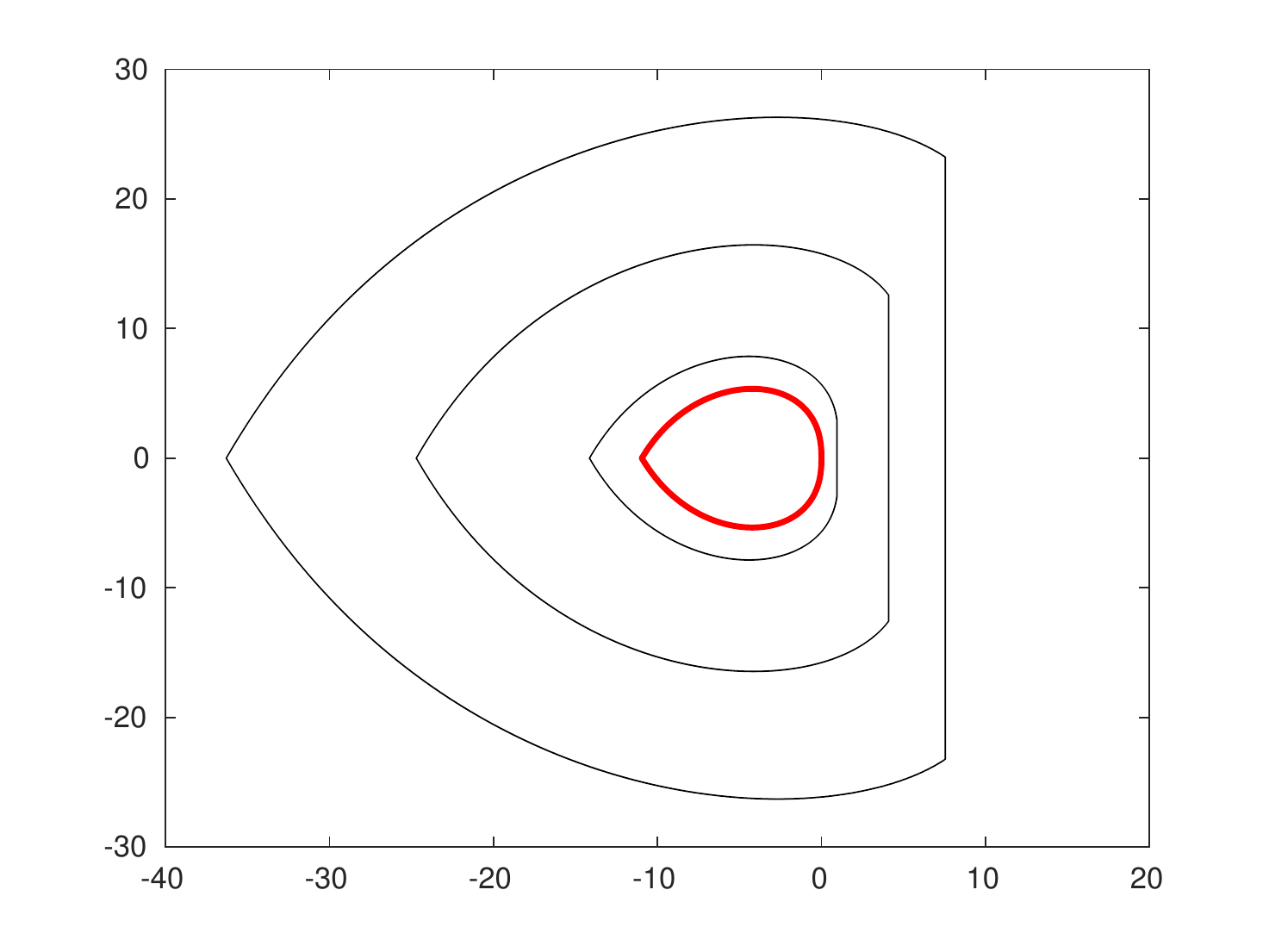}
     \put (47,73) {$\displaystyle \nu=1.3$}
		\put (50,-2) {$\displaystyle \mathrm{Re}(z)$}
		\put (2,30) {\rotatebox{90}{$\displaystyle \mathrm{Im}(z)$}}
     \end{overpic}
  \end{minipage}
  \hfill
  \begin{minipage}[b]{0.48\textwidth}
    \begin{overpic}[width=\textwidth]{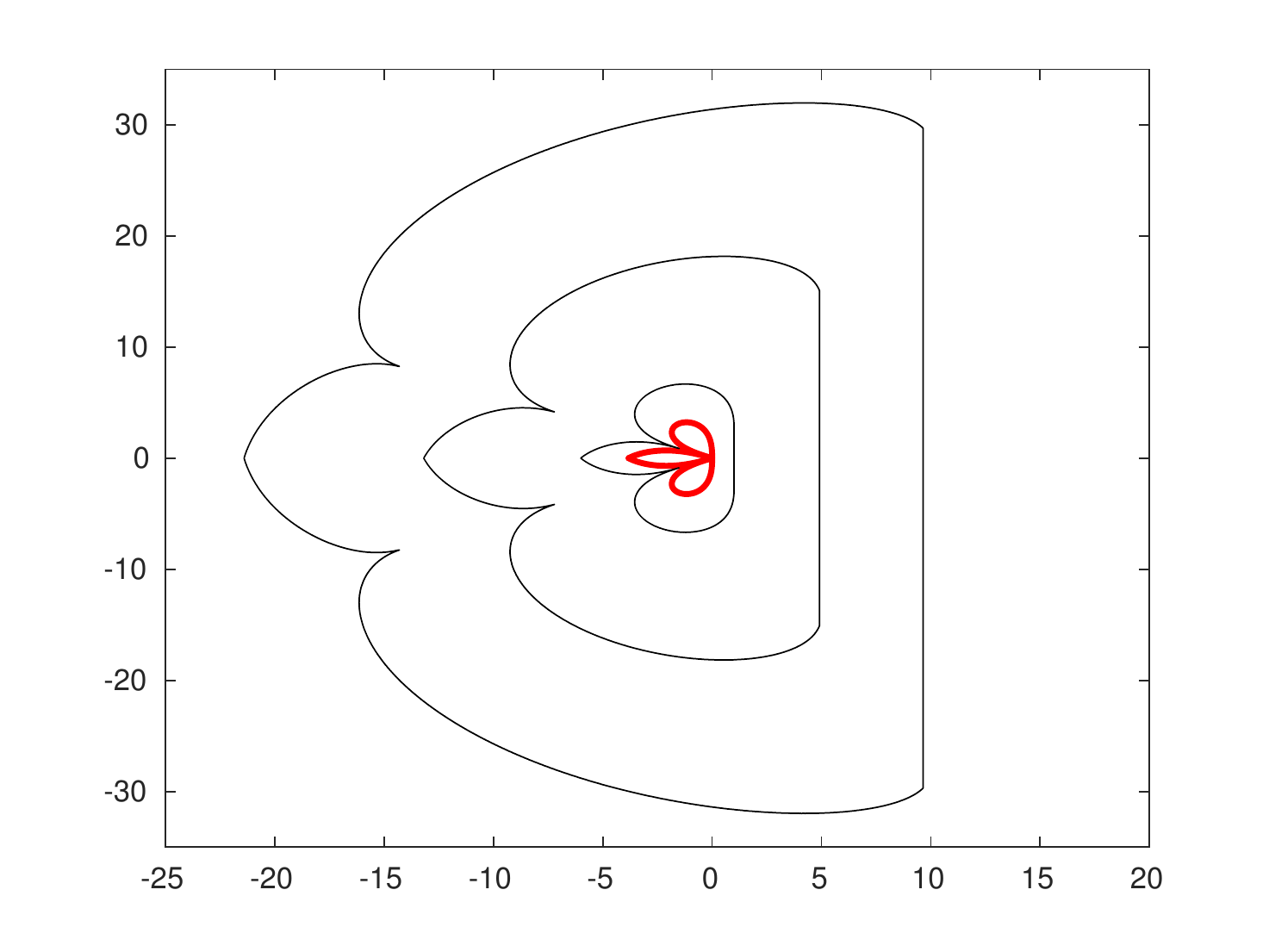}
    \put (47,73) {$\displaystyle \nu=1.6$}
     \put (50,-2) {$\displaystyle \mathrm{Re}(z)$}
		\put (2,30) {\rotatebox{90}{$\displaystyle \mathrm{Im}(z)$}}
     \end{overpic}
  \end{minipage}
  \caption{Examples of the curves $r_{\nu}^*(\theta,\epsilon)$ that bound $\|T(z)^{-1}\|$ for $2\sqrt{\max_{x\in[-1,1]}\frac{a(x)}{b(x)}}=5$. The red lines correspond to $\epsilon=0$ (which bounds the generalised spectrum) and the black lines correspond to $\epsilon=1,5$ and $10$. For $\nu\in(0,1)$, the curves are asymptotic to straight lines, for $\nu=1$ they are parabolic and for $\nu\in(1,2)$ they are bounded.}
\label{fig:pseudospectra}
\end{figure}

\subsection{Computing the resolvent and numerical convergence}
\label{sec:solve_desc}

To compute $y(x,t)$ via the method in \S \ref{cont_descrip}, we use the ultraspherical spectral method of Olver and Townsend \cite{Olver_SIAM_Rev} to apply $T(z)^{-1}$ and compute $\hat y(x,z)$. Explicitly, for each quadrature point $z_j$, we solve the corresponding linear system
$$
T(z_j)\hat y(x,z_j)=K(z_j).
$$
The method of \cite{Olver_SIAM_Rev} expands the solution of the relevant linear system in terms of Chebyshev polynomials of the first kind. To represent differential operators with variable coefficients, different bases of ultraspherical polynomials are used for the range and domain of the operator. This allows an almost banded (where filled in rows correspond to the boundary conditions) matrix representation of the operator that can be solved stably in optimal linear complexity $\mathcal{O}(m^2n)$ for truncation parameter $n$ and bandwidth $m$. We can then compute the residual of the computed solution (with respect to the full infinite matrix) and combine it with the results of Theorem \ref{thm:RES_BOUND} to bound the error of the approximation of $\hat y(x,z_j)$. We adaptively increase $n$ until this bound is below the required error tolerance $\eta$ which appears in \S \ref{cont_descrip}.\footnote{It is also possible to gain verified error control using interval arithmetic. Error control is obtained by using sparse approximations and suitable rectangular truncations corresponding to bounding tails of expansions of coefficients and solutions in ultraspherical polynomials. See \cite{brehard2018validated}, for an interval arithmetic implementation of the ultraspherical spectral method. In the current paper, we have not used interval arithmetic.} In all of the examples of this paper, the required $n$ is at most a few hundred. In particular, upon chasing the explicit constants in \eqref{analytic_bound}, the following is immediate.

\begin{theorem}[Computation of solution with error control]\label{thrown_in_theorem}
Given a fixed class of boundary conditions (as in \S \ref{sec:phys_setup}), there exists an algorithm $\Gamma$ such that the following holds. For any PDE parameters $a,b,\tilde \rho$, $\nu\in(0,2)$, forcing $F$ (whose Laplace transform can be evaluated to arbitrary accuracy), $\epsilon>0$, $t>0$, and sufficiently regular $y_0,y_1$,
$$
\left\|\Gamma(a,b,\tilde \rho,\nu,F,y_0,y_1)-y(\cdot,t)\right\|_{\mathcal{H}^2_{\mathrm{BC1}}}\leq \epsilon,
$$
where $y$ is the solution of \eqref{NL_gen_form} with $y(\cdot,t)=y_0$ and $y_t(\cdot,t)=y_1$.
\end{theorem}

\begin{remark}
Here `sufficiently regular $y_0,y_1$' means that a Chebyshev series for $K(z)$ can be computed with error control. For example, if $\nu>1$ and we use the Caputo definition of the fractional derivative, then $K(z)$ involves a term of the form $(by_0'')''/\tilde \rho$. Hence, $y_0$ must have a fourth derivative whose Chebyshev series can be approximated.
\end{remark}

As a concrete example to demonstrate convergence (we do not analyse this model physically), we consider the case of
$$
a=\cosh(x),\quad b=\sin(\pi x)+2,\quad {\tilde\rho}=\tanh(x)+2,
$$
for various $\nu$. We consider the initial conditions
$$
y(x,0)=\sin(2\pi x)(1-x^2)(1-x),\quad \pderiv{y}{t}(x,0)=0,
$$
where $x_0=-1$ is a clamped end and $x_0=1$ is a simply supported end, and the forcing
$$
F(x,t)=\cos(20t)\sin(\pi x).
$$
Figure \ref{fig:sparsity} shows a typical sparsity pattern for the matrix representing the operator $T(z)$.

We apply Algorithm \ref{alg:spec_meas} and Algorithm \ref{alg:spec_meas2} with $t_0=1$ and $t_1=10$, where $\eta$ is chosen so that its contribution to the error of the computed solution is negligible compared to the quadrature error. For Algorithm 1, we set $\beta=2$, chose $\sigma=\beta/t_1$ and then selected the minimum $\delta$ so that the generalised spectrum of $T(z)$ lies in the exterior of $S_{\delta,\sigma}$. Such a $\delta$ is easily computed using the formula \eqref{r_cond_gen_pseudo0} for $r_{\nu}^*(\theta,0)$. For Algorithm 2, if $\nu\geq1$, we set $\sigma=0$ and $\delta=1/(4\max_{x\in[-1,1]}\frac{a(x)}{b(x)})$. If $\nu<1$, we set $\sigma=0$ and then choose $\delta$ so that the corresponding parabola intersects the curve defined by $r_{\nu}^*(\theta,0)$ at $\mathrm{Re}(z)=10^{-16}/t_0$. Strictly speaking, the parabola does not bound the possible location of singularities of $T(z)^{-1}$ analysed in \S \ref{sec:finding_spec}. However, in practice, this is of little consequence since we can deform the contour of integration for $\mathrm{Re}(z)\ll 0$ and the change in the computed solution is negligible (owing to the choice $\mathrm{Re}(z)=10^{-16}/t_0$).

\begin{figure}[t!]
  \centering
  \begin{minipage}[b]{0.49\textwidth}
    \begin{overpic}[width=\textwidth]{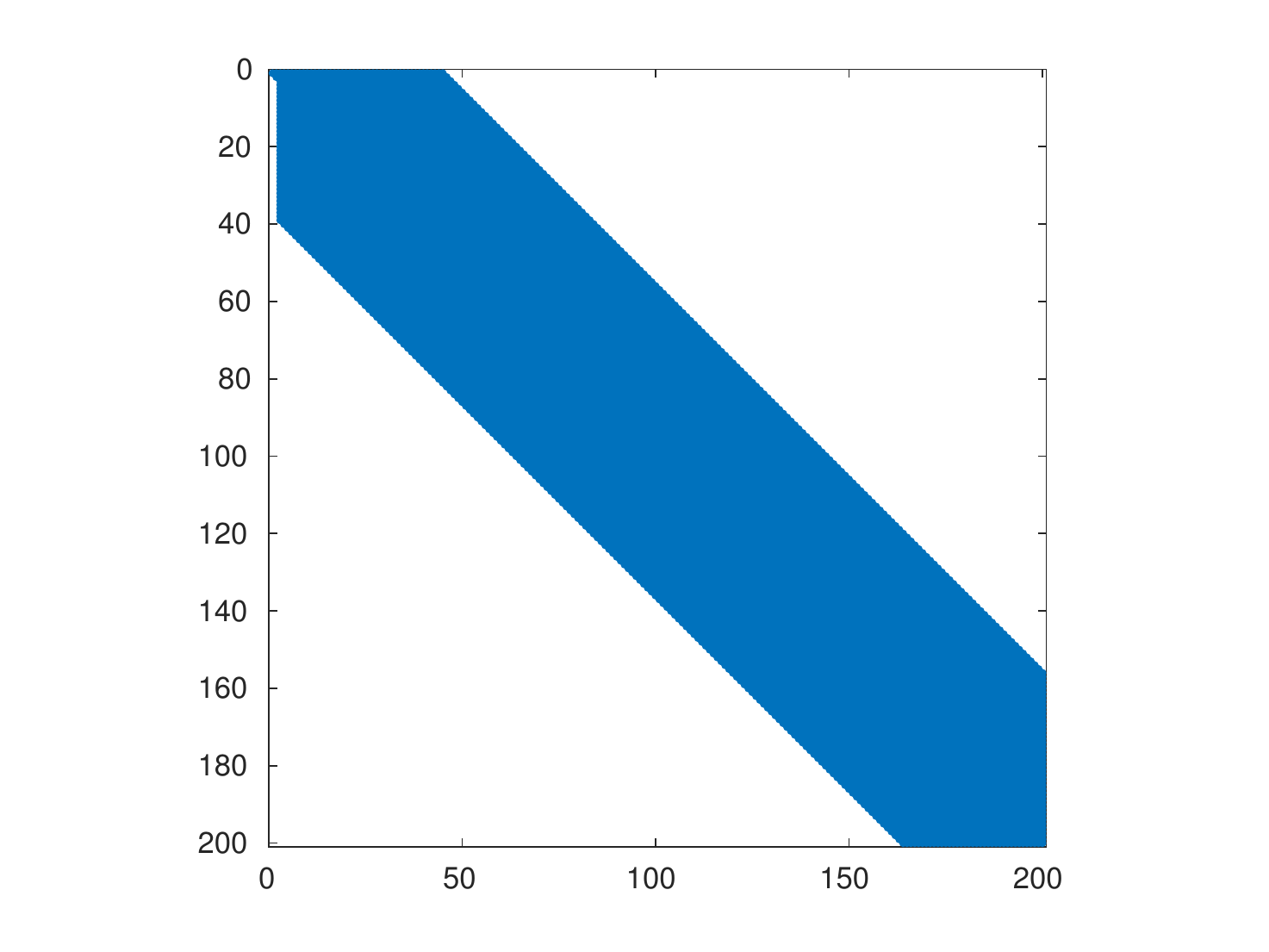}
     \end{overpic}
  \end{minipage}
  \hfill
  \begin{minipage}[b]{0.49\textwidth}
    \begin{overpic}[width=\textwidth]{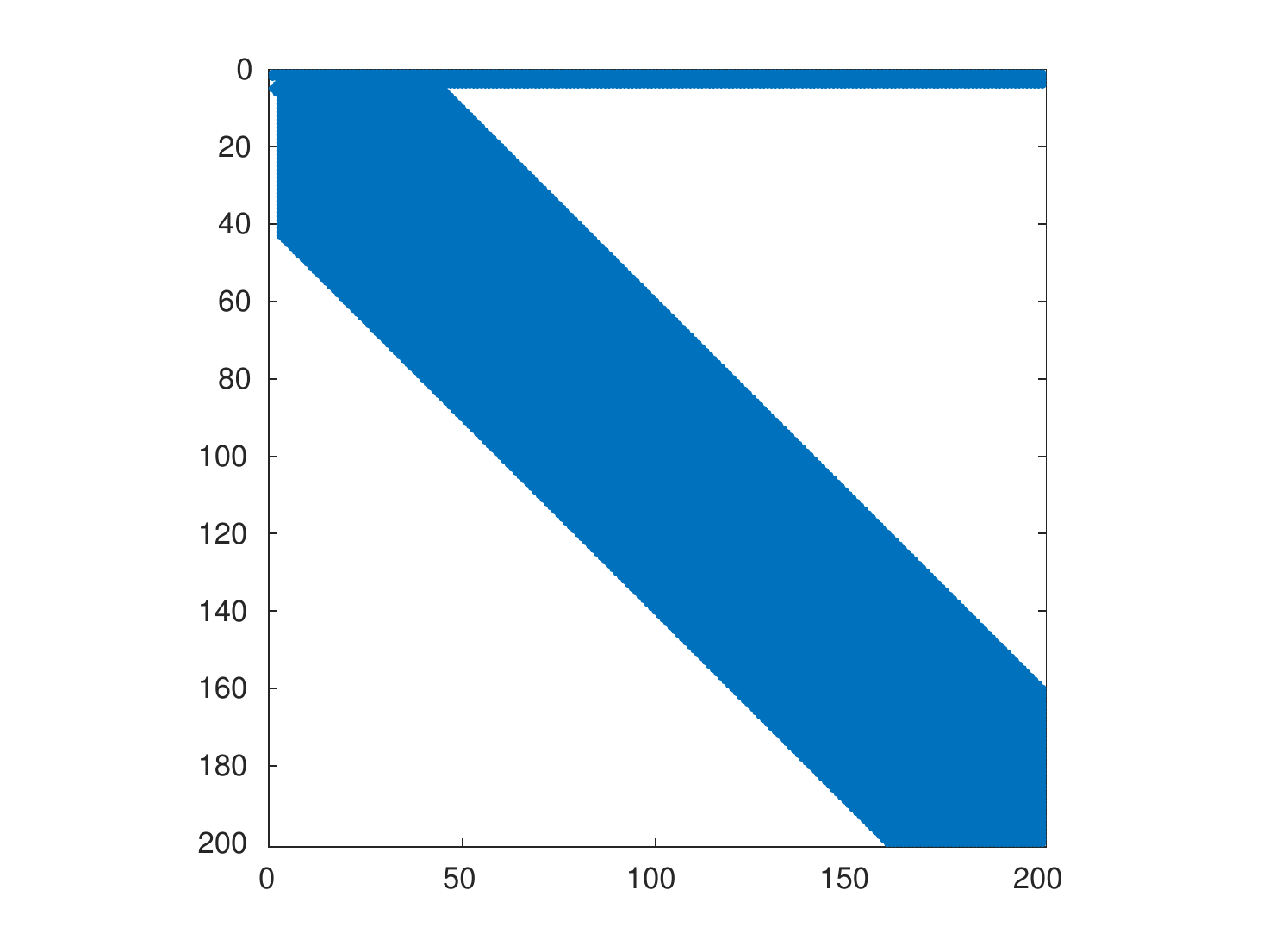}
		\end{overpic}
  \end{minipage}
  \caption{Sparsity pattern of the matrix representing $T(z)$ for the example in \S \ref{sec:solve_desc} (finite $200\times 200$ section of the infinite matrix shown). Left: Sparsity pattern before adding the boundary conditions. Right: Sparsity pattern of the almost banded matrix after adding the boundary conditions (which correspond to the filled in top four rows).}
\label{fig:sparsity}
\end{figure}

Figure \ref{fig:error} shows the quadrature error for a range of representative $\nu$, where the error is estimated by comparing to a larger $N$. We see that smaller $\nu<1$ favours the hyperbolic contour (Algorithm 1), whereas larger $\nu\uparrow 1$ favours the parabolic contour (Algorithm 2). This reflects the shape of the curve bounding the singularities of $T(z)^{-1}$, that was analysed in \S \ref{sec:finding_spec}. In general, the best choice of contour also depends on the desired accuracy. Note also that both quadrature rules are stable as $N$ increases. We found the hyperbolic contour to have a more consistent exponential convergence over different parameter choices (presumably due to the optimisation problem in Algorithm 2) and so use it in the numerical experiments of \S \ref{sec:phys_res}. Figure \ref{fig:examsol} shows corresponding computed solutions at different times.

\begin{figure}[t!]
  \centering
  \begin{minipage}[b]{0.49\textwidth}
    \begin{overpic}[width=\textwidth]{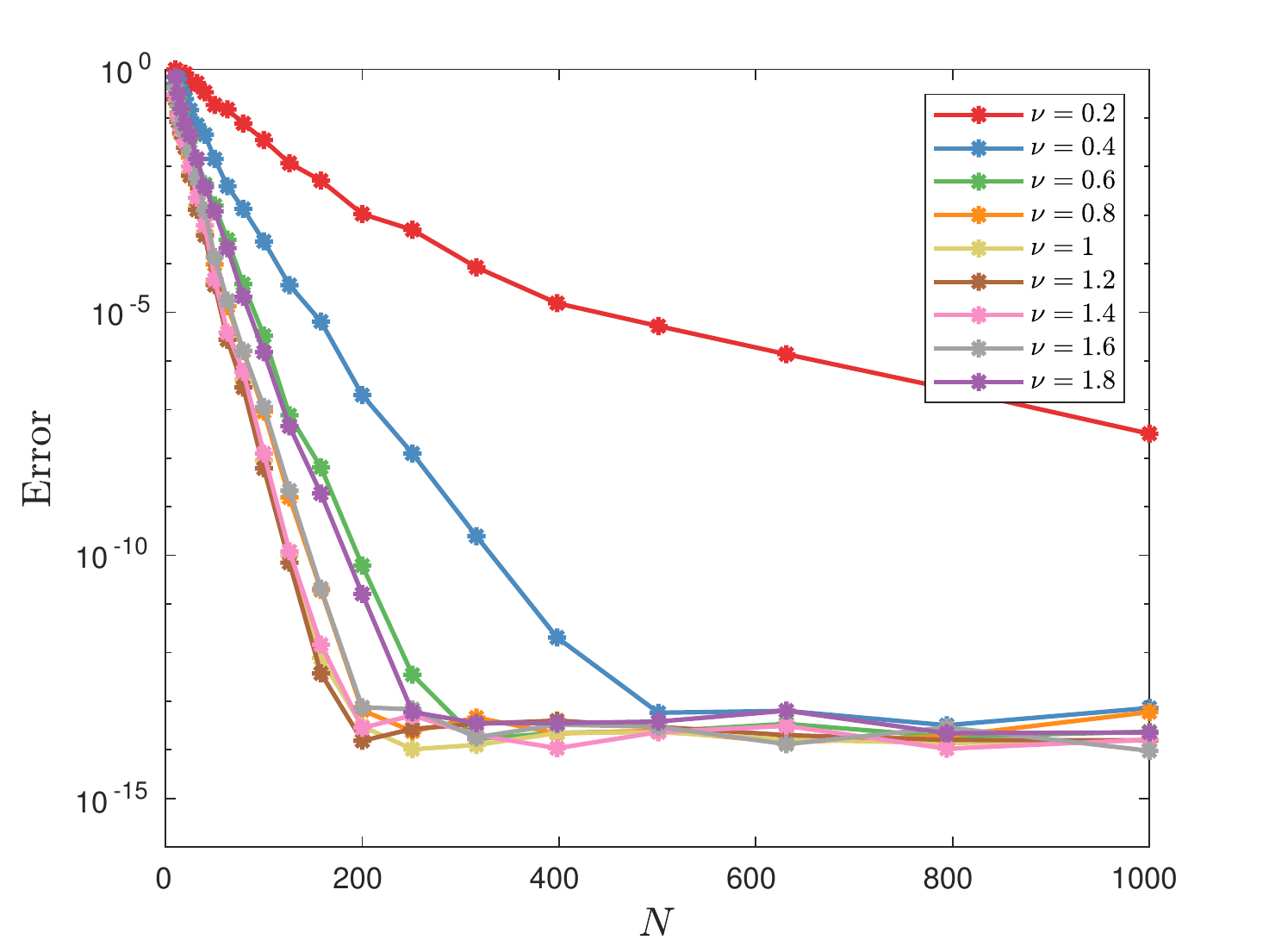}
		\put (37,73) {$\displaystyle \text{Algorithm 1}$}
     \end{overpic}
  \end{minipage}
  \hfill
  \begin{minipage}[b]{0.49\textwidth}
    \begin{overpic}[width=\textwidth]{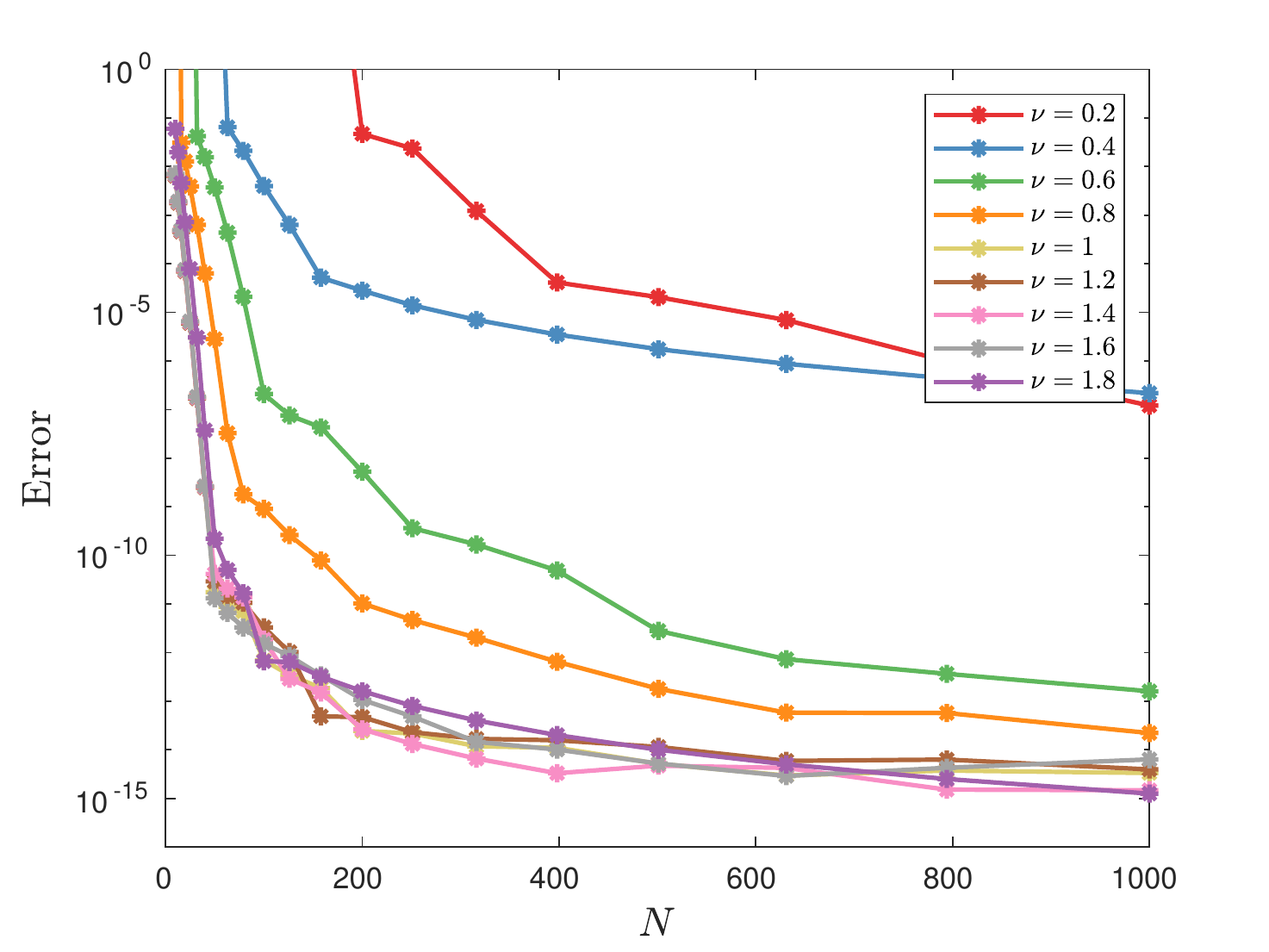}
		\put (37,73) {$\displaystyle \text{Algorithm 2}$}
		\end{overpic}
  \end{minipage}
  \caption{Maximum quadrature error over $t\in[1,10]$ for the $L^2$ norm of the solution.}
\label{fig:error}
\end{figure}

\begin{figure}[t!]
  \centering
  \begin{minipage}[b]{0.32\textwidth}
    \begin{overpic}[width=\textwidth]{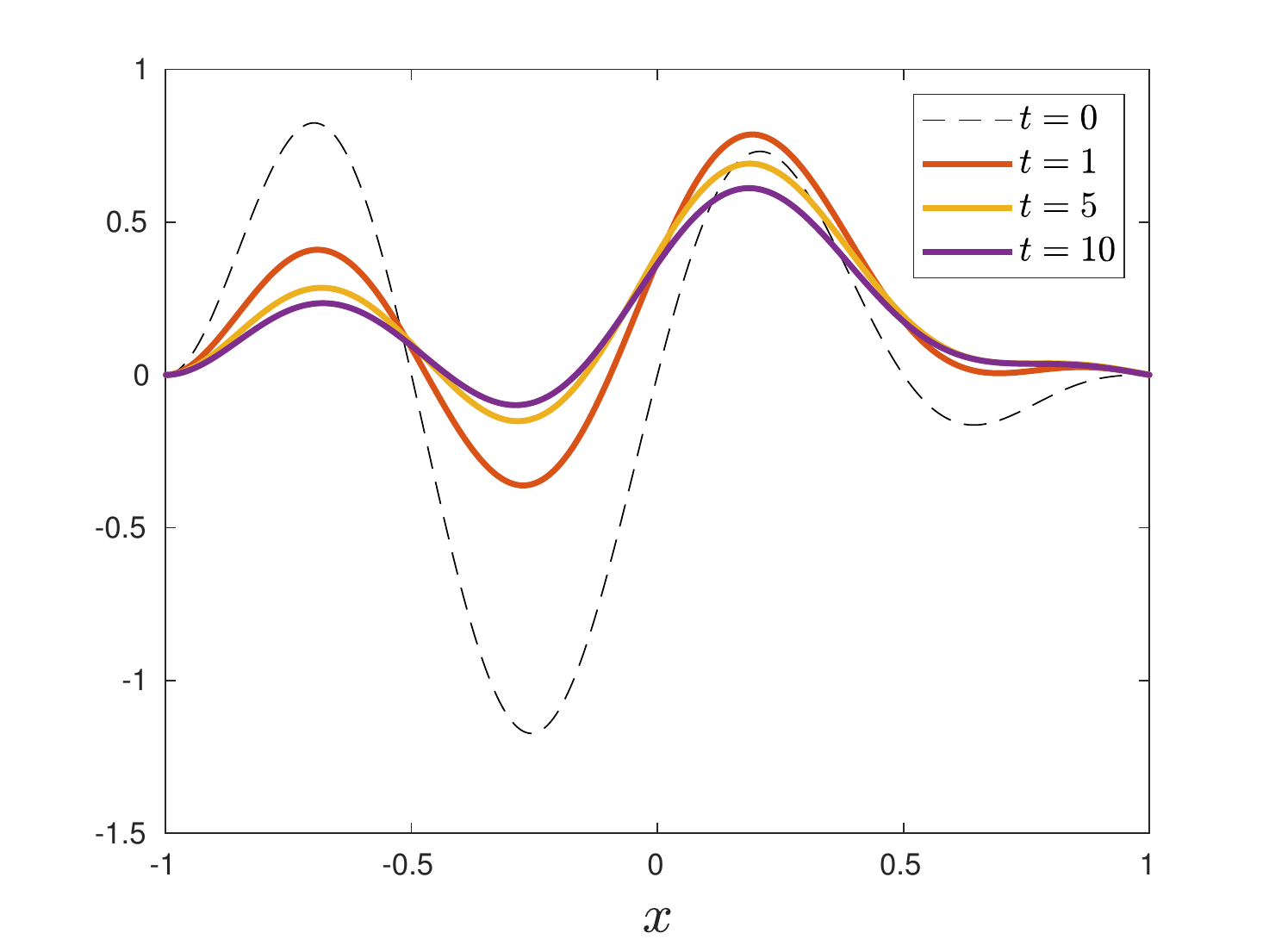}
		\put (42,73) {$\displaystyle \nu=0.4$}
     \end{overpic}
  \end{minipage}
  \hfill
  \begin{minipage}[b]{0.32\textwidth}
    \begin{overpic}[width=\textwidth]{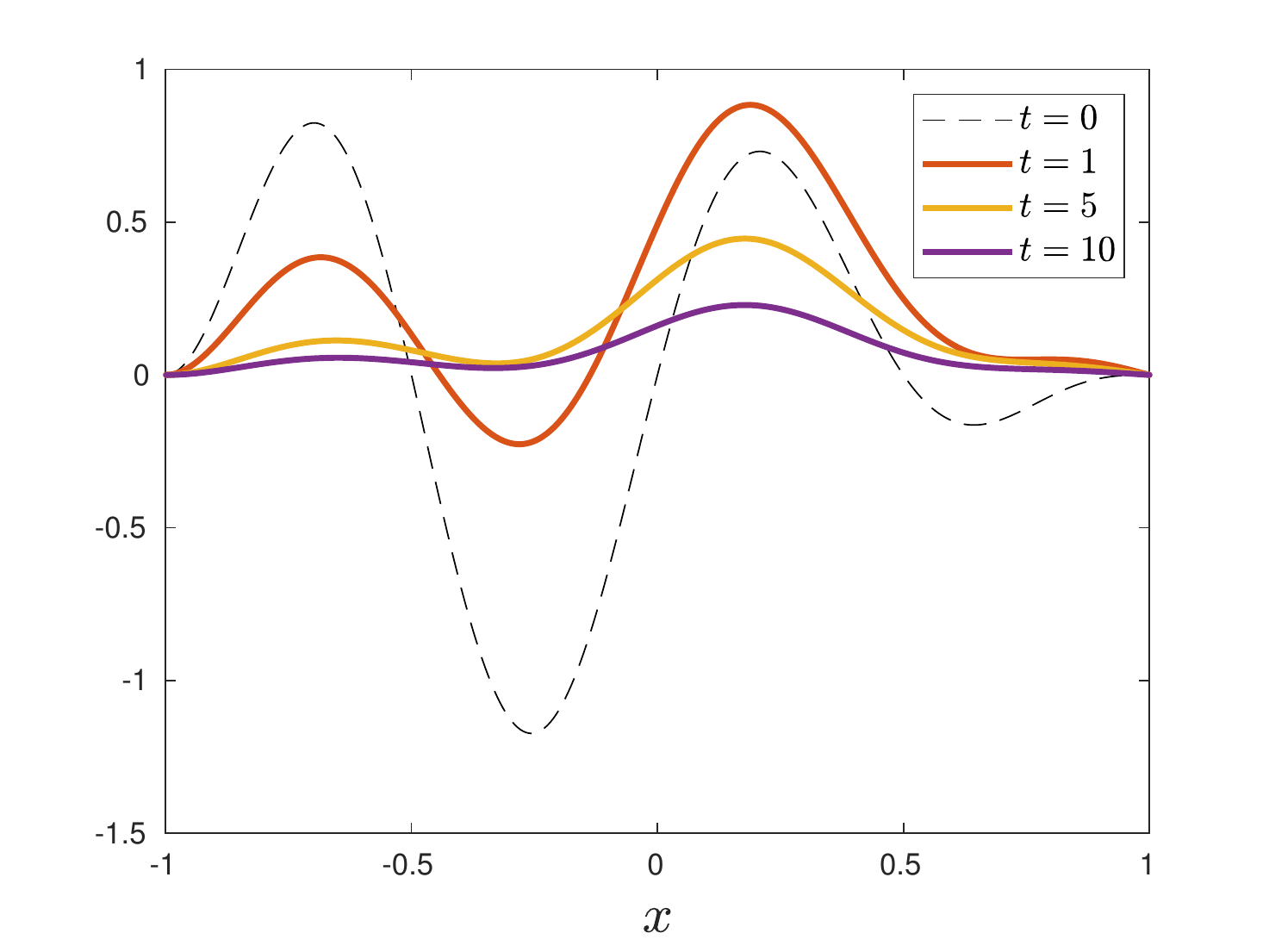}
		\put (42,73) {$\displaystyle \nu=0.8$}
		\end{overpic}
  \end{minipage}
	\hfill
  \begin{minipage}[b]{0.32\textwidth}
    \begin{overpic}[width=\textwidth]{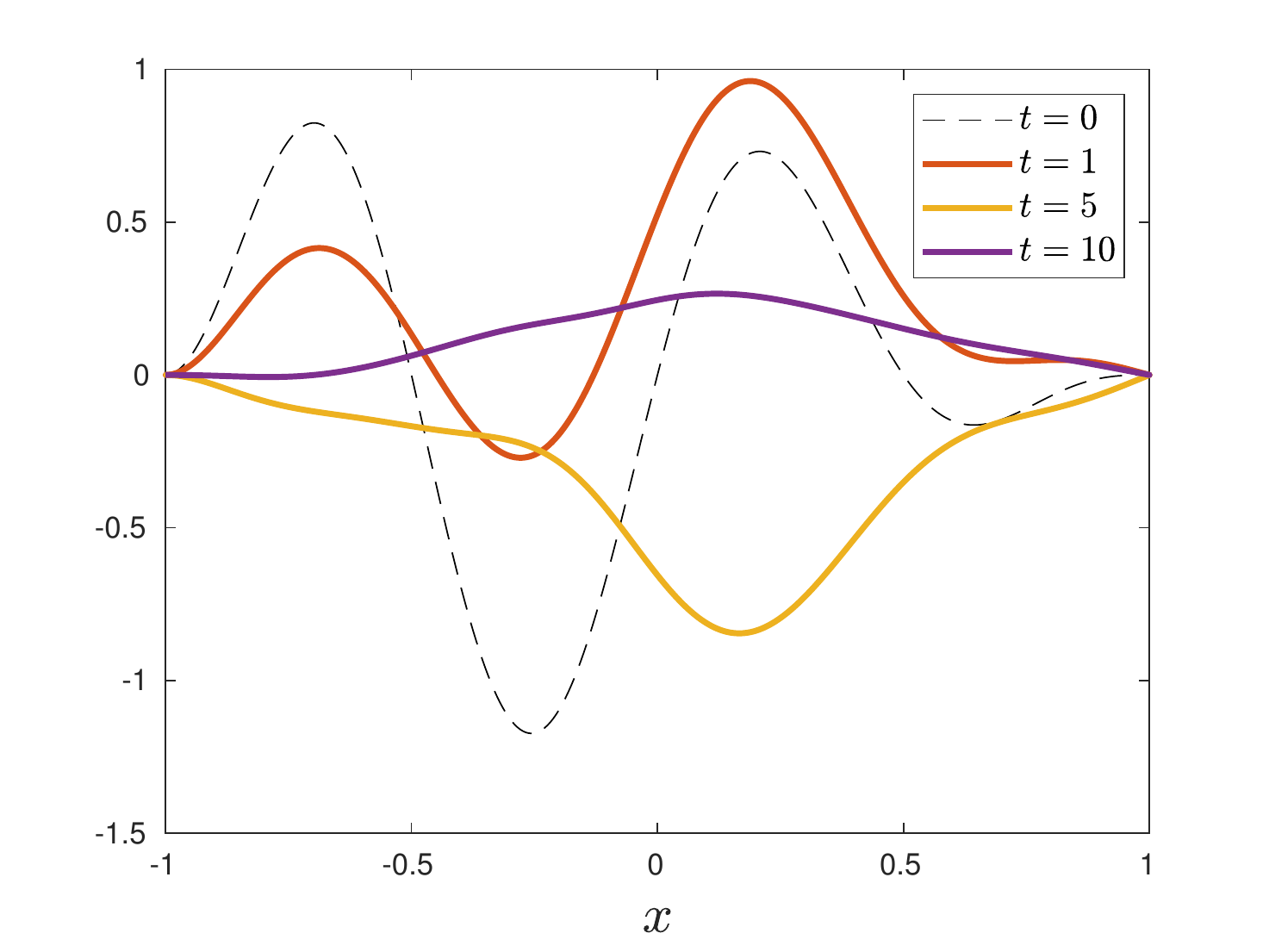}
		\put (42,73) {$\displaystyle \nu=1.6$}
		\end{overpic}
  \end{minipage}
  \caption{Computed solutions for the example in \S \ref{sec:solve_desc}.}
\label{fig:examsol}
\end{figure}

\section{Results}
\label{sec:phys_res}

In what follows, we follow the procedure outlined in \S \ref{sec:solve_desc} with an overall error bound (in non-dimensionalised parameters and with respect to the standard $L^2(-1,1)$ norm) on the computed solution of $10^{-8}$. In principle, it is possible to rigorously achieve such a bound by computing the explicit constants in relations such as \eqref{analytic_bound}. In practice, this is achieved by choosing $\eta$ sufficiently small so that the dominant error is the quadrature error, and then increasing $N$ adaptively until the estimated quadrature error is below $10^{-8}$. There are at least three factors that contribute to the efficiency of the overall method:
\begin{itemize}
	\item The computation of $T(z)^{-1}$ is rapid owing to the sparse spectral method used, which converges exponentially in the truncation parameter $n$ and has $\mathcal{O}(n)$ complexity.
	\item The quadrature rule converges rapidly (as well as being stable) with an error bounded by $\mathcal{O}(\exp(-cN/\log(N)))$ for $N$ quadrature points.
	\item The values of the Laplace transform (solving the linear systems at quadrature points) can be computed in parallel across different quadrature points and only need to be computed once, thereafter being reused for different times $t\in[t_0,t_1]$.
\end{itemize}

\subsection{Example 1}

Our first physical example considers a uniform viscoelastic beam of length $1\mathrm{m}$, simply supported at both ends, as described in \cite{Xu2020}. The relevant parameters and dimensions are:
\begin{align*}
    \rho A = 0.818 \hspace{1mm}\mathrm{kg}/\mathrm{m},\quad\nu = 0.64,\quad I=8.33\times 10^{-6}\hspace{1mm}\mathrm{m}^{4},\\E_{0}=5.04\times 10^{7}\hspace{1mm}\mathrm{Pa},\quad E_{1}=2.27\times 10^{5}\hspace{1mm}\mathrm{Pa},\quad F(x,t)=F_{0}\sin(\pi x/l)\sin(\omega t)\hspace{1mm}\mathrm{kg}/\mathrm{s}^2,
\end{align*}
where $F_{0}$ is dimensional. Here, the model in \cite{Xu2020} has a beam of length $l=1\mathrm{m}$ lying in the region $x\in[0,l]$, and the natural frequency of the beam is $\omega_{1}=\sqrt{E_{0}I/(\rho A)}$.

Since the formulation of our theory non-dimensionalises lengths on the semi-chord to a beam in the region $x\in[-1,1]$, we first present the fully non-dimensionalised form of the governing equation \eqref{eq:gov}. We denote non-dimensional parameters by a superscript $^*$, defined as: $\rho^{*}=A\rho/(\rho_0)$, $(x^*,y^*)=(x,y)/L$, $t^*=t/T$. Here, $\rho_{0}$ is the mass per unit area, $L$ is the semichord of the plate, and $T$ is the time-scale (taken as $T=1/f$ where $f$ is the frequency $1$Hz).
Under such scalings, the governing equation reduces to
$$
\rho^{*}\pderiv{^2y^*}{t^{*2}}+\pderiv{^{2}}{x^{*2}}\left[E_{0}^{*}I^{*}\pderiv{^{2}y^*}{x^{*2}}+E_{1}^*I^*\mathcal{D}^{\nu}_{\mathrm{C},t^*}\pderiv{^{2}y^*}{x^{*2}} \right]=F^*(x^*,t^*)
$$
Our test parameters reduce to
$$
    \rho^* = 1,\quad\nu = 0.64, \quad  E_{0}^*I^*=821.2, \quad E_{1}^*I^*=3.70, \quad F^*(x^*,t^*)=F_{0}^{*}\sin(\pi (x^*-1))\sin(\omega^* t^*),
$$
where we assume the beam has a width of $10\%$ of the length to calculate the mass per unit area, and due to linearity, we shall set $F_{0}^{*}=1$ without loss of generality.

First, we verify the numerical method and consider the case of an initially undeformed beam, with $y^{*}(x^*,0), y_{t^*}^{*}(x^*,0)=0$. In Figure \ref{fig:Case1Compare} we present the computed solution at $x^*=0.25$ for $0.1\leq t^*\leq1$ (left) and $t^*=5$ for $-1\leq x^*\leq 1$ (right), whilst comparing to the analytic solution given by
\begin{equation}
y^{*}(x^*,t^*)=F_{0}^{*}\times \mathrm{Im}\left[\frac{e^{i\omega^* t^*}\sin(\pi (x^{*}-1))}{\pi^{4}E_{0}^*I^*+\pi^{4}E_{1}^*I^*(i\omega^*)^{\nu}-\omega^{*2}}\right].
\end{equation}
In this case, as expected from the adaptive algorithm, the $L^2$ errors of the solution were bounded by the user-specified tolerance (in this case $10^{-8}$).

\begin{figure}[t!]
  \centering
  \begin{minipage}[b]{0.49\textwidth}
    \begin{overpic}[width=\textwidth]{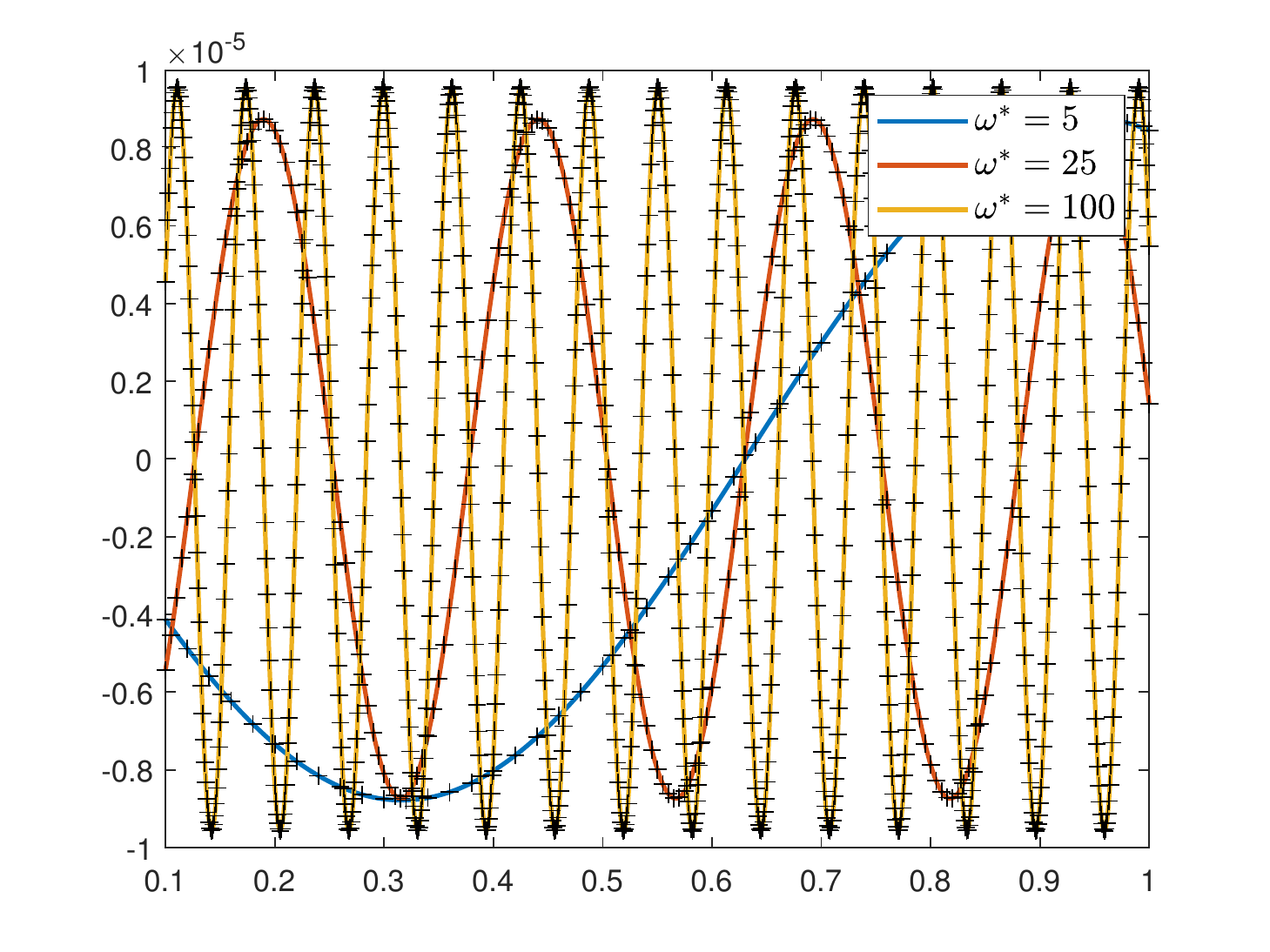}
		\put (50,-2) {$\displaystyle t^*$}
		\put (2,30) {\rotatebox{90}{$\displaystyle y^{*}(0.25,t^*)$}}
     \end{overpic}
  \end{minipage}
  \hfill
  \begin{minipage}[b]{0.49\textwidth}
    \begin{overpic}[width=\textwidth]{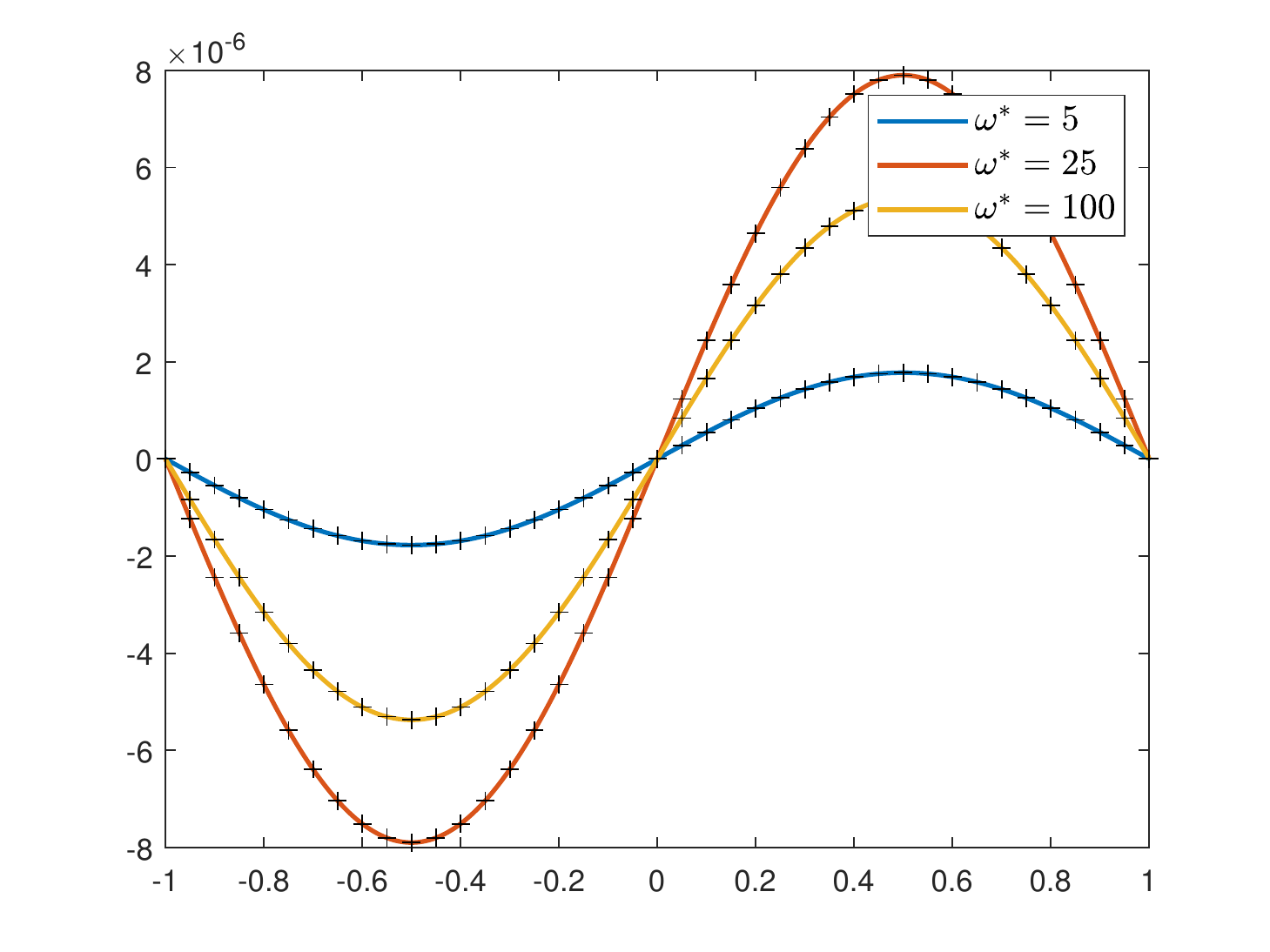}
		\put (50,-2) {$\displaystyle x^*$}
		\put (2,30) {\rotatebox{90}{$\displaystyle y^{*}(x^*,5)$}}
		\end{overpic}
  \end{minipage}
  
  \caption{Comparison of numeric (lines) and analytic (markers) solutions for $\omega^{*}=5,25,100$. Left: Evaluation at $x^{*}=0.25$ for different $t^*$. Right: Evaluation at $t^{*}=5$ for different $x^*$.}
\label{fig:Case1Compare}
\end{figure}

Whilst an analytic solution is available in the initially undeformed case, the same cannot be said for a more general forcing or more general initial conditions. Furthermore, for non-zero initial conditions, previous numerical approaches using the Riemann--Liouville definition of the fractional derivative are inaccurate as they do not account for the initial data properly. Our implementation with the Caputo derivative captures the required behaviour. We illustrate in Figure \ref{fig:Case1Initial} an example deformation for the same parameters as above, but now with an initial asymmetric deformation of
\begin{equation}
y^{*}(x^{*},0)=\sin^{2}(2\pi x^{*})(1+x^{*})(1-x^{*})^{2}
\end{equation}
and $y_{t^{*}}^{*}(x^{*},0)=0$, for fractional parameters $\nu=0.32$ and $\nu=0.64$. For this example we set the external forcing $F^{*}(x^{*},t^{*})=0$. As expected, the initial deformation generates more rapid and intense (greater amplitude) vibration in the plate with weaker damping ($\nu=0.32$), and these vibrations last for longer. The energy in the system is defined by
$$
E^{*}(t^{*})=\frac{1}{2}\int_{-1}^1 a(x^{*})|y_{x^{*}x^{*}}^{*}(x^{*},t^{*})|^2+\rho(x^{*})|y_{t^{*}}^{*}(x^{*},t^{*})|^2dx^{*}.
$$
To compute this quantity, we use the approximation
\begin{equation}
\label{bromwich_quad_example_22}
y_{t^{*}}^{*}(\cdot,t^{*})=q_{t^{*}}(t^{*})\approx \frac{h}{2\pi i}\sum_{j=-N}^Nz_je^{z_jt^{*}}\hat q(z_j)\gamma'(jh),\quad z_j=\gamma(jh),
\end{equation}
which is analogous to \eqref{bromwich_quad_example}. To compute $y_{x^{*}x^{*}}^{*}(x^{*},t^{*})$, we note that $y^{*}(\cdot,t^{*})$ is computed as a (finite) Chebyshev series and hence we differentiate term by term. This is justified by Theorem \ref{thm:RES_BOUND}, which can be used to provide error bounds for the computed solution in the space $\mathcal{H}^2_{\mathrm{BC1}}$.

The energy for the two different plates is illustrated in Figure \ref{fig:Case1Energy} (left), indicating a monomial-type decrease of total energy, $E^{*}(t^*)\propto {t^*}^{-2\nu}$. In fact, for $F^{*}\equiv 0$ and $y^{*}_{t^{*}}(x^{*},0)\equiv 0$, the energy is given by
\begin{equation}
\label{asymp_energy}
E^{*}(t^{*})=\underbrace{\frac{\sin^2(\pi\nu)\Gamma(\nu)^2}{2\pi^2}\frac{b^2}{a}\int_{-1}^1|y_{x^{*}x^{*}}^{*}(x^{*},0)|^2dx^{*}}_{=:e_1}\times t^{*-2\nu}+\mathcal{O}(t^{*-3\nu}),\quad \text{as } t^{*}\rightarrow\infty,
\end{equation}
the derivation of which may be found in Appendix \ref{app}. We have shown the leading asymptotics as dashed lines in Figure \ref{fig:Case1Energy} (left), which agree with our computed values. As expected, the energy of the weakly-damped, $\nu=0.32$, system does not decay as rapidly as the more strongly-damped, $\nu=0.64$, system. The differences between computed energy values and the leading term of \eqref{asymp_energy} are shown in Figure \ref{fig:Case1Energy} (right), along with the expected $\mathcal{O}(t^{*-3\nu})$ behaviour.

\begin{figure}[t!]
  \centering
  \begin{minipage}[b]{0.49\textwidth}
    \begin{overpic}[width=\textwidth]{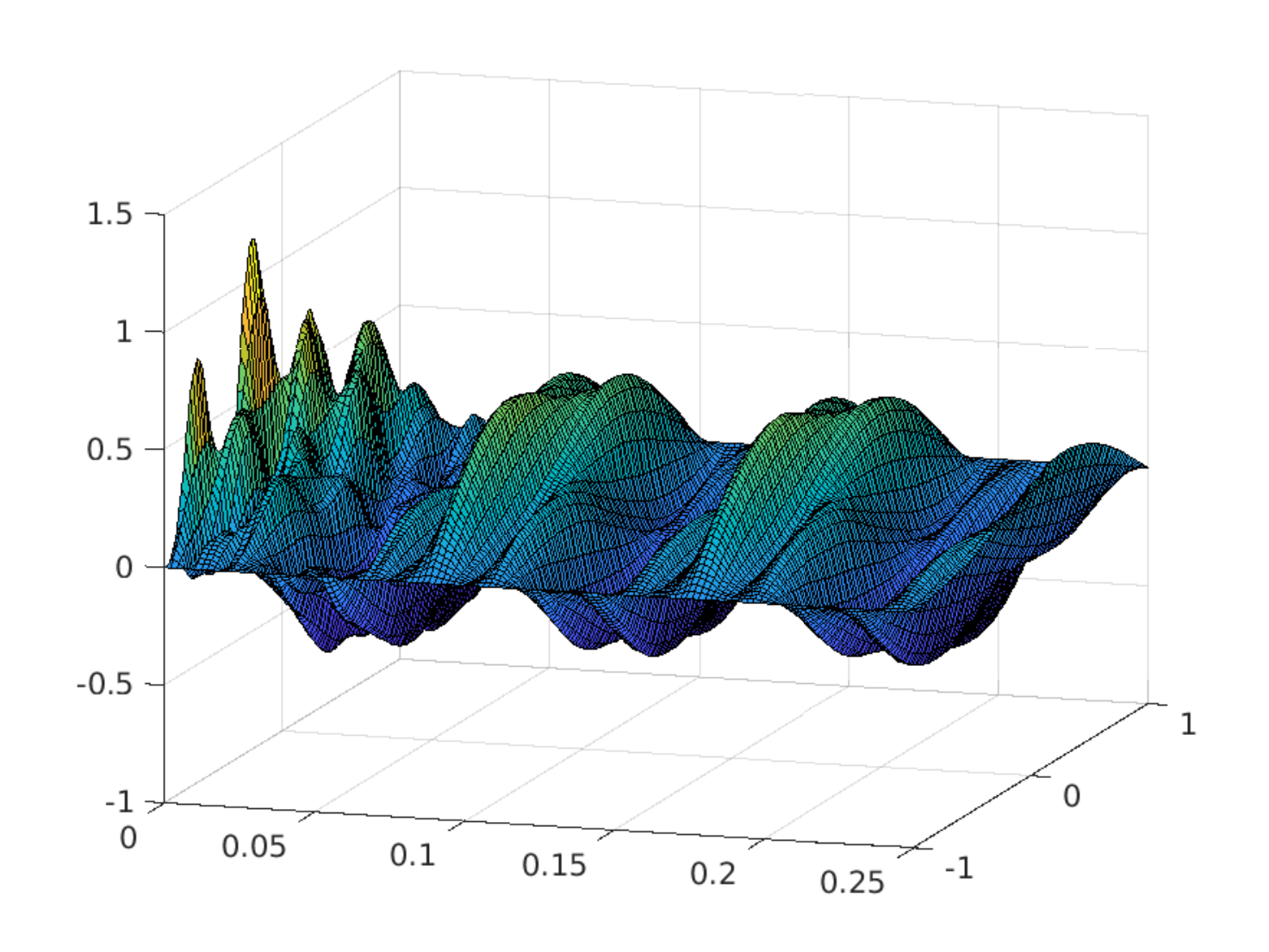}
		\put (45,0) {$\displaystyle t^*$}
		\put (87,9) {$\displaystyle x^*$}
		\put (45,70) {\rotatebox{0}{$\nu=0.32$}}
     \end{overpic}
  \end{minipage}
  \hfill
  \begin{minipage}[b]{0.49\textwidth}
    \begin{overpic}[width=\textwidth]{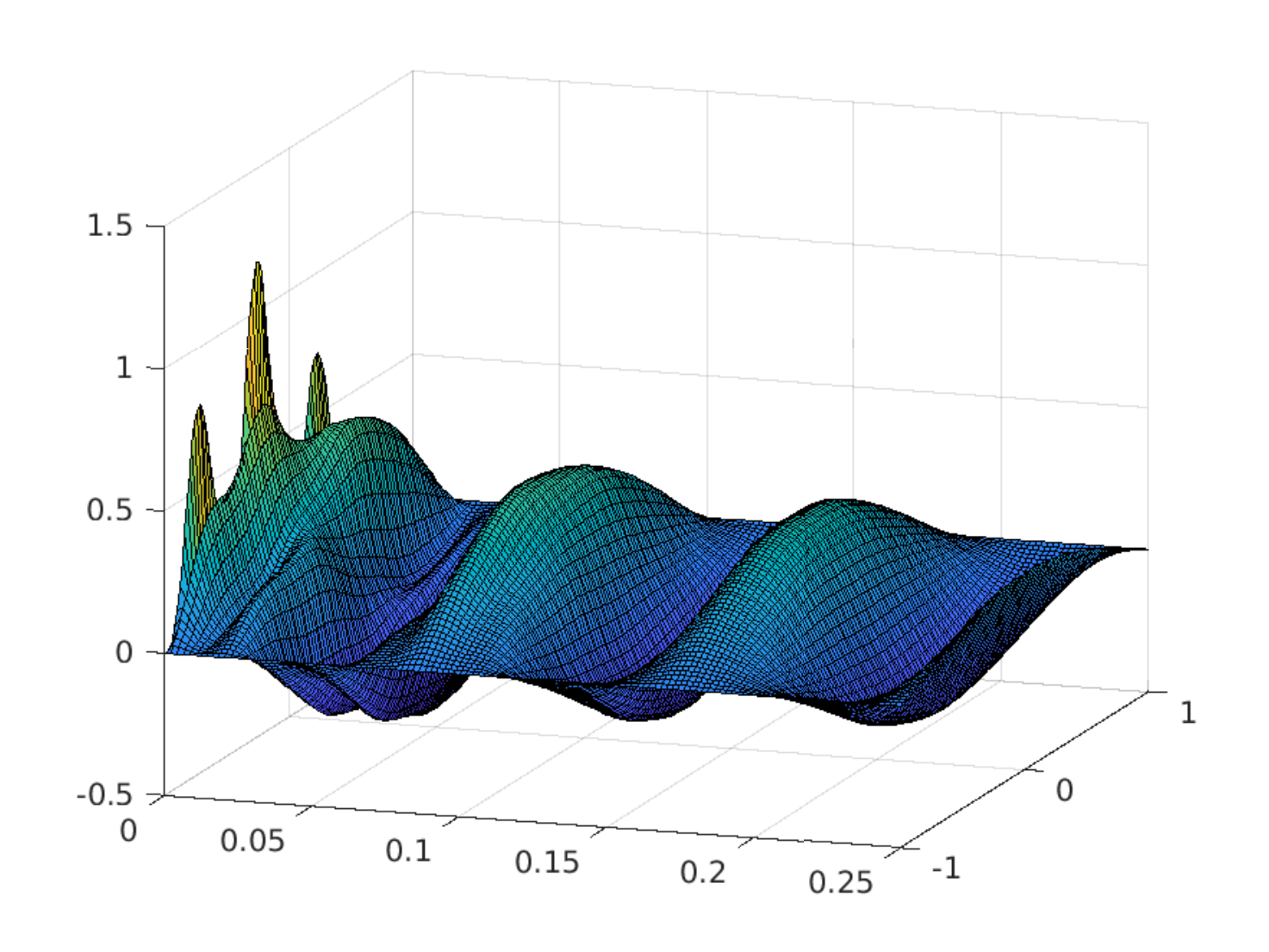}
		\put (45,0) {$\displaystyle t^*$}
		\put (87,9) {$\displaystyle x^*$}
		\put (45,70) {\rotatebox{0}{$\nu=0.64$}}
		\end{overpic}
  \end{minipage}
  \caption{Beam deformations $y^{*}(x^{*},t^{*})$ for non-zero initial conditions.}
\label{fig:Case1Initial}
\end{figure}

\begin{figure}[t!]
  \centering
	\vspace{5mm}
  \begin{minipage}[b]{0.49\textwidth}
    \begin{overpic}[width=1\textwidth]{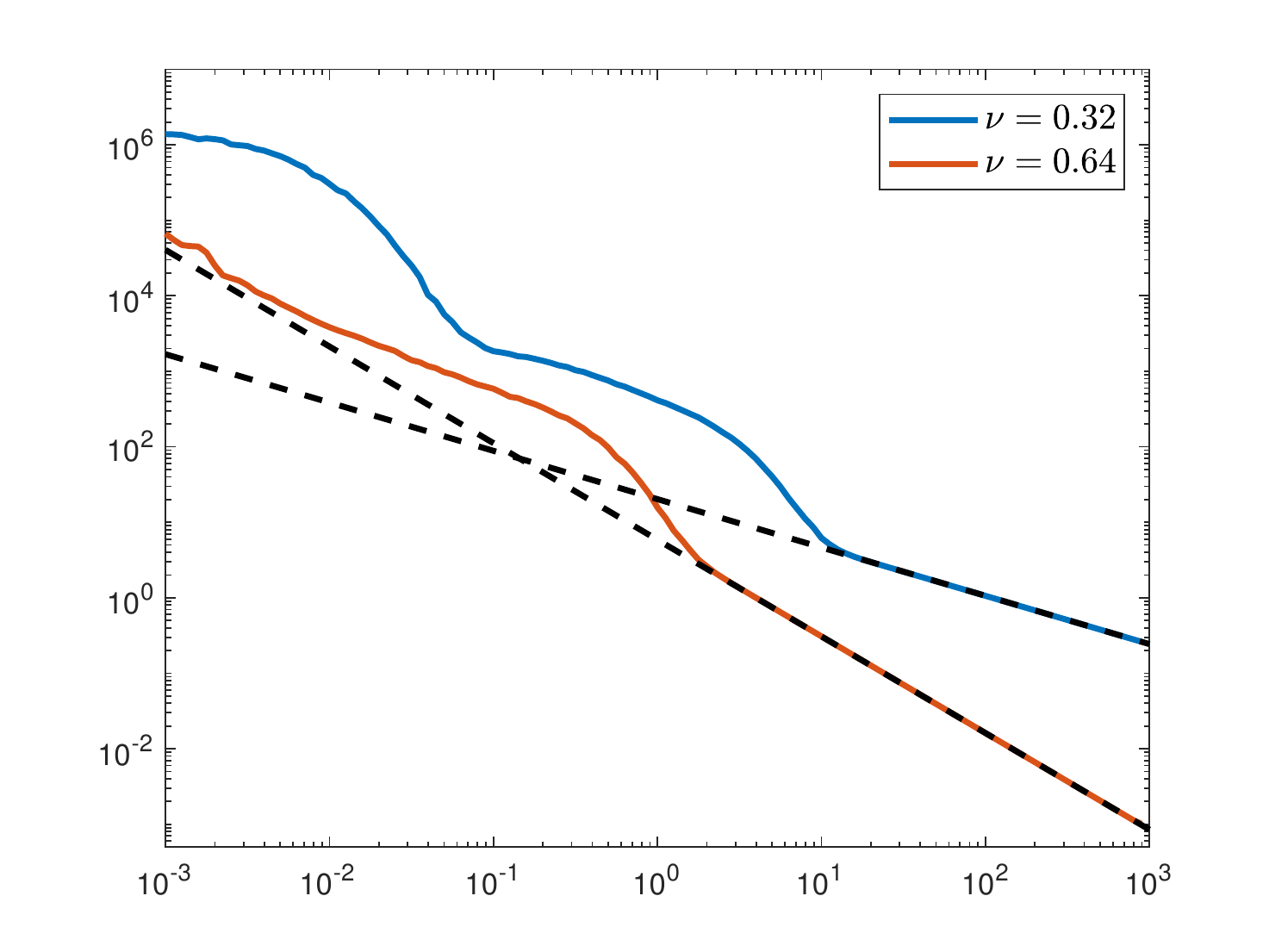}
		\put (50,-2) {$\displaystyle t^*$}
		\put (45,75) {\rotatebox{0}{$\displaystyle E^{*}(t^*)$}}
     \end{overpic}
  \end{minipage}
	\begin{minipage}[b]{0.49\textwidth}
    \begin{overpic}[width=1\textwidth]{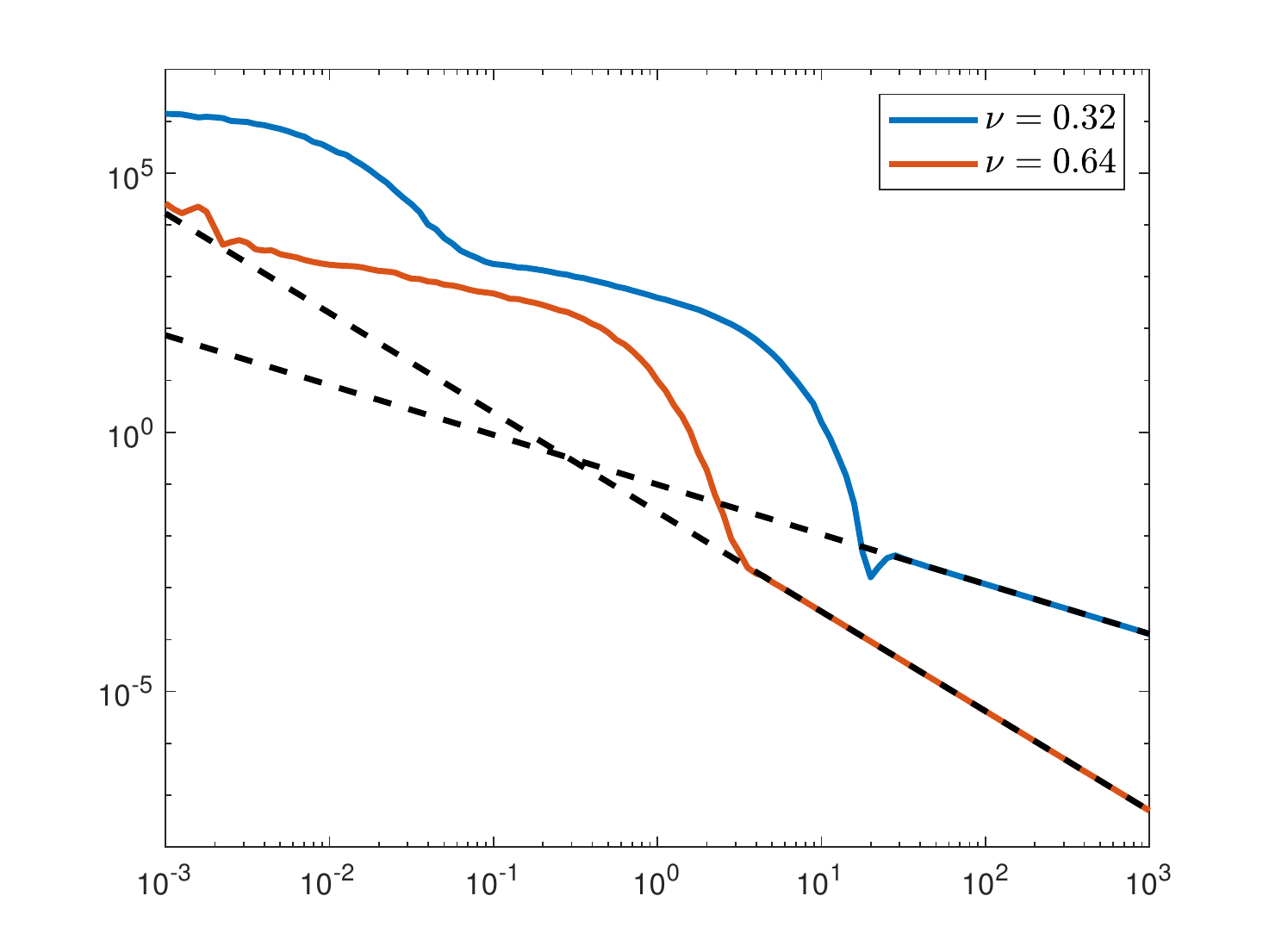}
		\put (50,-2) {$\displaystyle t^*$}
		\put (32,75) {\rotatebox{0}{$\displaystyle \left|E^{*}(t^*)-e_1\times {t^*}^{-2\nu}\right|$}}
     \end{overpic}
  \end{minipage}
  \caption{Left: Energy vs. time for two beams with identical initial deformations. The dashed lines show the leading asypmtotic term for the energy in \eqref{asymp_energy}. Right: The difference between the computed energy values and the leading asypmtotic term for the energy in \eqref{asymp_energy}. The dashed lines show the expected slopes of $\mathcal{O}({t^*}^{-3\nu})$.}
\label{fig:Case1Energy}
\end{figure}

\subsection{Example 2}

Our second physical example includes the effects of variable coefficients and is inspired by \cite{Example2}. The parameters are non-dimensionalised as before, giving
$$
\rho^* = 1,\quad E_{0}^*I^*= 1,\quad   F^*=(24-\pi^2(1+x^*)^{2}(1-x^*)^{2})\cos(\pi t^*),
$$
with both ends clamped, and initial conditions
$$
y^*(x^*,0)=(1+x^*)^{2}(1-x^*)^{2},\quad \pderiv{y^*}{t^*}(x^*,0)=0.
$$
When there is no damping, i.e., $E_{1}^*=0$, the analytic solution is given by
\begin{equation}
\label{an_no_damp}
y^*(x^*,t^*)=(1+x^*)^{2}(1-x^*)^{2}\cos(\pi t^*)
\end{equation}

We illustrate in Figure \ref{fig:Case2} the effect of a variable damping parameter given by
\begin{equation}
E_{1}^{*}I^{*}=1.01+\tanh(10x^*).
\end{equation}
This represents a beam that is very weakly damped for $-1<x^*<0$, but has stronger damping for $0<x^*<1$. The overall strength of the damping is also impacted by the fractional derivative, $\nu$, with larger $\nu$ indicating a more strongly damped beam. The energy at small times $t^*<1$ in all cases is reduced, however, at a slower rate for lower $\nu$ values, due to the initial deformation, as was the case for our previous example. The energy at larger times $t^*>1$ oscillates in time, as it is dominated by the forcing $\sim\cos(\pi t^*)$, and a more strongly damped beam (higher $\nu$) has overall lower energy and visibly weaker vibrations. Since, in this case, the beam's properties are asymmetric, despite an initial symmetric deformation and symmetric forcing, the beam's response is asymmetric, with greater oscillation in the weakly damped half, $-1<x^*<0$.

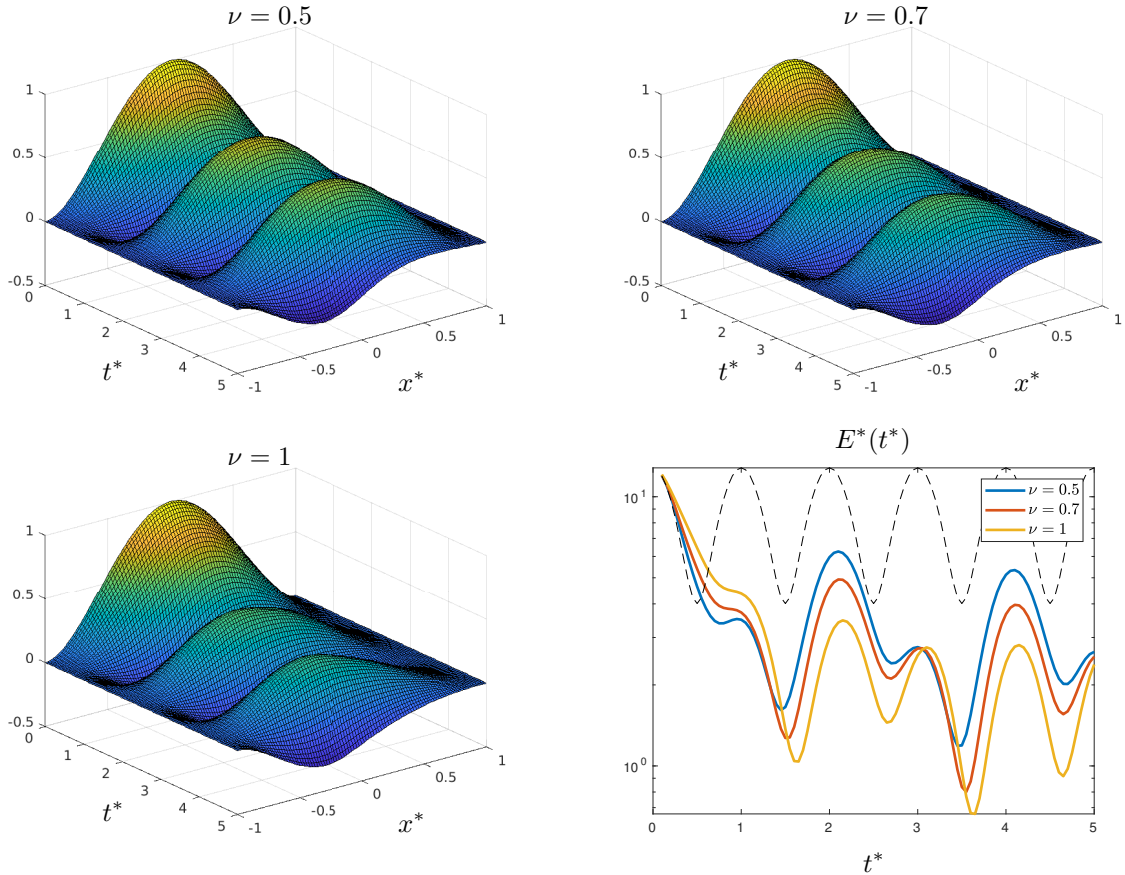
\begin{figure}[h!]
  \centering
  \begin{minipage}[b]{0.48\textwidth}
    \begin{overpic}[width=\textwidth]{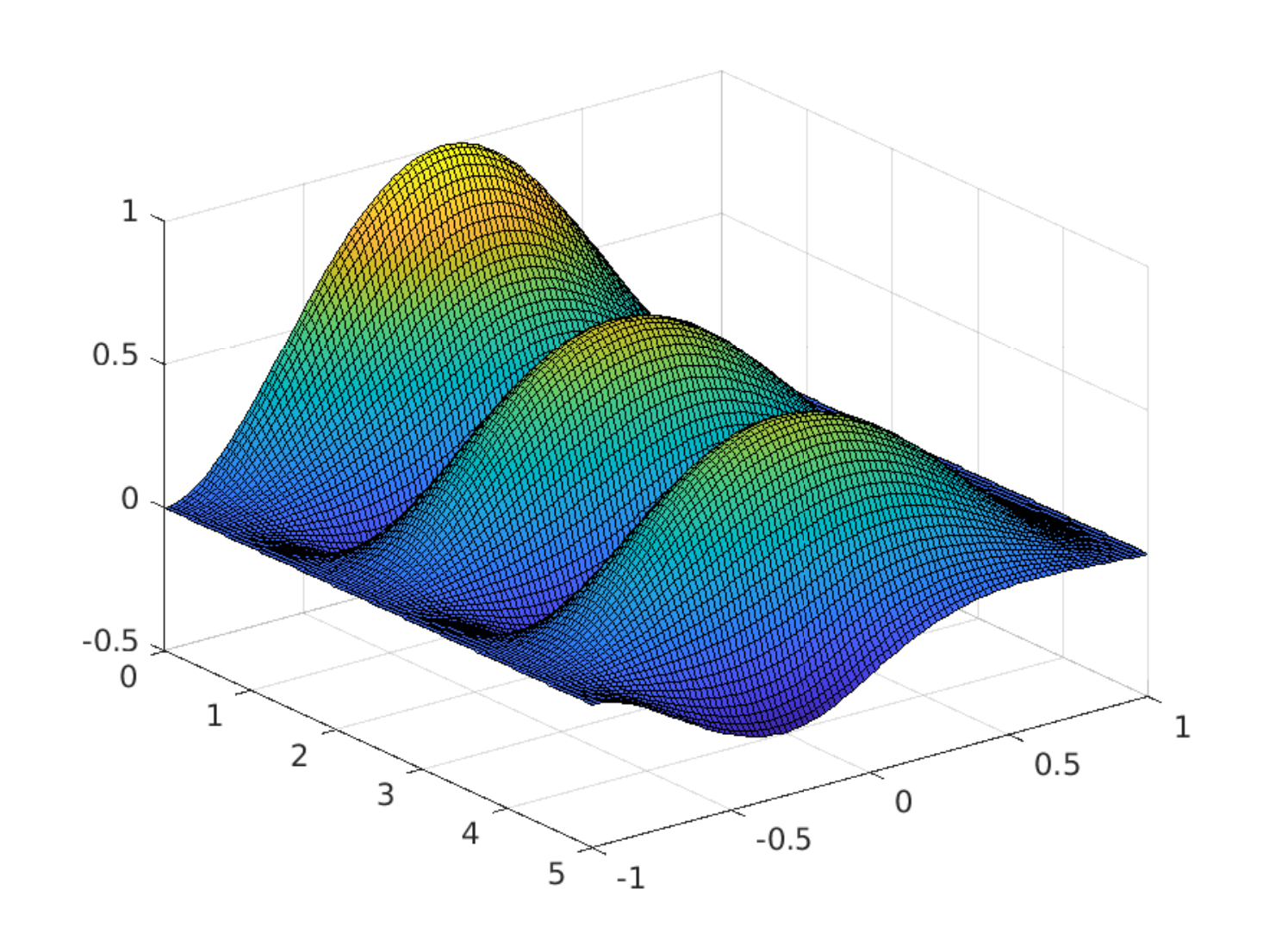}
		\put (23,7) {$\displaystyle t^*$}
		\put (75,5) {$\displaystyle x^*$}
		\put (45,70) {\rotatebox{0}{$\nu=0.5$}}
     \end{overpic}
  \end{minipage}
  \hfill
  \begin{minipage}[b]{0.48\textwidth}
    \begin{overpic}[width=\textwidth]{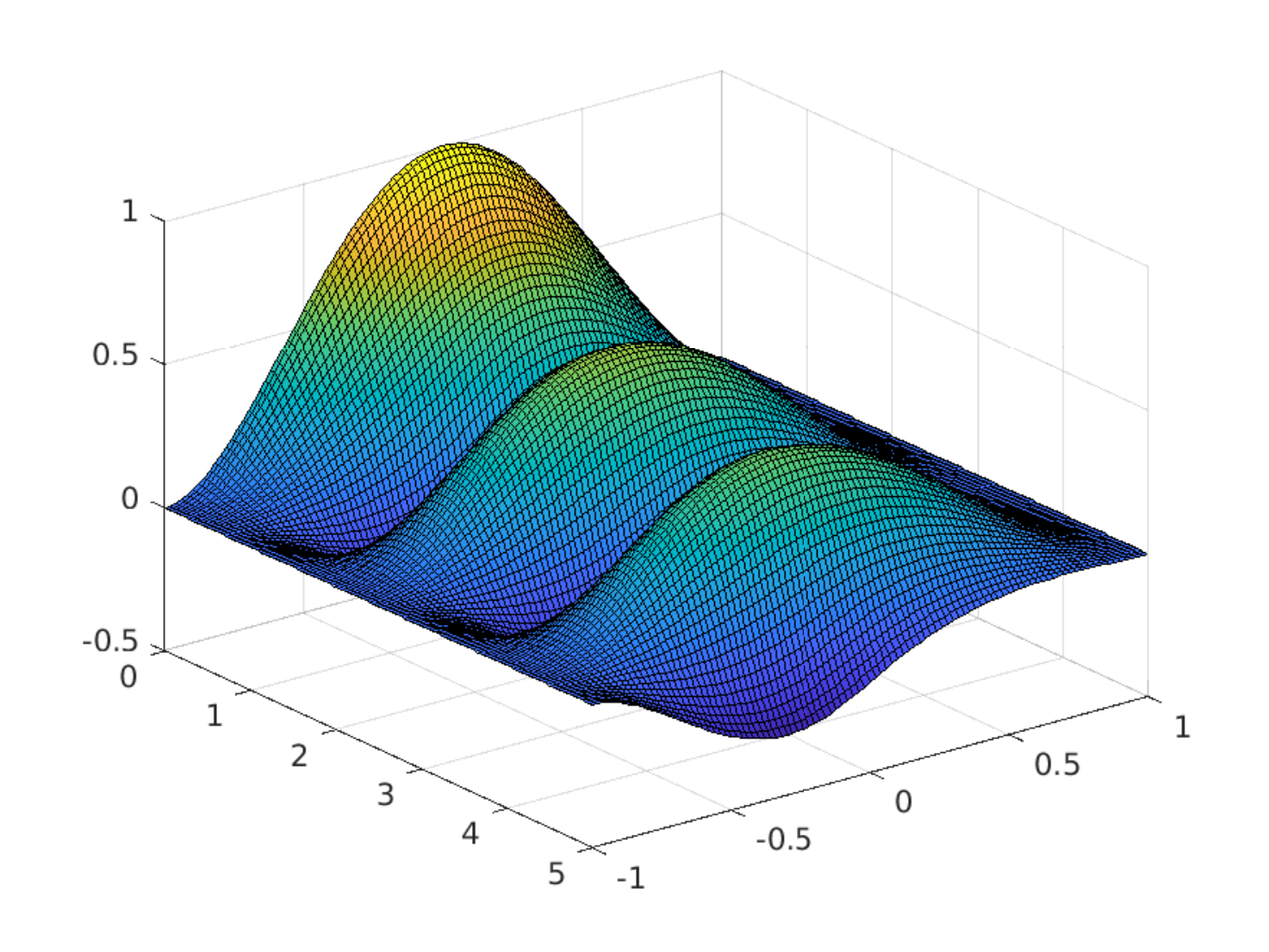}
		\put (23,7) {$\displaystyle t^*$}
		\put (75,5) {$\displaystyle x^*$}
		\put (45,70) {\rotatebox{0}{$\nu=0.7$}}
		\end{overpic}
  \end{minipage}
  \\\vspace{2mm}
    \begin{minipage}[b]{0.48\textwidth}
    \begin{overpic}[width=\textwidth]{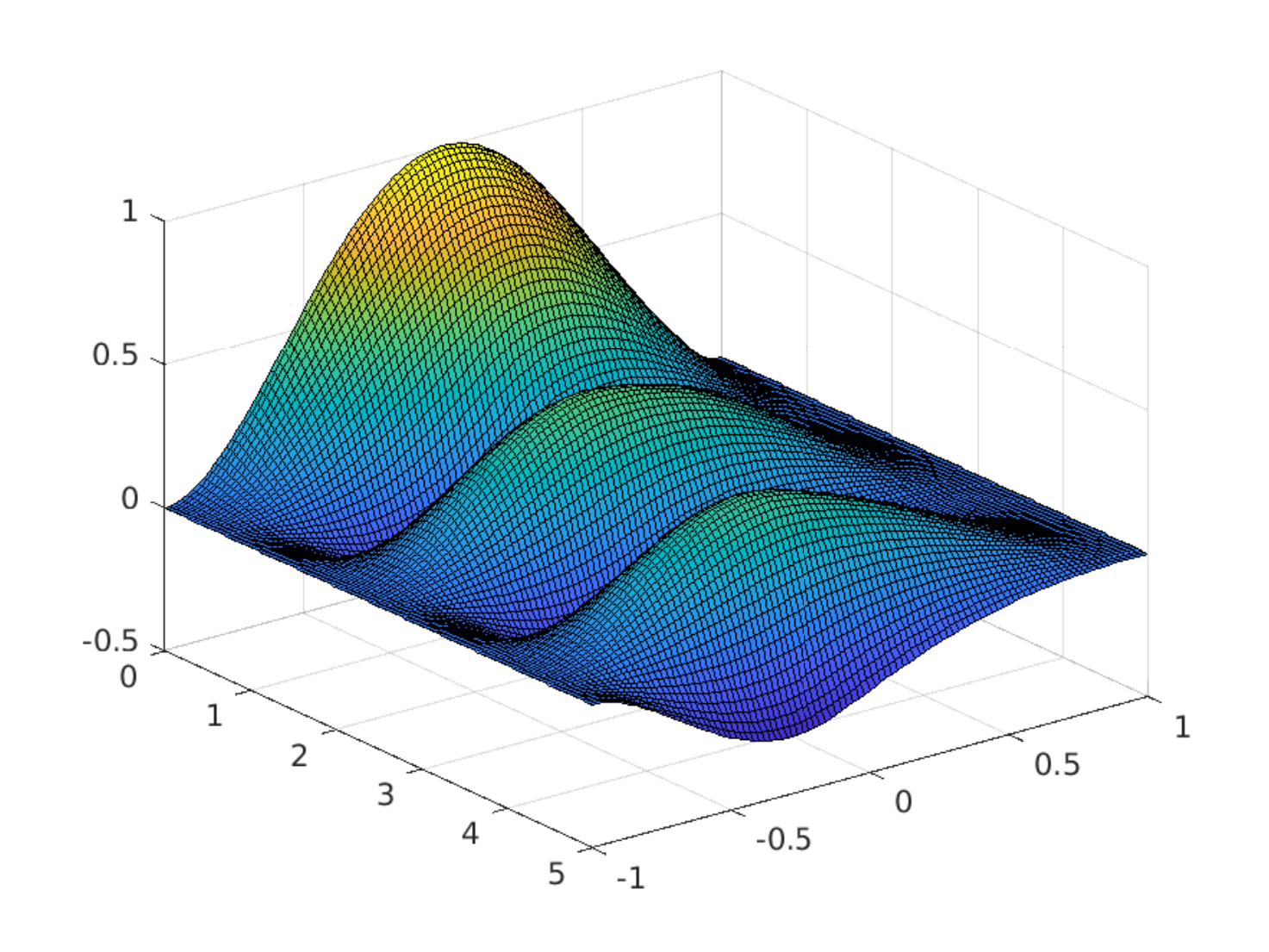}
		\put (23,7) {$\displaystyle t^*$}
		\put (75,5) {$\displaystyle x^*$}
		\put (45,70) {\rotatebox{0}{$\nu=1$}}
		\end{overpic}
  \end{minipage}
    \hfill
  \begin{minipage}[b]{0.48\textwidth}
    \begin{overpic}[width=\textwidth]{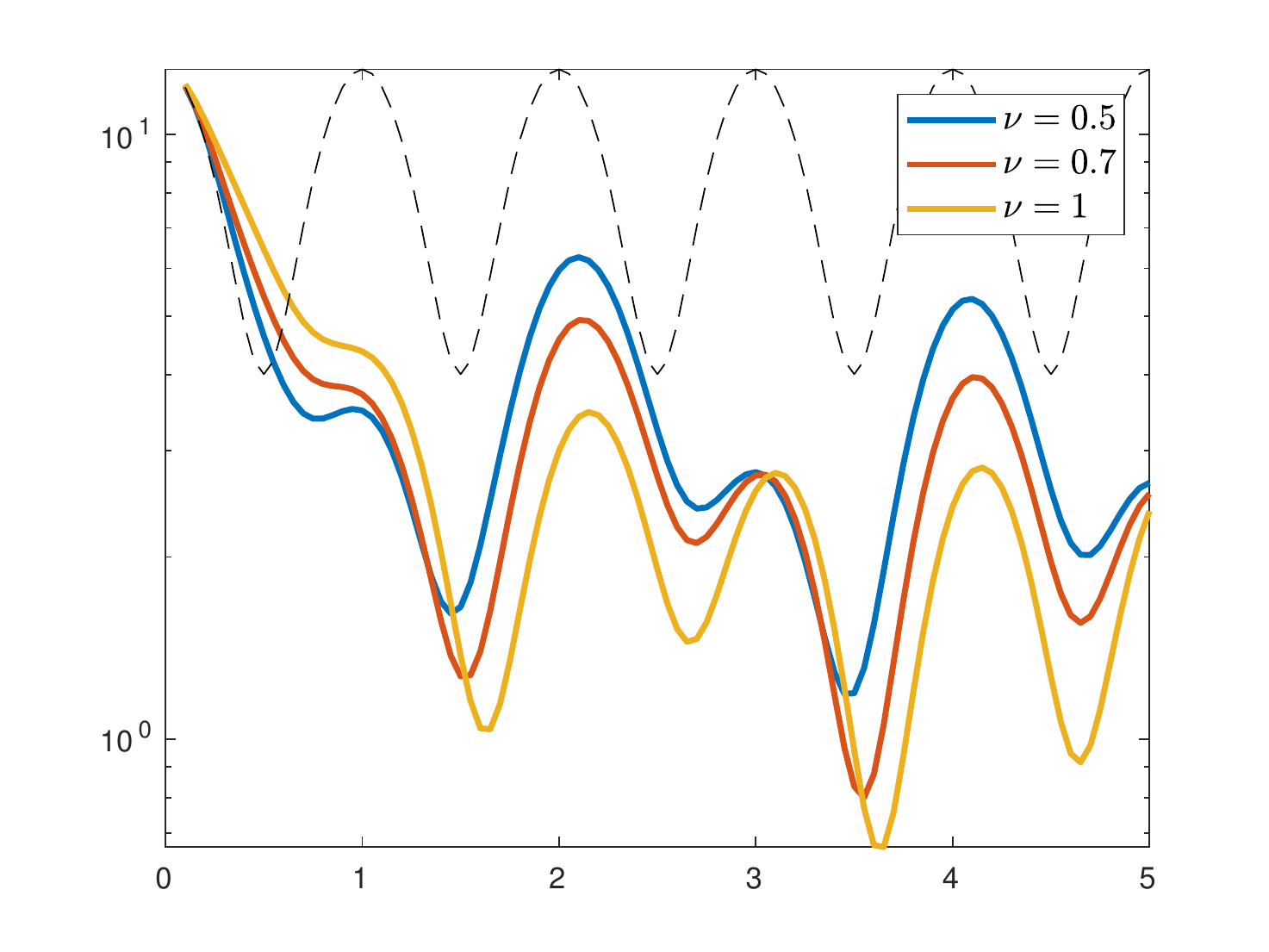}
		\put (50,-2) {$\displaystyle t^*$}
		\put (45,73) {\rotatebox{0}{$\displaystyle E^{*}(t^*)$}}
		\end{overpic}
  \end{minipage}
  \caption{Beam deformations $y^*(x^*,t^*)$ for Example 2. Bottom right: Energy vs. time for each case. The dashed line shows the energy for the solution \eqref{an_no_damp} in the case of no damping.}
\label{fig:Case2}
\end{figure}

\subsection{Example 3}
In this final example, we consider a time-fractional derivative of order greater that $1$ in a constituent relation. Such a situation arises for a composite structure obeying a fractional Burgers model \cite{B1}
\begin{equation}
\sigma + m\pderiv{^{\upsilon}\sigma}{t^{\upsilon}}+n\pderiv{^{2\upsilon}\sigma}{t^{2\upsilon}}= p\pderiv{^{\mu}\epsilon}{t^{\mu}}+q\pderiv{^{2\mu}\epsilon}{t^{2\mu}},
\end{equation}
where $\upsilon,\mu\in(0,1)$. Whilst infrequently used in the modelling of beams, the fractional Burgers model has received recent attention for describing the behaviour of viscoelastic fluids \cite{B2,B3}, and thus illustrating our method with a fractional parameter $1<\nu<2$ may prove useful in future applications. Our test case here is therefore unphysical (we shall not implement the full Burgers model), and selects identical parameters to those of Example 2, with the exception that now $\nu$ shall be greater than $1$. We plot the deformed surface for $\nu=1.2$ and $\nu=1.8$ in Figure \ref{fig:Case3}.

Since $\nu>1$ in both of these cases, the fractional derivative term in \eqref{eq:gov} acts to dampen the system more strongly than in Example 2 (where $\nu\leq 1$). However, in addition to damped wave-like behaviour, there is also dispersion-type behaviour, as seen in the fractional wave equation \cite{FDWave}. This dispersive behaviour is most clearly illustrated for  $\nu=1.8$, where the half of the plate with non-negligible $E_{1}^*$ values ($0<x^*<1$) shows strong asymmetry versus the undamped side ($-1<x^*<0$) which undergoes high-amplitude oscillations. We illustrate this particular feature in Figure \ref{fig:vibrations}, where we plot the deformation of the plate at three different times, $t^{*}=0.5, 1, 1.5,$ and $2$.

\begin{figure}[t!]
  \centering
  \begin{minipage}[b]{0.48\textwidth}
    \begin{overpic}[width=\textwidth]{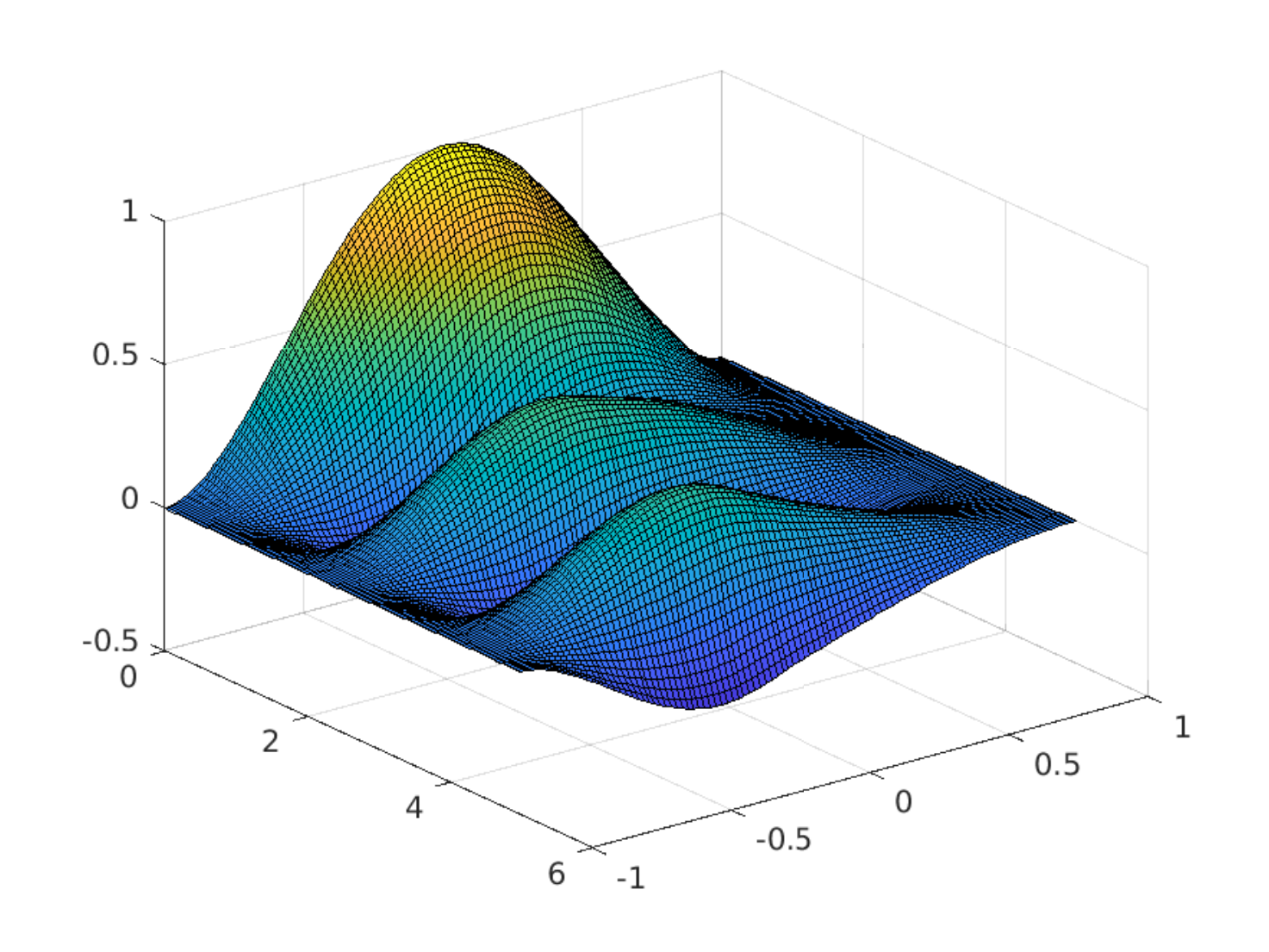}
		\put (23,7) {$\displaystyle t^*$}
		\put (75,5) {$\displaystyle x^*$}
		\put (45,70) {\rotatebox{0}{$\nu=1.2$}}
     \end{overpic}
  \end{minipage}
  \hfill
  \begin{minipage}[b]{0.48\textwidth}
    \begin{overpic}[width=\textwidth]{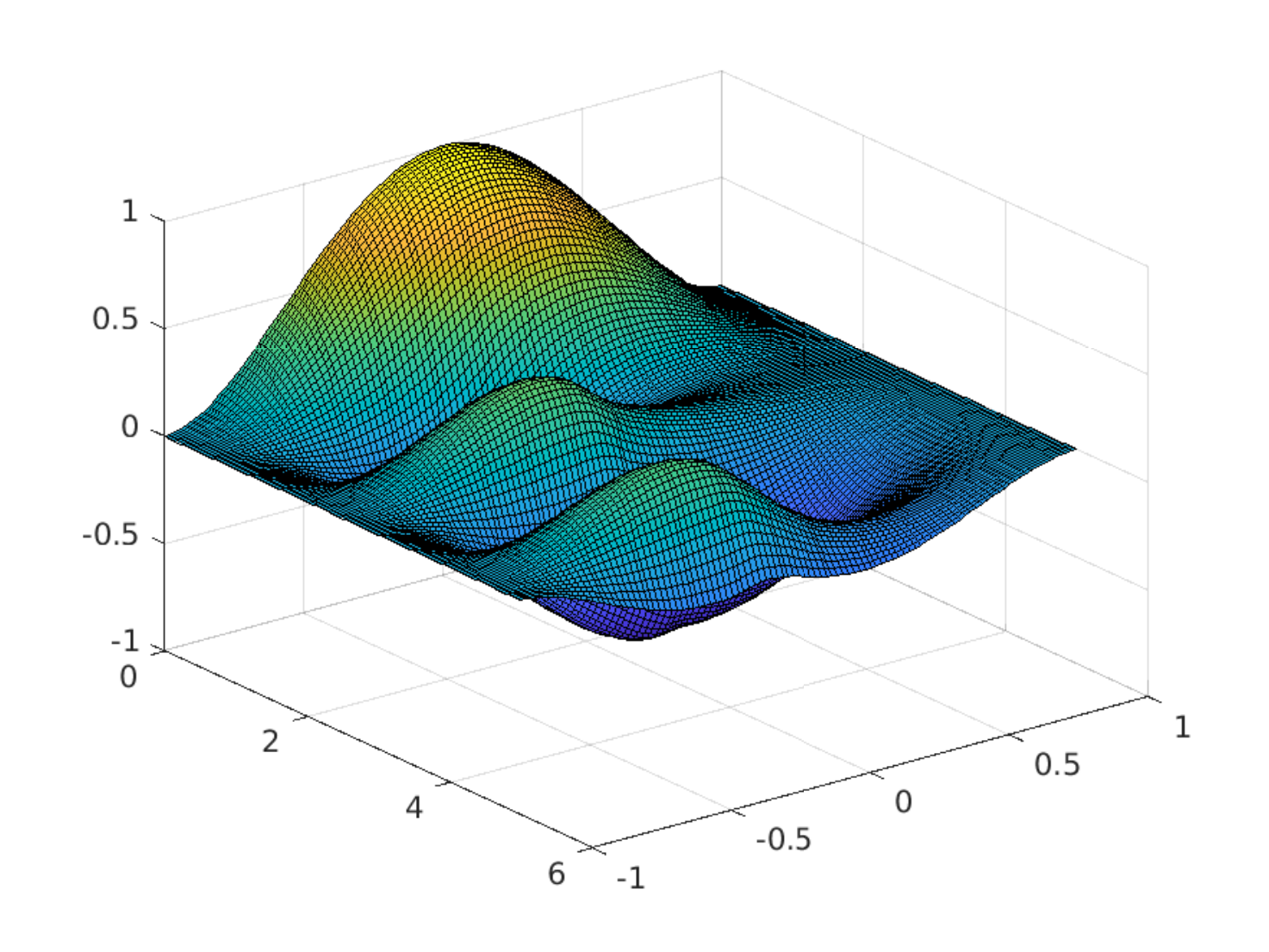}
		\put (23,7) {$\displaystyle t^*$}
		\put (75,5) {$\displaystyle x^*$}
		\put (45,70) {\rotatebox{0}{$\nu=1.8$}}
		\end{overpic}
  \end{minipage}
  \caption{Beam deformations $y^*(x^*,t^*)$ for Example 3.}
\label{fig:Case3}
\end{figure}

\begin{figure}[t!]
  \centering
	\begin{minipage}[b]{0.48\textwidth}
    \begin{overpic}[width=\textwidth]{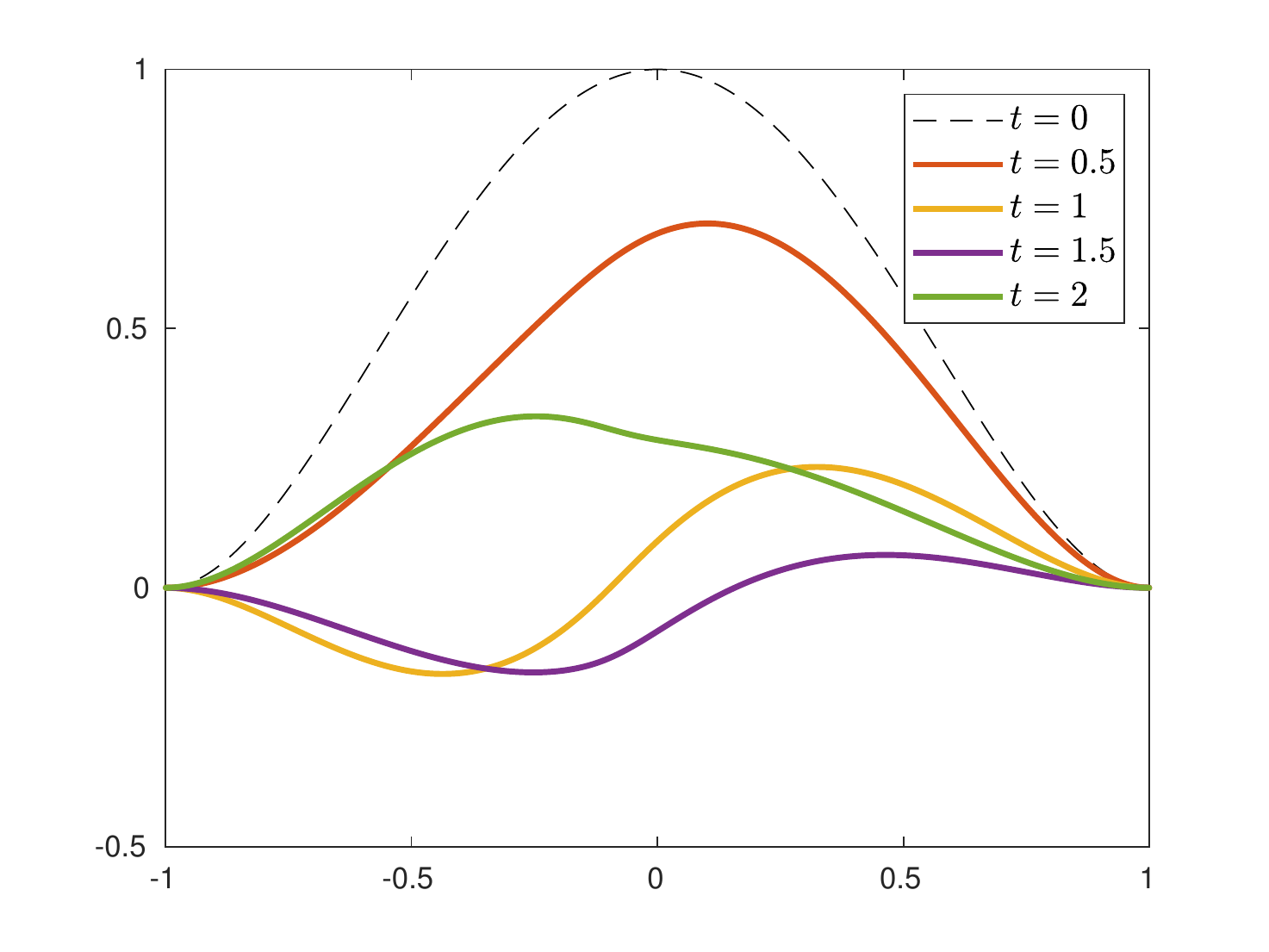}
		\put (50,0) {$\displaystyle x^*$}
		\put (45,73) {\rotatebox{0}{$\nu=1.2$}}
		\end{overpic}
  \end{minipage}
	\hfill
	\begin{minipage}[b]{0.48\textwidth}
    \begin{overpic}[width=\textwidth]{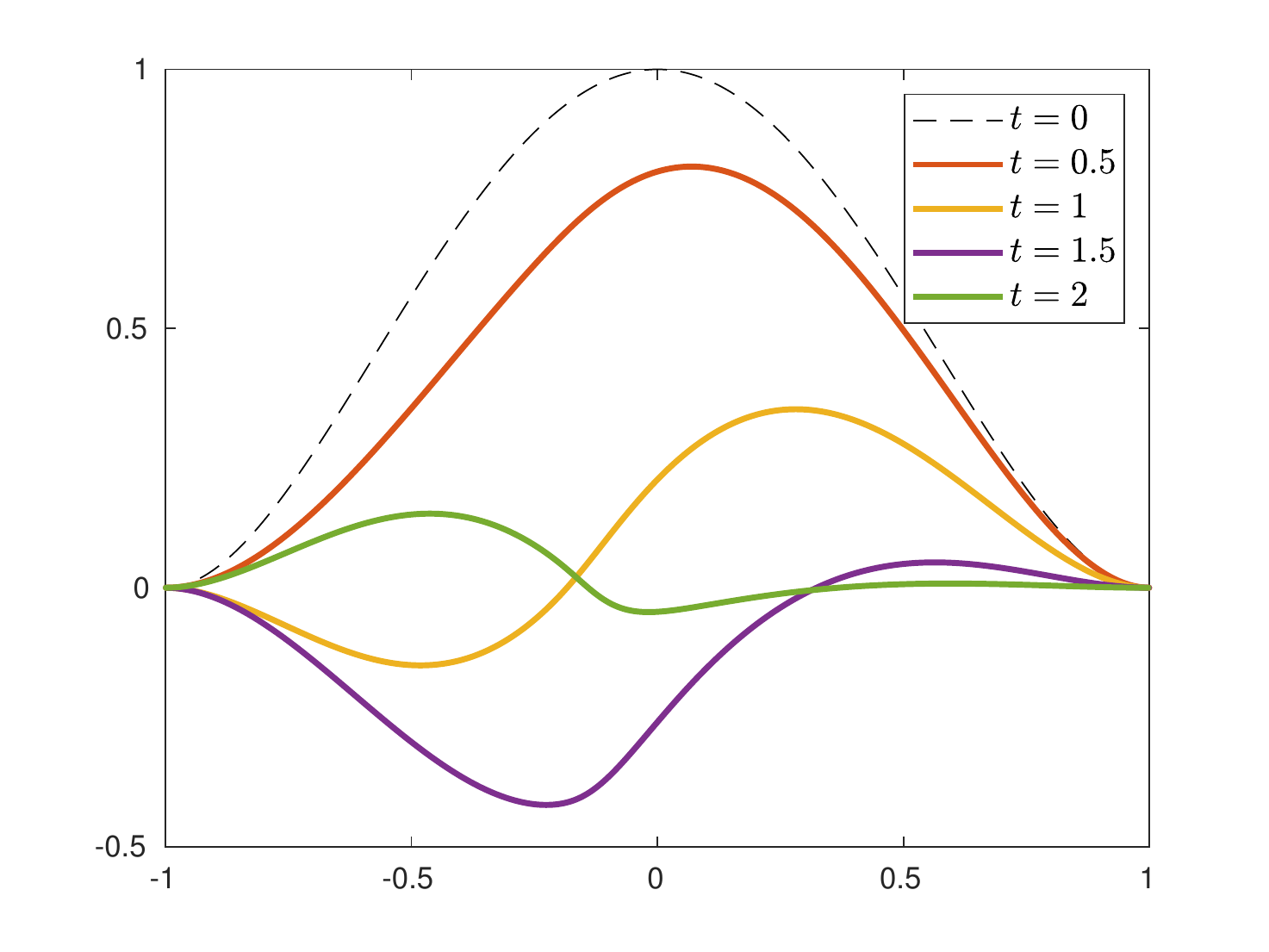}
		\put (50,0) {$\displaystyle x^*$}
		\put (45,73) {\rotatebox{0}{$\nu=1.8$}}
		\end{overpic}
  \end{minipage}
  \caption{Beam deformations $y^*(x^*,t^*)$ for Example 3 at times $t^{*}=0.5,1,1.5$ and $2$.}
\label{fig:vibrations}
\end{figure}

\section{Conclusion}
\label{sec:conc}

We have presented an approach that allows time-fractional PDEs to be solved efficiently and accurately. Some distinct advantages of this approach are: the resolvents can be computed in parallel and reused for different times, the mitigation of large memory consumption and poor accuracy typically associated with local (in time) methods, the method avoids singularities of the solution at $t=0$ (thus making analysis of the singularities or tuning unnecessary), explicit error control, stability, and rapid convergence (both in the quadrature rule and solution of the linear systems). A drawback of this approach is that information on the location of the generalised spectrum (corresponding to the singularities of the Laplace transform) is needed to choose the integration contour. However, since the generalised spectrum of a differential operator is typically easier to study than that of a discretisation, the analysis of the generalised spectrum is made much easier through an infinite-dimensional approach. We expect this approach to be feasible for many further families of time-fractional PDEs and hope that this new perspective will open up the door for further applications of contour methods.

We have illustrated our approach through application to the time-fractional equation governing the small-amplitude vibrations of a thin viscoelastic beam satisfying a fractional Kelvin--Voigt constituent relationship. This problem is ubiquitous and fundamental to beam-vibration studies, and there is a wealth of literature dedicating itself to obtaining fast, accurate solutions. However, this prior work has struggled to incorporate non-zero initial conditions (initial deformations of the beam) and control the error in predictions. Our approach has been shown to overcome these difficulties and produce results for the time-evolution of the beam deformations in cases of (spatially) constant and variable viscoelastic properties. Since viscoelastic constitutive relations are determined through comparison to experiment and therefore incur a measurable error, we foresee our method being particularly useful for ensuring that this error is not exceeded during subsequent vibration modelling. Furthermore, since the method permits any number of time-fractional terms, the door is opened for more complex constitutive relations to be used in the future modelling of viscoelastic materials.

We finish with some outlines for future work. We have focused our analysis on one spatial dimension. However, the results of \S \ref{sec:three} can be generalised to fractional beam equations in higher spatial dimensions. We expect that efficient solvers such as \cite{fortunato2020ultraspherical} will be useful in gaining an efficient and accurate ``infinite-dimensional'' method. Another possibility is the inclusion of spatial fractional derivatives. Here, we expect that spectral methods such as those developed in \cite{zayernouri2014fractional,zayernouri2013fractional,hale2018fast} can be effectively combined with an inverse Laplace transform in time. In some cases, it may be possible to extend these techniques to non-linear problems using exponential integrators \cite{lopez2010quadrature}. However, due to the Laplace transform approach, it may be difficult to deal with non-linear problems without time-stepping (e.g., via Newton's method\footnote{If one applies Newton's method, this induces coefficients in the linearised PDE that depend on space and time. Thus, dealing with time-variable coefficients may lead to this being possible in the future.}). Finally, time-dependent coefficients pose a challenging problem since taking the Laplace transform leads to a convolution. There are works that deal with this case \cite{lee2013laplace}, and it would be interesting to combine them with the techniques of this paper.

\section*{Acknowledgements}

MJC is supported by a Research Fellowship at Trinity College, Cambridge, and a Foundation Sciences Mathématiques de Paris postdoctoral fellowship at École normale supérieure. LJA is supported by EPSRC Early Career Fellowship EP/P015980/1.

{\small
\linespread{0.92}\selectfont{}
\bibliographystyle{unsrt}
\bibliography{semigroup}
\linespread{1}\selectfont{}}

\appendix
\section{Long-time energy}\label{app}
To approximate the long-time energy, $E(t)$, of the system governed by \eqref{NL_gen_form} with $0<\nu<1$, and $F=0$, we take the Laplace transform, $t\to z$, yielding
$$
z^{2}\hat y(x,z)-zy(x,0)-y_{t}(x,0)+\frac{1}{\tilde\rho(x)}\pderiv{^{2}}{x^{2}}\left[a(x)\pderiv{^{2}\hat y}{x^{2}}+b(x)\left[z^{\nu}\pderiv{^{2}\hat y}{x^{2}}-z^{\nu-1}y_{xx}(x,0)\right] \right]=0.
$$
When considering the inverse Laplace transform, the method of steepest descents tells us that the large time approximation is dominated by the small $z$ behaviour. Therefore, we approximate our equation to
$$
\frac{1}{\tilde\rho(x)}\pderiv{^{2}}{x^{2}}\left[a(x)\pderiv{^{2}\hat y}{x^{2}}\right]\sim \frac{z^{\nu-1}}{\tilde\rho(x)}\pderiv{^{2}}{x^{2}}\left[b(x) y_{xx}(x,0)\right]+\mathcal{O}(1),
$$
where the $\mathcal{O}(1)$ error on the right hand side exists only when $y_{t}(x,0)\neq0$. For simplicity, assume that $a$ and $b$ are constant. Integrating in $x$ gives,
$$
\hat y(x,z)\sim \frac{b}{a}z^{\nu-1}y(x,0)+\mathcal{O}(1)+\mathcal{O}(z^{\nu}),
$$
where we keep both the $\mathcal{O}(1)$ and $\mathcal{O}(z^{\nu})$ errors since the $\mathcal{O}(1)$ error disappears when $y_{t}(x,0)=0$. Inverting the Laplace transform yields
\begin{equation}
y(x,t)\sim \frac{b}{2\pi i a}y(x,0)\int_{\gamma}(z^{\nu-1}+\mathcal{O}(1,z^{\nu}))e^{zt}dz
\end{equation}
where $\gamma$ is the Bromwich contour, and our branch cut at $z=0$ extends along the negative real axis. Deforming $\gamma$ to the branch cut and integrating thus yields
\begin{equation}
y(x,t)\sim \frac{b}{\pi a}y(x,0)\sin(\pi\nu)\Gamma(\nu)t^{-\nu}+\mathcal{O}(t^{-2\nu}).
\end{equation}
Here we note that any $\mathcal{O}(1)$ term will integrate to zero around the branch cut, therefore the error in either case $y_{t}(x,0)=0$ or $y_{t}(x,0)\neq0$ is $\mathcal{O}(t^{-2\nu})$.
The dominant contribution to the energy arises from the spatial, $x$, derivative, therefore
\begin{equation}
E(t)\sim\frac{1}{2\pi^{2}}\sin^{2}(\pi \nu)[\Gamma(\nu)]^{2}\frac{b^{2}}{a}\int_{-1}^{1}|y_{xx}(x,0)|^{2}dx\, t^{-2\nu} + \mathcal{O}(t^{-3\nu}).
\end{equation}

Note that, for the case of $F\neq0$, or $\nu>1$, the dominant contribution as $z\to0$ must be reassessed. Any singularities of the Laplace transform of $F$ must be taken into account during the contour deformation. Further, in the case of $a$ and/or $b$ being variable in $x$, one would need to include their integrated contribution in solving the approximated governing equation for small $z$. Such a solution will yield a similar yet cosmetically unpleasant-looking, leading-order term for the energy.

\end{document}